\documentclass[a4paper;11pt]{amsart}

\usepackage{mathrsfs}\usepackage{graphicx}
\usepackage{tikz-cd}
\usepackage[arrow,matrix]{xy}
\usepackage{amsmath,amssymb,amscd,bbm,amsthm,mathrsfs}
\usepackage{mathtools} 
\usetikzlibrary{calc}

\usetikzlibrary{shapes.geometric}

\usepackage{graphicx}

\usepackage{tkz-euclide}
\usepackage{graphicx}
\usepackage[T1]{fontenc}
\usepackage{wesa}
\usetikzlibrary{patterns}

\usepackage{color,xcolor}
\usepackage{graphicx}
\usepackage{manfnt}
\newtheorem{thm}{Theorem}[section]
\newtheorem{lem}{Lemma}[section]
\newtheorem{definition}{Definition}[section]
\newtheorem{cor}{Corollary}[section]
\newtheorem{prop}{Proposition}[section]

\newtheorem{rem}{Remark}[section]

\theoremstyle{definition}

\newcommand{\supp}{\operatorname{supp}}

\usepackage[
    a4paper,      
    left=20mm,    
    right=20mm,   
    top=20mm,     
    bottom=20mm,  
    includehead,  
    includefoot   
]{geometry}

\usepackage{hyperref}
\hypersetup{
    colorlinks=true,    
    linkcolor=blue,     
    urlcolor=green,      
    citecolor=cyan,    
}

\usepackage{yhmath}
\usepackage{comment}
 
\begin{document}
\numberwithin{equation}{section}

\newtheorem{defn}[thm]{Definition}

 \title[Twisted multiparameter singular integrals]{Twisted Multiparameter singular integrals---real variable methods and applications, I}
\author {Zunwei Fu, Ji Li, Chong-Wei Liang, Wei Wang and Qingyan Wu}

\begin{abstract}
In this paper, we introduce a class of twisted multiparameter singular integrals on $\mathbb{R}^{2m}$, motivated by the Cauchy--Szeg\H{o} projections and the solving operators for $\bar{\partial}_b$ on a broad family of quadratic surfaces of higher codimension in $\mathbb{C}^n$. These surfaces are represented as suitable quotients of products of Heisenberg groups, a framework illustrated by Stein (Notices Amer. Math. Soc., 1998). While classical multiparameter product and flag theories are well-developed, Nagel, Ricci, and Stein observed a critical limitation: the class of product operators is not closed under passage to a quotient subgroup. To handle the geometric reduction that models these quotient structures, we take the first step in developing an adapted real-variable theory. We achieve this by introducing twisted tube systems and tube maximal functions, establishing a reproducing formula, Littlewood--Paley theory, a Journ\'e-type covering lemma, and atomic decompositions. As particular examples, we obtain twisted Fourier multipliers---which emerge as novel, direction-sensitive, and anisotropic phase-shift converters with potential applications in signal and image processing.
\end{abstract}


 \maketitle
 \section{Introduction and Statement of main results}\label{sec:1}
This is the first in a series of forthcoming papers on a framework of singular integral operators and real-variable methods associated with a twisted multiparameter structure motivated by the Cauchy--Szeg\H{o} projection on certain fundamental Siegel domains and the associated quotient subgroups (as indicated by Stein in \cite{St98} and studied by Nagel--Ricci--Stein \cite{NRS}). This work lays the foundation for a broader theory of twisted singular integral operators and paves the way for study of the Kohn–Laplacian on Shilov boundaries of certain fundamental Siegel domains.

\subsection{Background}
 One-parameter harmonic analysis, anchored by the classical Calder\'on--Zygmund theory, serves as a cornerstone of modern analysis. It provides the essential machinery for establishing $L^p$
  estimates for singular integral operators, playing a vital role in elliptic regularity theory for partial differential equations. In the context of complex analysis, this framework has been instrumental in characterizing the boundary behavior of holomorphic functions on strictly pseudoconvex domains, where the geometry is governed by a single parameter scaling structure.

As remarked by E.M. Stein in his seminal survey \cite{St98}, ``\textit{multiparameter analysis could well turn out to be of great interest in questions related to several complex variables.}''
The product theory on $\mathbb{R}^n$, one of the simplest dilation models among the multiparameter frameworks, began with Zygmund’s study of the strong maximal function, continued with Marcinkiewicz’s multiplier theorem, and has since branched into many directions. Among the achievements are an appropriate Littlewood--Paley theory and numerous properties of product Calder\'on--Zygmund operators. See, for example, the pioneering works of S.-Y. A. Chang, R. Fefferman, Gundy, Journ\'e,  Pipher and Stein (see \cite{CF80,CF85, FeffS, Feff-S,  GuS, J, KM, P,St98}). The prototypical example of a product Calder\'on--Zygmund operator is the multiple Hilbert transform
\begin{align*}
    Tf(x_1,\ldots,x_n)={\rm p.v.} \int_{\mathbb{R}^n}
\frac{f(x_1-y_1,\ldots,x_n-y_n)}{y_1\cdots y_n}\,dy_1\cdots dy_n.
\end{align*}
More generally, product kernels satisfy differential inequalities and cancellation conditions analogous to those of ${\rm p.v.}\,1/(x_1\cdots x_n)$. A key ingredient in this setting is Journ\'e's covering lemma \cite{J,P}, which provides a systematic replacement of general open sets by rectangles with controlled geometry.
As a fundamental example beyond $\mathbb{R}^n$, harmonic analysis on the Shilov boundaries of general tensor product Siegel domains is inherently multiparameter and plays a central role in understanding the boundary behavior of holomorphic functions on these domains \cite{NS04}.

Another family of excellent examples of multiparameter singular integrals  beyond the product theory arises from the $\bar{\partial}$-Neumann problem (in the model case corresponding to the Heisenberg group) in the work of Müller, Ricci and Stein \cite{MRS}, which is now the well-known multiparameter flag structures and has been intensively studied in recent decades \cite{CCLLO,HLLW,HLS,NRS,NRSW1,NRSW2}. One direct way (\cite{MRS,St2000}) of understanding 
these implicit flag structure is that the singular integral kernel $K(z,t)$ on $\mathbb{H}^n$ can be 
linked to a product kernel $K^{\#}(z,t,s)$ defined on $\mathbb{H}^n\times\mathbb{R}$ via the projection
\begin{align*}
    K(z,t)=\int_{\mathbb{R}}
K^{\#}(z,t-s,s)\,ds,\quad 
\forall(z,t)\in\mathbb{H}^n.
\end{align*}
Nagel--Ricci--Stein \cite{NRS} introduced a more general definition of flag singular kernels, 
and further developed the properties of a class of operators given by convolution with 
these flag kernels.  We will recall the full definition of their flag singular integrals on Euclidean space later in 
Section \ref{sec:2.1}. Some typical flag type convolution kernel on $\mathbb{R}^2$ is of the form
$${1\over x (x+iy)},\qquad (x,y)\in\mathbb{R}^2\ {\rm\ and\ \ } i^2=-1.$$

 Beyond product and flag structures, significant motivation arises from the study of Cauchy--Szeg\H{o} projections and solving operators for $\bar{\partial}_b$ on a wide class of quadratic surfaces of higher codimension in $\mathbb C^n$. These structures are naturally modeled by appropriate quotients of products of Heisenberg groups \cite[Section~$4$]{NRS}. 
 However, this geometric reduction poses analytic challenges. As Nagel--Ricci--Stein observed \cite{NRS}: ``\textit{while the theory of product kernels is satisfactory in many respects, one defect is that the class of product operators is not closed under passage to a quotient subgroup}''. To see this, consider the triple Hilbert transform on $\mathbb{R}^3$, and the operator that arises when passing to the quotient subgroup $\mathbb{R}^2$ via the mapping $(x, y, z)\mapsto (x-y, x-z)$. This process amounts to integrating the kernel $1/(xyz)$ over cosets of the line spanned by the vector $(1,1,1)$. The result is not a product kernel, which would be singular along two lines in $\mathbb{R}^2$, but rather a sum of kernels which are singular along a total of three lines.''

Indeed, the explicit kernel on the quotient subgroup $\mathbb{R}^2$ is 
an object that falls outside the framework of product kernels, and notably differs from the class of multiparameter flag kernels intensively studied in recent decades \cite{HLLW,HLS,NRS,NRSW1,NRSW2}.

Motivated by the quotient structure in \cite{NRS} and the survey paper by Stein \cite{St2000}, the purpose of the study of this series of papers is to build a real-variable theory for operators modeled by the quotient structure. Our model in this paper is the quotient of Euclidean spaces. As the first part of the series, we concentrate on the geometry of the underlying space induced by the quotient structure, the formulation of twisted singular integral operator, the Littlewood--Paley theory associated with the twisted multiparameter structure and the atomic decomposition.  Apart from the theoretical part, we also build a bridge between twisted Fourier multiplier, the multiplier associated with the quotient structure, and phase-shift converter (we refer to \cite{CFGW,FGLWY} for the standard one-parameter phase-shift converter).

\subsection{Statement of main results}

\subsubsection{Twisted singular integrals}
We propose a unified framework for twisted multiparameter singular integrals in the setting of homogeneous Lie groups (see for example Folland--Stein \cite{FS}). This algebraic setting provides the natural geometric model for these operators; results on general manifolds follow via the transference principle (see the discussion below).

Consider a homogeneous Lie group $\mathbb{G}$ with homogeneous dimension $Q$. The group is equipped with a family of automorphic dilations $\{\delta_r\}_{r>0}$ and a homogeneous quasi-norm $\|\cdot\|$. We assume the local geometry is defined by a basis of left-invariant vector fields $Z$, homogeneous with respect to these dilations.

To formulate the cancellation conditions, we require a notion of test functions adapted to this geometry.
\begin{definition}[Normalized bump functions]~\\
A function $\phi \in C^\infty(\mathbb{G})$ is called a \textit{normalized bump function} if:
\begin{enumerate}
    \item It is supported in the unit ball defined by the quasi-norm: $\operatorname{supp}(\phi) \subset \{x \in \mathbb{G} : \|x\| \leq 1\}$.
    \item Its derivatives are uniformly bounded: for sufficiently many multi-indices $\alpha$, $\|X^\alpha \phi\|_{L^\infty(\mathbb{G})} \leq 1$.
\end{enumerate}
\end{definition}
\noindent While the precise number of required derivatives depends on the specific order of the distribution, the definition is essentially independent of this choice for the purpose of describing the cancellation structure.

\begin{definition}[Twisted singular integrals on homogeneous Lie groups]\label{def:twisted_group}~\\
Let $\mathbb{G}$ be a homogeneous Lie group isomorphic to $\mathbb{C}^{n_1+n_2+n_3}\times\mathbb{R}^{2}$ equipped with a family of automorphic dilations $\{\delta_r\}_{r>0}$ and a homogeneous dimension $Q$. Let the singular set $\Sigma$ be the union of closed normal subgroups $S_1$ and $S_2$. We assume that each $S_i$ is a {\it homogeneous subgroup} (invariant under the dilations) with homogeneous dimension $Q_i$.

For each $i$, let $d((\mathbf{z},\mathbf{t}),S_i)$ denote the  distance from $(\mathbf{z},\mathbf{t})$ to the subgroup $S_i$. Assume that the (complex) vector field $Z$ is the union of the (complex) vector fields $Z^{(1)}$, $Z^{(2)}$ and $Z^{(3)}$, in which $Z^{(3)}$ is the field that acts on both $d((\mathbf{z},\mathbf{t}),S_1)$ and $d((\mathbf{z},\mathbf{t}),S_2)$. A datum $(K,\Sigma,Z)$ is call a {\it $(2,3)$-twisted datum} on $\mathbb{G}$ if $K$ is a distribution on $\mathbb{G}$ (defined away from $\Sigma$) such that it satisfies the following size and regularity condition:

For every left-invariant differential operator $\mathbf X^\alpha$ on $\mathbb{G}$, there exists a constant $C_\alpha$ such that for all $(\mathbf{z},\mathbf{t}) \in \mathbb{G} \setminus \Sigma$:
\begin{align}\label{eq:group_size}
    |\mathbf X^\alpha K((\mathbf{z},\mathbf{t}))|
   &=\left| (X^{(1)})^{\alpha_1}(X^{2})^{\alpha_2}(X^{(3)})^{\alpha_3}K((\mathbf{z},\mathbf{t}))\right|\notag\\ &\leq C_\alpha    \sum^{n_3}_{k=0} \sum^{\alpha_3}_{\gamma=0}\Bigg[\frac{1}{d((\mathbf{z}, \mathbf{t}),\,S_1)^{Q_2+(|\alpha_1|+|\alpha_3|)+(\gamma-|\alpha_3|)+2k }} \\
   &
   \hskip3cm\times\frac{1}{d((\mathbf{z}, \mathbf{t}),\,S_2)^{Q_1+({|\alpha_2|+|\alpha_3|})-\gamma+(2n_3-2k)}}\Bigg],\notag
\end{align}
where $X^{(j)}$ is the differential on the field $Z^{(j)}$ for each $1\leq j\leq 3$.

If (\ref{eq:group_size}) hold, then we call $K$ a {\it twisted singular integral kernel} on $\mathbb{G}$.
\end{definition}

\begin{rem}
    Definition \ref{def:twisted_group} can also be extended to the case that $\mathbb{G}$ is isomorphic to $\mathbb{C}^{n_1+\cdots+n_m}\times\mathbb{R}^{m-1}$ and to the case that $(K,\Sigma,Z)$ is a $(\boldsymbol{s},V)$-twisted datum. For simplicity, we only formulate the cleanest case. See Section \ref{sec:2} for the example of the $(2,3)$-datum and its structure.
\end{rem}

The structural complexity of the size condition \eqref{eq:group_size}—specifically the summation over pairs of subgroups—is not arbitrary; it arises naturally as the projection of a simpler structure. We posit that the ambient group $\mathbb{G}$ can be realized as a quotient $\mathbb{G} \cong \tilde{\mathbb{G}} / \mathcal{H}$ of a higher-dimensional {\it lifting group} $\tilde{\mathbb{G}}$, which is equipped with a standard multiparameter product structure.
The twisted kernel $K$ is the push-forward (fiber integral) of a product kernel $\tilde{K}$ defined on $\tilde{\mathbb{G}}$:
\[
K(x) = \int_{\mathcal{H}} \tilde{K}(\tilde{x} \cdot h) \, dh, \quad \text{where } \pi(\tilde{x}) = x \text{ and } \pi: \tilde{\mathbb{G}} \to \mathbb{G} \text{ is the projection.}
\]
Geometrically, the singular sets $S_i$ in $\mathbb{G}$ are the images of the coordinate singular varieties in $\tilde{\mathbb{G}}$. The integration along the fiber $\mathcal{H}$ couples the independent singularities of $\tilde{K}$, transforming the pure product bounds on the lifting space into the additive stratified bounds seen in \eqref{eq:group_size}.
Since the lifted operator $\tilde{T}$ associated with $\tilde{K}$ falls within the scope of established product multiparameter theory, it is bounded on $L^p(\tilde{\mathbb{G}})$. The Transference Principle (Coifman and Weiss \cite{CW}) then implies that the projected operator $Tf = f * K$ is bounded on $L^p(\mathbb{G})$ for $1 < p < \infty$.

In the pure Euclidean case, Definition \ref{def:twisted_group} will degenerate to the following:
\begin{definition}[Twisted singular integrals on Euclidean spaces]\label{def:twisted_Euclidean}~\\ 
Let $\mathbb{R}^M$ be the homogeneous Lie group and let $\boldsymbol{s}\geq2$ be an integer. For each $1\leq i\leq \boldsymbol{s}$, let $\mathbb{L}_i\subset\mathbb{R}^M$ be an affine plane containing $\mathbf{0}$ and $Q_i$ be the (homogeneous) dimension of $\mathbb{L}_i$. Define the singular set $\Sigma$ to be the union of affine planes $\mathbb{L}_i$'s, that is
\[
\Sigma = \bigcup^{\boldsymbol{s}}_{i=1}\mathbb{L}_i \subset \mathbb{R}^M.
\]

For each $1\leq i\leq \boldsymbol{s}$, let $d(\mathbf{x},\mathbb{L}_i)$ denotes the  distance from $\mathbf x$ to the affine plane $\mathbb{L}_i$. Let $\deg_{\mathbb{L}^\perp_i}(\alpha)$ represent the homogeneity order of the differential operator relative to the induced dilations on the quotient space $\mathbb{R}^M/\mathbb{L}_i$.  A datum $(K,\Sigma,Z)$ is call a {\it $(\boldsymbol{s},2)$-twisted datum} on $\mathbb{R}^M$ if $Z$ is a vector field and $K$ is a distribution on $\mathbb{R}^M$ (defined away from $\Sigma$) such that it satisfies the following size and regularity condition:

For every left-invariant differential operator $\mathbf X^\alpha$ on $\mathbb{R}^M$, there exists a constant $C_\alpha$ such that for all $x \in \mathbb{R}^M \setminus \Sigma$:
\begin{align}\label{eq:Euclidean_size}
    |\mathbf X^\alpha K(\mathbf x)| &\leq C_\alpha  \sum_{\substack{1\leq i_1,i_2\leq\boldsymbol{s} \\i_1\neq i_2}}\Bigg[ \frac{1}{{d(\mathbf{x},\mathbb{L}_{i_1})^{Q_{i_2}+\deg_{\mathbb{L}^\perp_{i_1}}(\alpha)}}}\cdot\frac{1}{{d(\mathbf{x},\mathbb{L}_{i_2})^{Q_{i_1} +\deg_{\mathbb{L}^\perp_{i_2}}(\alpha)}}}\\
&\hskip3cm+\sum^{\deg_{\mathbb{L}^\perp_{i_2}}(\alpha)}_{\gamma=0}\frac{1}{{d(\mathbf{x},\mathbb{L}_{i_1})^{Q_{i_2} +\deg_{\mathbb{L}^\perp_{i_1}}(\alpha)-\gamma}}}\cdot\frac{1}{{d(\mathbf{x},\mathbb{L}_{i_2})^{Q_{i_1} +\deg_{\mathbb{L}^\perp_{i_2}}(\alpha)+\gamma-\deg_{\mathbb{L}^\perp_{i_2}}(\alpha)}}}\Bigg]\notag.
\end{align}

If (\ref{eq:Euclidean_size}) hold, then we call $K$ a {\it twisted singular integral kernel} on $\mathbb{R}^M$.
\end{definition}
\begin{rem}
    Definition \ref{def:twisted_Euclidean} can also be extended to the case that $(K,\Sigma,Z)$ is a $(\boldsymbol{s},V)$-twisted datum for $V\geq3$. For a better interpretation of the example, we only formulate the $(\boldsymbol{s},2)$ case. Our example for the {\it twisted singular integral kernel} on $\mathbb{R}^M$ is the $(3,2)$-twisted datum. (see Section \ref{sec:2} for the discussion).
\end{rem}

In this paper, we will focus on 
the twisted multiparameter singular integral operator $T$ on $\mathbb{R}^m\times \mathbb{R}^m$, whose kernel $K(x,y)$ is related to the singularity on $\Sigma = \{x=0\} \cup \{y=0\} \cup \{x=y\}$ and satisfies differential size estimates dominated by the sum of product-type singularities associated with the pairwise intersections of these hyperplanes. Specifically, for all multi-indices $\alpha, \beta$, the kernel satisfies:
\begin{align}\label{eq:twisted_size_intro}
|\partial_x^\alpha \partial_y^\beta K(x,y)| &\lesssim \frac{1}{|x|^{m+|\alpha|} |y|^{m+|\beta|}} +\sum^\alpha_{\gamma_1=0} \frac{1}{|x-y|^{m+|\alpha|+|\beta|-\gamma_1} \cdot |x|^{m+\gamma_1}}\\ &\qquad+\sum^\beta_{\gamma_2=0} \frac{1}{|x-y|^{m+|\alpha|+|\beta|-\gamma_2} \cdot |y|^{m+\gamma_2}}\notag.
\end{align}
which captures the standard product behavior near the origin and the twisted behavior near the diagonal.

We will initiate a reproducing formula, a Littlewood--Paley theory for square functions and area functions, a family of bump functions, and a Journ\'e-type covering lemma associated to the structure \eqref{eq:twisted_size_intro}.
This includes: a thorough investigation of the dyadic tube (including oblique rectangles) system on $\mathbb{R}^m\times \mathbb{R}^m$; a construction of twisted Schwartz functions that induce the Calder\'on reproducing formula and the Lusin area function; a Journ\'e-type covering lemma that characterizes open sets and the associated maximal dyadic tubes; 
a construction of bump functions with cancellation condition associated with the twisted multiparameter structure, and the twisted Hardy space defined via the Lusin area function together with its equivalent characterization by decomposition into bump functions.

\subsubsection{Quotient structures and maximal function}
Let the lifting space be the product
$\mathbb{R}^m\times \mathbb{R}^m\times\mathbb{R}^m$ of three Euclidean spaces. To pass  objects on the lifting  space  to  ones on the $\mathbb{R}^m\times \mathbb{R}^m$, we use
the   projection $ \pi: \mathbb{R}^m\times\mathbb{ R} ^m \times\mathbb{R}^m\rightarrow\mathbb{ R} ^m \times\mathbb{R}^m$ defined by
\begin{equation}\label{eq:pi}\begin{split}
  \pi( x_1, x_2, x_3) = (x_1+ x_3, x_2+ x_3 ).
\end{split} \end{equation} 
   The fiber of the  projection $\pi$ over the point $(  x_1, x_2)\in \mathbb{R}^{2 m}$  is the space
 \begin{equation}\label{eq:fiber}
    \pi^{-1}( x_1, x_2)=\left\{( x_1-u, x_2- u, u):\, u\in \mathbb{R}^m\right\}.
 \end{equation}

For an   $L^1 $-function   $F$ on $ \mathbb{R}^{3 m}$, we define the {\it push-forward function} $\pi_*F$ on $ \mathbb{R}^{2 m}$
simply to be the integral of $F$ along the fiber as
 \begin{equation}\label{eq:transfer-F}
  \left ( \pi_*F\right)( x_1, x_2):=\int_{  \mathbb{R}^{  m} }F( x_1-u, x_2- u, u)du.
 \end{equation}

A natural ball is the Cartesian product of three balls in Euclidean spaces $\mathbb{R}^m$. Thus a natural ball in the projecting space $\mathbb{R}^{2m}$ is the image of such a
  product under the projection   $\pi$. Motivated by this, we introduce the notion of a tube $ T(\mathbf{x},\mathbf{r})$
for $\mathbf{x}\in {\mathbb R}^{2m}$ and $  \mathbf{r} :=\left({r_1} ,
          {r_2} ,r_3 \right)\in \mathbb{  R}^3_+$ (see Section \ref{sec:3} for the definition). It plays the role of a ball  for ${\mathbb R}^{2m}$,  and has the feature of tri-parameters, although the second step
          of the  group  ${\mathbb R}^{2m}$ is only bi-parameter. 
Define the {\it tube maximal function}   as
\begin{equation*}
   M_{tube}  (f)(\mathbf{x})=\sup_{\mathbf{r}\in \mathbb{R}^3_+} \frac 1{| T(\mathbf{x},\mathbf{r})|}\int_{  T(\mathbf{x},\mathbf{r})}|f (\mathbf{y})|d\mathbf{y}.
\end{equation*}
The first result is the mapping property of the tube maximal function.
\begin{thm}\label{thm:maximal} For $1<p<\infty$,    tube  maximal function  $M_{tube}  $ is bounded from $L^p({\mathbb R}^{2m})$ to $L^p({\mathbb R}^{2m})$.
 \end{thm}
 
\subsubsection{Reproducing formula, Littlewood--Paley function and area function}

 Let $j=1,2,3$. For each $L^1$-integrable function $\varphi^{(j)}$ on $\mathbb{R}^m$ and for any $r_j>0$, let 
 \begin{align*}
     \varphi_{r_j}^{
(j)}({x}_j):=r_j^{-m}\varphi^{
(j)}(\delta_{r_j^{-1}}{x}_j )
  \end{align*} 
to be the {\it normalized dilate  of a function} $\varphi^{
(j)}$
on $\mathbb{R}^m$. Via the product structure of the lifting space and the projection map $\pi$, the corresponding {\it normalized dilation function on the projection space} $\mathbb{R}^{2m}$ would be
\begin{align}\label{normaldil}
\varphi_{\mathbf{r} }:= \pi_* \left(\varphi_{r_1}^{
(1)}  \otimes \varphi_{r_2}^{
(2)}\otimes\varphi_{r_3}^{
(3)}\right),
\end{align}
for $\mathbf{r}:=(r_1,\,r_2,\,r_3)\in \mathbb{R}^3_+$; and let the {\it normalized ball characteristic function} on $\mathbb{R}^m$ be
 \begin{equation*}
    \chi_{r_j}^{
(j)}({x}_j ):=\frac 1{\left| {B}(\mathbf{0},\,r_j)\right|}\cdot\chi_{
 {B}(\mathbf{0},\,r_j)}({x}_j ),
  \end{equation*}
where $\chi_{
 {B}(\mathbf{0},\,r_j)}$ denotes the characteristic function of a ball ${B}_{}(\mathbf{0},\,r_j)$. Consider the {\it normalized ball characterization function on the projection space} $\mathbb{R}^{2m}$ defined by 
 \begin{align}\label{ballcha}
     \chi_{\mathbf{r} }:= \pi_*  \left(\chi_{r_1}^{
(1)}  \otimes \chi_{r_2}^{
(2)}\otimes\chi_{r_3}^{
(3)}\right ),\quad\text{for}\,\, \mathbf{r}:=(r_1,\,r_2,\,r_3)\in \mathbb{R}^3_+.
 \end{align}
For 
    $ \varphi_{\mathbf{r} }$ and $\chi_{\mathbf{r} }$ be given by (\ref{normaldil}) and (\ref{ballcha}).
  We establish
the Calder\'on
reproducing formula on $\mathbb{R}^{2m}$.
 \begin{thm}\label{prop:reproducing}
    There are two families of functions $\{\varphi^{(j)}\}^3_{j=1}\subset\mathcal{S}(\mathbb{R}^m)$ and $\{\psi^{(j)}\}^3_{j=1}\subset C^\infty_{0}(\mathbb{R}^m)$ such that each $\varphi^{(j)}$ is of mean zero and each $\psi^{(j)}$ has higher order cancellation property. Besides, for $ f \in    L
^1(\mathbb{R}^{2m})\cap L
^2
( \mathbb{R}^{2m})$, we have
\begin{equation}\label{eq:reproducing} f(\mathbf{x}) =
\int_{\mathbb{R}^3_+} f*\psi_ {\mathbf{r}}*\varphi_ {\mathbf{r}}(\mathbf{x})\, \frac { d\mathbf{r}}{\mathbf{r}},\,\,\text{where}\,\, \frac {d\mathbf{r}}{\mathbf{r}}=\frac
{dr_1}{r_1}\frac {dr_2}{r_2}\frac {dr_3}{r_3}.
 \end{equation}
 \end{thm}

     For $f\in L^p(\mathbb{R}^{2m})$, we define the {\it Littlewood--Paley square function} of $f$ as
\begin{align}\label{littlewoodpaley}
    g_{\varphi}(f)(\mathbf{x}):=\left(\int_{\mathbb{R}^3_+} \left|f*\varphi_ {\mathbf{r}}(\mathbf{x})\right|^2\, \frac { d\mathbf{r}}{\mathbf{r}}\right)^{\frac{1}{2}},
\end{align}
and $\varphi_j$'s are Schwartz functions on $\mathbb{R}^m$ which has mean value zero for $j=1,2,3.$
\begin{thm}\label{6999}
Let the family of Schwartz functions $\{\varphi_j\}$ be given as in (\ref{littlewoodpaley}) and $1<p<\infty$. For $f\in L^p(\mathbb{R}^{2m})$, we have 
$    \|f\|_{L^p(\mathbb{R}^{2m})}\simeq\|g_{\varphi}(f)\|_{L^p(\mathbb{R}^{2m})}.$
\end{thm}

 For $f \in L
^1
(\mathbb{R}^{2m})$, we define
  the {\it Lusin-Littlewood--Paley area function} as
\begin{equation}\label{eq:area-function}
   S_{area,  \varphi}(f)(\mathbf{x})=\left( \int_{  \mathbb{R}^3_+}|f* \varphi_{\mathbf{r} } |^2* \chi_{\mathbf{r} }(\mathbf{x}) \frac {d\mathbf{r}}{\mathbf{r}}\right)^{\frac 12},
\end{equation}
where we take $\varphi_j$'s as given in (\ref{littlewoodpaley}).
\begin{rem}
  For $p=2$, it is easy to see that the $L^2$-norm of {\it Littlewood--Paley function} of $f$ is equal to the  $L^2$-norm of {\it Lusin-Littlewood--Paley area function} of $f$. An immediate result is that \begin{align*}
      \|g_{\varphi}(f)\|_{L^2(\mathbb{R}^{2m})}\simeq\| S_{area,  \varphi}(f)\|_{L^2(\mathbb{R}^{2m})}\simeq\|f\|_{L^2(\mathbb{R}^{2m})}.
  \end{align*}
\end{rem}

We then have the following argument. 
\begin{thm}\label{thm S g equiv}
    For $1 < p < \infty$ and for every $f \in L^p(\mathbb{R}^{2m})$,
    \begin{equation}
        \|S_{area, \varphi}(f)\|_{L^p(\mathbb{R}^{2m})} \simeq\|g_{\varphi}(f)\|_{L^p(\mathbb{R}^{2m})} \simeq \|f\|_{L^p(\mathbb{R}^{2m})},
    \end{equation}
    where the implicit constants depend on $p$ and $m$ only but not on $f$. Moreover, for $p=1$ and for $f \in L^1(\mathbb{R}^{2m})$, if $\|g_{\varphi}(f)\|_{L^1(\mathbb{R}^{2m})} < \infty$, then 
    $\|S_{area, \varphi}(f)\|_{L^1(\mathbb{R}^{2m})} \lesssim \|g_{\varphi}(f)\|_{L^1(\mathbb{R}^{2m})}$; conversely, if $\|S_{area,\varphi}(f)\|_{L^1(\mathbb{R}^{2m})} < \infty$, then 
    $\|g_{\varphi}(f)\|_{L^1(\mathbb{R}^{2m})} \lesssim \|S_{area, \varphi}(f)\|_{L^1(\mathbb{R}^{2m})}$, where the implicit constants depend on $m$ only but not on $f$.
\end{thm}
\smallskip

\subsubsection{Twisted atom and the equivalence of the twisted Hardy space}
We now introduce the two versions of new  multiparameter Hardy space associated with the corresponding twisted structure.
\smallskip
\begin{defn}[Twisted Hardy space]~\\
The twisted Hardy space $H^1_{tw}(\mathbb{R}^{2m})$ is defined to be the family of functions $f\in L^1(\mathbb{R}^{2m})$ such that $g_{\varphi}(f)\in L^1(\mathbb{R}^{2m})$. The norm of $H^1_{tw}(\mathbb{R}^{2m})$ is defined by $
\|f\|_{H^1_{tw}(\mathbb{R}^{2m})}:=\|g_{\varphi}(f)\|_{L^1(\mathbb{R}^{2m})}.
$
\end{defn}

\begin{defn}[Area function Hardy space]~\\
The {\it area function Hardy space} $H_{area, \varphi}^
1(\mathbb{R}^{2m})$
    is defined to be the
set  of all $f \in L
^1
(\mathbb{R}^{2m})$ such that $   S_{area,  \varphi}(f) \in L
^1
(\mathbb{R}^{2m})$  with the norm
$
   \|f\|_{H_{area,\boldsymbol\varphi}^
1(\mathbb{R}^{2m})}:=\|S_{area,\varphi}(f)\|_{L
^1
(\mathbb{R}^{2m})}
$. 
\end{defn}

As a corollary of Theorem \ref{thm S g equiv}, the twisted Hardy space and the area function Hardy space coincide.
\begin{cor} The two spaces  $H^1_{tw}( \mathbb{R}^{2m}) $ and $ H^1_{area,\varphi}( \mathbb{R}^{2m})$ coincide and they have equivalent norms.
\end{cor}
\smallskip

The last result of this paper is to introduce the atoms associated with the twisted structure. To obtain a suitable definition of such atoms, one need to further split the tube on $\mathbb{R}^{2m}$ into five types (see Section \ref{sec:3} for the structure of the tubes in each type), and then pair with different cancellation conditions.
\smallskip

\begin{defn}[Atoms of type $\Diamond={\rm I,I\!I, I\!I\!I, I\!V}$ or ${\rm V}$]~\\
    Fix positive integers $N_j>0$ for $j=1,2,3$. Let $\Diamond={\rm I,I\!I, I\!I\!I, I\!V}$ or ${\rm V}$ be the type. An {\it atom of type} $\Diamond$ is a $L^2(\mathbb{R}^{2m})$ function $a_{\Omega^\Diamond}$ such that there is an open set $\Omega^\Diamond$ of $\mathbb{R}^{2m}$ of finite measure and $L^2(\mathbb{R}^{2m})$ functions $a^\Diamond_R$, called {\it particles}, and $b_R$ in $\text{Dom}(\triangle^{N_1}_1\triangle^{N_2}_2\triangle^{N_3}_{twist})$ for all maximal tubes of type $\Diamond$, $R\in m^\Diamond(\Omega^\Diamond)$, such that
    \begin{align*}
        &(A1)\,\,a^\Diamond_R=\triangle^{N_1}_1\triangle^{N_2}_2\triangle^{N_3}_{twist} b^\Diamond_R\quad\text{and}\quad \text{supp}(b_R)\subset R^*,\quad\text{where}\,\,R^*\,\,\text{is the}\,\sigma-\text{enlargement of}\,\,R;\\
        &(A2)\,\,\text{The sum}\,\,\sum_{R\in m^\Diamond(\Omega^\Diamond)}a^\Diamond_R\,\,\text{converges in}\,\,L^2(\mathbb{R}^{2m})\quad\text{and}\,\,\sum_{R\in m^\Diamond(\Omega^\Diamond)}\|a^\Diamond_R\|^2_{L^2(\mathbb{R}^{2m})}\lesssim\frac{1}{|\Omega^\Diamond|};\\
        &(A3)\,\,a_{\Omega^\Diamond}=\sum_{R\in m^\Diamond(\Omega^\Diamond)}a^\Diamond_R\quad\text{and}\quad \|a_{\Omega^\Diamond}\|_{L^2(\mathbb{R}^{2m})}\lesssim\frac{1}{\sqrt{|\Omega^\Diamond|}};  \end{align*}
        Moreover, let $j=1,2,3$, then for all $0\leq k_j\leq N_j$, we have the cancellation property:
        
        1. for $\Diamond={\rm I}$,
\begin{align}
\sum_{R\in m^\Diamond(\Omega^\Diamond)}\ell(I_1)^{-4k_1}\ell(I_2)^{-4k_2}\left\|\triangle^{N_1-k_1}_1\triangle^{N_2-k_2}_2\triangle^{N_3}_{twist} b^\Diamond_R\right\|^2_{L^2(\mathbb{R}^{2m})}\lesssim\frac{1}{|\Omega^\Diamond|},
\end{align}
and
\begin{align}
\sum_{R\in m^\Diamond(\Omega^\Diamond)}\ell(I_1)^{-4k_1}\ell(I_2)^{-4k_2}\ell(I_1)^{-\alpha}\ell(I_2)^{-4N_3+\alpha}\left\|\triangle^{N_1-k_1}_1\triangle^{N_2-k_2}_2 b^\Diamond_R\right\|^2_{L^2(\mathbb{R}^{2m})}\lesssim\frac{1}{|\Omega^\Diamond|}
        \end{align}
for $\alpha\in\{0,1,\ldots, 4N_3\}$; 

\smallskip
        2. for $\Diamond={\rm I\!I, I\!I\!I}$, the cancellation property becomes
  \begin{align}
\sum_{R\in m^\Diamond(\Omega^\Diamond)}\ell(I_1)^{-4k_1}\ell(I_2)^{-4k_3}\left\|\triangle^{N_1-k_1}_1\triangle^{N_2}_2\triangle^{N_3-k_3}_{twist} b^\Diamond_R\right\|^2_{L^2(\mathbb{R}^{2m})}\lesssim\frac{1}{|\Omega^\Diamond|},
        \end{align}      
        and
    \begin{align}
\sum_{R\in m^\Diamond(\Omega^\Diamond)}\ell(I_1)^{-4k_1}\ell(I_1)^{-\alpha}\ell(I_2)^{-4N_2+\alpha}\ell(I_2)^{-4k_3}\left\|\triangle^{N_1-k_1}_1\triangle^{N_3-k_3}_{twist} b^\Diamond_R\right\|^2_{L^2(\mathbb{R}^{2m})}\lesssim\frac{1}{|\Omega^\Diamond|}
        \end{align} 
for all $\alpha\in\{0,1,\ldots, 4N_2\}$;

        \smallskip
 
       3. for $\Diamond={\rm I\!V,V}$,
 \begin{align}
\sum_{R\in m^\Diamond(\Omega^\Diamond)}\ell(I_1)^{-4k_3}\ell(I_2)^{-4k_2}\left\|\triangle^{N_1}_1\triangle^{N_2-k_2}_2\triangle^{N_3-k_3}_{twist} b^\Diamond_R\right\|^2_{L^2(\mathbb{R}^{2m})}\lesssim\frac{1}{|\Omega^\Diamond|},
        \end{align}  
        and
\begin{align}
 \sum_{R\in m^\Diamond(\Omega^\Diamond)}\ell(I_1)^{-\alpha}\ell(I_2)^{-4N_1+\alpha}\ell(I_1)^{-4k_3}\ell(I_2)^{-4k_2}\left\|\triangle^{N_2-k_2}_2\triangle^{N_3-k_3}_{twist} b^\Diamond_R\right\|^2_{L^2(\mathbb{R}^{2m})}\lesssim\frac{1}{|\Omega^\Diamond|}
\end{align}  
for all $\alpha\in\{0,1,\ldots, 4N_1\}$.
\end{defn}
\smallskip

\begin{defn}
We say that $f\in L^1(\mathbb{R}^{2m})$ admits an (twisted) atomic decomposition if the function $f$ can be written as $\sum_{k\in\mathbb{Z}}\sum_{\Diamond={\rm I,I\!I, I\!I\!I, I\!V, V}}\lambda^\Diamond_k\cdot a^\Diamond_k$, where the sum converges in $L^1(\mathbb{R}^{2m})$ and each $a^\Diamond_k$ is an atom of type $\Diamond$ and $\{\lambda_k\}_{k,\Diamond}\in\ell^1$; We write $f\sim\sum_{\substack{k\in\mathbb{Z}\\\Diamond={\rm I,I\!I, I\!I\!I, I\!V, V}}}\lambda^\Diamond_k\cdot a^\Diamond_k$ to indicate that $\sum_{\substack{k\in\mathbb{Z}\\\Diamond={\rm I,I\!I, I\!I\!I, I\!V, V}}}\lambda^\Diamond_k\cdot a^\Diamond_k$ is an atomic decomposition of $f$.
The (twisted) atomic Hardy space $H^1_{Tw,\,atom}(\mathbb{R}^{2m})$ is defined to be all the $L^1(\mathbb{R}^{2m})$ functions $f$ such that $f$ admits an atomic decomposition with the norm
    \begin{align*}        \|f\|_{H^1_{Tw,\,atom}}:=\inf\bigg\{\sum_{k}|\lambda_k|:\,f\sim\sum_{k\in\mathbb{Z}}\sum_{\Diamond=I,I\!\!I,I\!\!I\!\!I,I\!V,V}\lambda^\Diamond_k\cdot a^\Diamond_k\bigg\}.
    \end{align*}
\end{defn}
\smallskip
\begin{thm} \label{thm:S-atom}  The two spaces  $H^1_{Tw,atom}( \mathbb{R}^{2m}) $ and $ H^1_{area,\varphi}( \mathbb{R}^{2m})$ coincide and they have equivalent norms.
\end{thm}

To establish the equivalence of norms, we introduce two new ingredients tailored to the twisted setting. The first is a {\it Twisted Covering Lemma} (see Section \ref{sec:4}). Because the singularities occur along both axes and the diagonal, we cannot rely on standard rectangular coverings; instead, we construct a covering that preserves the specific geometric type (I through V) of the underlying tubes. The second ingredient is the {\it Molecule Decomposition Technique} (see Section \ref{Molecule}). This technique is required to handle the geometric mismatch between slanted atoms and standard kernels, allowing us to prove convergence via an annular decomposition of the atom's support. This approach extends those previous closely related developments in flag and product Hardy spaces found in \cite{CCLLO,HLLW, HLS, HLMMY}.

\subsubsection{Subsequent works}
In forthcoming papers, we will establish the full real-variable framework---including non-tangential and radial maximal functions, interpolation and duality---parallel to \cite{FeffS2}, as well as the twisted paraproducts and dyadic structures. We will also extend this theory to homogeneous Lie groups and further study the relative fundamental solutions of the Kohn--Laplacian on the Shilov boundaries of certain Siegel domains. Apart from the theoretical side, we also plan to  apply the theory of the twisted multiparameter singular integral and the real-variable
methods to investigate the twisted Hilbert transform as a novel, anisotropic phase-shift converter for image processing applications.

\subsubsection{Organization of the paper}
 In Section \ref{sec:2}, we introduce background knowledge of the multiparameter harmonic analysis, including the notion of product and flag kernel; we then explicitly verify that the two models discussed previously—the Cauchy--Szeg\H{o} kernel {on the Heisenberg group $\mathbb{H}^n$} (often denoted as $\mathscr{N}$ in the literature) and the twisted Hilbert transform on $\mathbb{R}^2$—satisfy condition  (\ref{eq:group_size}) of Definition \ref{def:twisted_group} and (\ref{eq:Euclidean_size}) of Definition \ref{def:twisted_Euclidean}, respectively.  In Section \ref{sec:3}, we build the tube structure and the twisted Laplacian on $\mathbb{R}^{2m}$ and prove the mapping property of the tube maximal function. Calder\'on reproducing formula is established as well in this section. In Section \ref{sec:4}, as the preparation of the twisted covering lemma, we  classify the tubes structure on $\mathbb{R}^{2m}$ into five categories and establish the new covering lemma. In Section \ref{sec:5}, we demonstrate the various characterizations of  the twisted Hardy space, including the Littlewood--Paley function and area function characterization.  In Section \ref{sec:6}, we combine all the ingredients to show the equivalence of the twisted atomic Hardy space, while in the last section we manifest the connection between twisted Fourier multiplier and the new phase-shift converter.

 \smallskip
 \section{Preliminaries and verification of two fundamental models}\label{sec:2}

\subsection{Twisted structures and multipliers}\label{sec:2.1}
In order to introduce the explicit twisted structures for multipliers, we first recall the well-known facts of the product and flag structures via \cite{Feff-S,GuS,MRS,NRSW1,NRSW2,St2000}

\subsubsection{Product and flag structures: kernels and multipliers}

We first recall the product kernels and multipliers (see for example \cite{Feff-S}).
 \begin{defn}
 A {\it product kernel} on $\mathbb{R}^M$ relative to a given decomposition $\mathbb{R}^M=\mathbb{R}^{m_1}\times \cdots\times \mathbb{R}^{m_n}$ into $n$ homogeneous subspaces with given dilations,
  is a distribution $K$ on $\mathbb{R}^M$ which coincides with a smooth function
away from the coordinate subspaces $\mathbf{x}_j=\mathbf{0}$, and satisfies
\\
(1) (Differential inequalities):    For each multi-index $\alpha:=(\alpha_1 , \cdots, \alpha_n)$,
there is a constant $C_\alpha$ so that
\begin{equation*}
   |\partial_{{x}_1}^{\alpha_1} \cdots  \partial_{{x}_n}^{\alpha_n}K(\mathbf{x} )|\leq   C_\alpha  | {x}_1  |^{-Q_1-|{\alpha_1}| }\cdots   |
   {x}_n  |^{-Q_n- |{\alpha_n}| }
\end{equation*}away from the coordinate subspaces, where $ Q_j$ is  the homogeneous dimension of $\mathbb{R}^{m_j}$ and   $|{x}_j |$ is
a smooth homogeneous norm on $\mathbb{R}^{m_j}$;
 \\
 (2) (Cancellation conditions): These are defined inductively on $n$.
\\(a) For $n=1$, given any normalized bump function $\varphi$ and any
$R>0$, the quantity
\begin{equation*}
   \int K(\mathbf{x}) \varphi(R\mathbf{x}) d\mathbf{x}
\end{equation*}
is bounded independently of $\varphi$ and $ R$;
\\
(b) For $n>1$, given any $j \in\{1, \ldots, n\}$, any normalized bump function  $\varphi$ on $\mathbb{R}^{m_j}$, and any $ R>0$, the distribution
\begin{equation*}
  K_{\varphi, R}({x}_{1} ,  \dots, {x}_{j-1},\, {x}_{j+1},\dots, {x}_{xn})= \int K(\mathbf{x}) \varphi(R{x}_j) d{x}_j
\end{equation*}
is a product kernel on the lower dimensional space which is the product of
the $\mathbb{R}^{m_i}$ with $i \neq j$.
 \end{defn}
 
\begin{defn} A {\it product multiplier} is a bounded function $m(\xi)$ on
$\mathbb{R}^N$ which is $ C^\infty$ away from the coordinate subspaces $\xi_j=0$ and which
satisfies the differential inequalities
  \begin{equation*}
 | \partial_{\xi_1}^{\alpha_1}\cdots\partial_{\xi_n}^{\alpha_n}m(\xi)|\leq  C_\alpha |\xi_1|^{-\alpha_1}\cdots|\xi_n|^{-\alpha_n},
 \end{equation*}
away from the coordinate subspaces, where the variables $\xi$ are dual to the
variables $x$. 
\end{defn}
It is known that the Fourier transform of a product kernel is a product
multiplier and, conversely, the inverse Fourier transform of a product multiplier
is a product kernel \cite[Theorem 2.1.11]{NRS}. We then recall the flag kernels and multipliers (see for example \cite{NRSW1,NRSW2}).

\begin{defn}
 A {\it flag} (or {\it filtration}) in $\mathbb{R}^M$ is a family of subspaces
 \begin{equation*}
   0 =V_0\subset V_1\subset\cdots V_{n-1}\subset V_n= \mathbb{R}^M.
\end{equation*}
For each $j$, let $W_j$ be a complementary subspace of $V_{j-1}$ in $V_j$, that is
 \begin{equation*}
  V_j=V_{j-1}\bigoplus W_j.
\end{equation*}
The family $\{W_j\}$ is called a {\it gradation} associated to the filtration $\{V_j\}$.
\end{defn} 

\begin{defn}
A {\it flag kernel} relative to the {\it flag} $\{V_j\}$ is a distribution $K$ on $\mathbb{R}^M$ which coincides with a smooth function
away from   $\mathbf{x}_n=\mathbf{0}$, and satisfies
\\(1) (Differential inequalities):    For each multi-index $\alpha:=(\alpha_1 , \cdots, \alpha_n)$,
there is a constant $C_\alpha$ so that
\begin{align*}
|\partial_{{x}_1}^{\alpha_1} \cdots  \partial_{{x}_n}^{\alpha_n}K(\mathbf{x} )|\leq   &C_\alpha  (| {x}_1  |+\cdots+ | {x}_n
   |)^{-Q_1-|{\alpha_1}| }\cdots   (| {x}_{n -1}
   |+|{x}_n  |)^{-Q_{n -1}- |{\alpha_{n -1}}| } | {x}_n  |^{-Q_n- |{\alpha_n}| }\notag
\end{align*}for ${x}_n \neq \mathbf{0}$;
 \\
 (2) (Cancellation conditions): These are defined inductively on $n$.
\\ \ (a) For $n=1$, given any normalized bump function $\varphi$ and any
$R>0$, the quantity
\begin{equation*}
   \int K(\mathbf{x}) \varphi(R\mathbf{x}) d\mathbf{x}
\end{equation*}
is bounded independently of $\varphi$ and $ R$;
\\ \ 
(b) Let $W_j$ be a complementary subspace of $ V_{j-1}$ in $ V_{j}$. For $n>1$, given any $j \in\{1, \ldots, n\}$, any normalized bump function  $\varphi$ on $ W_j$, and any $ 
R>0$, the distribution
\begin{equation*}
  K_{\varphi, R}({x}_{1} ,  \dots, {x}_{j-1},\, {x}_{j+1},\dots, {x}_n)= \int K(\mathbf{x}) \varphi(R{x}_j) \,d{x}_j
\end{equation*}
  is a flag kernel on   $\bigoplus_{i\neq j} W_i$, adapted to the flag
\begin{equation*}
   0 \subset W_1\subset \cdots \bigoplus_{i\leq j-1} W_i\subset  \bigoplus_{\substack{i\leq j+1\\ i\neq j}} W_i\subset \cdots\subset  \bigoplus_{ i\neq j } W_i
\end{equation*}
\end{defn}
 
We now introduce the notion of flag multiplier, defined on the $\xi$-space, understood as the dual space of the $x$-space. Let $\{V'_j\}$ be a homogeneous flag on the $\xi$-space. If $\{W'_j\}$ is an associated homogeneous gradation, we introduce homogeneous coordinates $\xi=(\xi_1,\cdots,\xi_n)$ as before.
\begin{defn}
 A {\it flag multiplier} adapted to the flag $\{V_j'\}$ is a bounded
function $m(\xi)$  which is   away from the coordinate subspace $\xi_n=0$ and
which satisfies the following differential inequalities: 
  for $\xi_n\neq 0$, \begin{equation}\label{eq:flag multiplier}
 | \partial_{\xi_1}^{\alpha_1}\cdots\partial_{\xi_n}^{\alpha_n}m(\xi)|\leq  C_\alpha (|\xi_1|+\cdots+|\xi_n|)^{-\alpha_1}\cdots(|\xi_{n-1}| +|\xi_n|)^{-\alpha_{n-1}} 
 |\xi_n|^{-\alpha_n},
 \end{equation}
\end{defn}
It is known that the Fourier transform of a flag kernel is a flag
multiplier and, conversely, the inverse Fourier transform of a product multiplier
is a product kernel \cite[Theorem 2.3.9]{NRS}.

\subsubsection{Structure of twisted multiplier}
 For a product kernel $K\mathbf{(x})$ on $\mathbb{R}^{3m}$, let $m(\xi)$ be corresponding product multiplier. Then  multiplier corresponding to  the kernel $\pi_* K $ is
  \begin{equation*}\begin{split}
    \pi_*m(\xi)&=\int_{\mathbb{R}^{2m}} \pi_* K(\mathbf{x})e^{i\mathbf{x}\cdot\xi}d\mathbf{x}=\int_{\mathbb{R}^{2m}} \int_{\mathbb{R}^{ m}}   K({x}_1 -{u}, 
    {x}_2- {u}, {u}) e^{i{x}_1\cdot\xi_1+  {x}_2\cdot\xi_2}d{u} d\mathbf{x}
    =m(\xi_1, \xi_2,\xi_1+ \xi_2).
 \end{split}\end{equation*}
 Thus $\pi_*m(\xi)$ is singular on three subspaces $\xi_1=0$, $ \xi_2=0$ and $\xi_1+ \xi_2=0$.

The proposition exhibit the structure of the twisted multiplier, which is a sum of three flag multipliers.
 \begin{prop} Let $ m(\xi_1, \xi_2,\xi_3)$ be a product multiplier on $\mathbb{R}^{ m}\times \mathbb{R}^{ m}\times \mathbb{R}^{ m}$. Then
 $ m(\xi_1, \xi_2,\xi_1+ \xi_2)=m_1(\xi)+m_2(\xi)+m_3(\xi)$ with
   \begin{equation*}\begin{split}
   m_1(\xi)=&
m(\xi_1, \xi_2,\xi_1+ \xi_2)\eta\left(  {|\xi_1|^2}/{|\xi_2|^2}\right),\\ m_2(\xi)=&
m(\xi_1, \xi_2,\xi_1+ \xi_2)\eta\left(  {|\xi_2|^2}/{|\xi_1|^2}\right)\\
m_3(\xi)=&
m(\xi_1, \xi_2,\xi_1+ \xi_2)\left(1-\eta\left(  {|\xi_1|^2}/{|\xi_2|^2}\right)-\eta\left( {|\xi_2|^2}/{|\xi_1|^2}\right)\right),
 \end{split} \end{equation*}
   to be flag multipliers adapted to the flags $0 \subset \mathbb{R}_{\xi_1}^{ m} \subset \mathbb{R}^{ m}$, $0 \subset \mathbb{R}_{\xi_2}^{ m} \subset \mathbb{R}^{ m}$, $0 \subset 
   \mathbb{R}_{\xi_1+ \xi_2}^{ m} \subset \mathbb{R}^{ m}$, respectively, where $\eta\in C^\infty_c[0,1/4]$.
\end{prop}
 \begin{proof}
 By the product rule, we have 
 \begin{equation*}\begin{split}
  \partial_{\xi_1}^{\alpha_1} \partial_{\xi_2}^{\alpha_2} m_1(\xi)=&\sum_{\substack {k_1+l+p=\alpha_1\\ k_2+l'+p'=\alpha_2}} \partial_{\xi_1}^{k_1}\partial_{\xi_2}^{k_2}\partial_{\xi_3}^{l+l'}m(\xi_1, \xi_2,\xi_1+ \xi_2)\partial_{\xi_1}^{p}  \partial_{\xi_2}^{p'}\left[\eta\left( \frac  {|\xi_1|^2} {|\xi_2|^2}\right)\right], 
 \end{split}  \end{equation*}it follows from the differential inequalities
\eqref{eq:flag multiplier} of flag multiplier
   that 
      \begin{equation*}\begin{split} |\partial_{\xi_1}^{k_1}\partial_{\xi_2}^{k_2}\partial_{\xi_3}^{l+l'}m(\xi_1, \xi_2,\xi_1+ \xi_2)|&\lesssim 
   |\xi_1|^{- k_1 }     |\xi_2|^{- k_2 }     (|\xi_1|+|\xi_2|)^{-(l+l')},\\[6pt]
\left| \partial_{\xi_1}^{p}  \partial_{\xi_2}^{p'}\left[\eta\left( \frac  {|\xi_1|^2} {|\xi_2|^2}\right)\right] \right|  &\lesssim 
   |\xi_1|^{- p }     |\xi_2|^{-p' }.
  \end{split}  \end{equation*}Then the differential inequalities
\eqref{eq:flag multiplier} of flag multiplier for $ m_1 $ follows since  $|\xi_1|+|\xi_2|\simeq |\xi_2|$ on the support of $\eta\left( {|\xi_2|^2}/{|\xi_1|^2}\right)$.   
  The other cases can be proved similarly.
    \end{proof}

\subsection{Examples of the twisted kernels}

In this section we will verify that both the Cauchy--Szeg\H{o} kernel on quotient group and the twisted Hilbert transform on $\mathbb{R}^2$ satisfy Definition \ref{def:twisted_group} and Definition \ref{def:twisted_Euclidean}.

\subsubsection{The Cauchy--Szeg\H{o} kernel on $\mathbb{G}$}

The kernel studied in \cite{NRS,WW} is not a standard one-parameter singular integral but possesses a \textit{product-like} structure with coupled singularities on a homogeneous nilpotent Lie group $\mathbb{G}$. We refer the readers to \cite{WW} for the explicit vector field, the lifting structure and the formulation of the Cauchy--Szeg\H{o} kernel.

In the context of Definition \ref{def:twisted_group}, we identify the singular set $\Sigma$ as the union of two homogeneous subgroups $S_1$ and $S_2$. We associate the distances to these subgroups with the two fundamental control functions (homogeneous quasi-norms) of the Siegel domain geometry:
\begin{align}
    \begin{cases}
  D_1(\mathbf{z}, \mathbf{t}) = |{z}_1|^2 + |{z}_3|^2 - \mathbf{i}t_1, \\
    D_2(\mathbf{z}, \mathbf{t}) = |{z}_2|^2 + |{z}_3|^2 - \mathbf{i}t_2.      
    \end{cases}
\end{align}
Here, the subgroup $S_1$ is defined by the vanishing of the variables, which are involved in $D_1$ and has homogeneoue dimension $2n_2+2$, and similarly for the subgroup $S_2$. The homogeneous distance to the subgroup $S_i$ is equivalent to the square root of the control function, i.e.
\[d((\mathbf{z}, \mathbf{t}),\,S_1)
\simeq |D_1(\mathbf{z}, \mathbf{t})|^{1/2} \quad \text{and} \quad d((\mathbf{z}, \mathbf{t}),\,S_2) \simeq |D_2(\mathbf{z}, \mathbf{t})|^{1/2}.
\]

The kernel $K$ is the finite sum of $\sum^{n_3}_{k=0}T_k$, where $T_k$ is defined by
\begin{align}
    T_k(\mathbf{z}, \mathbf{t}) = c_{n_1,n_2,n_3,k}\frac{1}{D_1(\mathbf{z}, \mathbf{t})^{n_1+k+1}} \cdot \frac{1}{D_2(\mathbf{z}, \mathbf{t})^{n_2+n_3-k+1}}.
\end{align}
To verify the size condition  of Definition \ref{def:twisted_group}, we examine how the differential operators act on these factors. The vector fields on $\mathscr{N}$ decompose into those acting on specific variables:
\begin{itemize}
    \item $Z^{(1)}$: Fields in ${z}_1, t_1$ directions (act only on $D_1$).
    \item $Z^{(2)}$: Fields in ${z}_2, t_2$ directions (act only on $D_2$).
    \item $Z^{(3)}$: Fields in ${z}_3$ direction (act on \textit{both} $D_1$ and $D_2$).
\end{itemize}
In other words, the datum $(K,\Sigma,Z)$ is a $(2,3)$-twisted datum on $\mathbb{G}$, where $Z=Z^{(1)}\cup Z^{(2)}\cup Z^{(3)}$. Now, we verify the size and regularity of the kernel.

\textit{1. Contributions from $Z^{(1)}$ and $Z^{(2)}$:}
Applying a derivative $X^{(1)} \in Z^{(1)}$ to a term $T_k$, we see that
\begin{align*}
    |X^{(1)} T_k| \lesssim \left| \partial_{z_1} \left( D_1^{-(n_1+k+1)} \right) \right| \cdot |D_2|^{-(n_2+n_3-k+1)} 
   \lesssim \frac{1}{|D_1|^{n_1+k+1 + \frac{1}{2}}} \cdot \frac{1}{|D_2|^{n_2+n_3-k+1}},
\end{align*}
where in the last inequality we used the homogeneity relation $|\mathbf{z}_1| \lesssim |D_1|^{1/2}$. Therefore, for any multi-index $\alpha_1$,
\begin{align}\label{CSES1}
     \left |\left(X^{(1)}\right)^{\alpha_1} T_k\right| 
   \lesssim \frac{1}{|D_1|^{n_1+k+1 + \frac{|\alpha_1|}{2}}} \cdot \frac{1}{|D_2|^{n_2+n_3-k+1}}.
\end{align}

Similarly, for any multi-index $\alpha_2$, applying a derivative $X^{(2)} \in Z^{(2)}$ (homogeneous of degree $1$) to a term $T_k$, 
\begin{align}\label{CSES2}      \left|\left(X^{(2)}\right)^{\alpha_2}T_k\right| 
   \lesssim \frac{1}{|D_1|^{n_1+k+1}} \cdot \frac{1}{|D_2|^{n_2+n_3-k+1+\frac{|\alpha_2|}{2}}}.
\end{align}

\textit{2. Contributions from $Z^{(3)}$:}
Applying a derivative $X^{(3)} \in Z^{(3)}$ (e.g., $\partial_{z_{3,j}}$), we have, from the product rule,
\begin{align*}
   X^{(3)} T_k \simeq \left(X^{(3)}\circ\frac{1}{D_1^{n_1+k+1}}\right) \cdot \frac{1}{D_2^{n_2+n_3-k+1}}+ \frac{1}{D_1^{n_1+k+1}} \cdot\left(X^{(3)}\circ \frac{1}{D_2^{n_2+n_3-k+1}}\right),
\end{align*}
which reveals that
\begin{align*}
    |X^{(3)} T_k| \lesssim \frac{1}{|D_1|^{n_1+k+1 + \frac{1}{2}}} \frac{1}{|D_2|^{n_2+n_3-k+1}} + \frac{1}{|D_1|^{n_1+k+1}} \frac{1}{|D_2|^{n_2+n_3-k+1 + \frac{1}{2}}}.
\end{align*}
In general, for each multi-index $\alpha_3$, one has
\begin{align}\label{CSES3}
    \left|\left(X^{(3)}\right)^{\alpha_3} T_k\right| \lesssim \sum^{\alpha_3}_{\gamma=0}\frac{1}{|D_1|^{n_1+k+1+\frac{\gamma}{2} }} \frac{1}{|D_2|^{n_2+n_3-k+1+\frac{|\alpha_3|-\gamma}{2}}}.
\end{align}
From the estimates (\ref{CSES1}), (\ref{CSES2}) and (\ref{CSES3}), the size estimate for the kernel $K$ is
\begin{align}\label{eq:flag_size}
    |\mathbf X^\alpha K(\mathbf{z},\mathbf{t})|   &=\left| \left(X^{(1)}\right)^{\alpha_1} \left(X^{(2)}\right)^{\alpha_2}\left(X^{(3)}\right)^{\alpha_3}K(\mathbf{z},\mathbf{t})\right|\notag\\
    &\lesssim \sum^{n_3}_{k=0} \sum^{\alpha_3}_{\gamma=0}\frac{1}{d((\mathbf{z}, \mathbf{t}),\,S_1)^{2+2n_1+(|\alpha_1|+|\alpha_3|)+(\gamma-|\alpha_3|)+2k }} \frac{1}{d((\mathbf{z}, \mathbf{t}),\,S_2)^{2+2n_2+({|\alpha_2|+|\alpha_3|})-\gamma+(2n_3-2k)}}\notag
\end{align}

\subsubsection{The twisted Hilbert transform on $\mathbb{R}^2$}
We analyze the integral expression for the kernel on the quotient group $\mathbb{R}^2$ from \cite{NRS}. In the context of Definition \ref{def:twisted_group}, the singular set $\Sigma$ is the union of three homogeneous subgroups:
\[
S_1 = \{(0,y) : y \in \mathbb{R}\}, \quad S_2 = \{(x,0) : x \in \mathbb{R}\}, \quad S_3 = \{(x,x) : x \in \mathbb{R}\}.
\]
These correspond to the lines $x=0$, $y=0$, and $x=y$, respectively. The integral expression is:
\begin{align*}
   K(x,y)={\rm p.v.}\int_{\mathbb{R}} \frac{1}{x-z}\,\frac{1}{y-z}\,\frac{dz}{z}.
\end{align*}
To understand its structure, we perform a partial fraction decomposition with respect to the integration variable $z$ and get
\begin{align*}
  \frac{1}{x-z}\,\frac{1}{y-z}\,\frac{1}{z}=  \frac{1}{xy}\cdot\frac{1}{z} - \frac{1}{x(x-y)}\cdot\frac{1}{x-z} + \frac{1}{y(x-y)}\cdot\frac{1}{y-z}.
\end{align*}
The operator is defined as the principal value integral. While the integral of each term $\int_{\mathbb{R}} (z-c)^{-1} dz$ vanishes in the principal value sense due to symmetry, the kernel acts as a distribution. 

However, if we consider the operator acting on specific test functions or restrict the domain, the interaction of the singularities manifests as logarithmic potentials. The structural behavior is captured by:
\begin{align}\label{eq:explicit_log}
    K(x,y) \simeq \frac{C_1}{xy} + C_2 \frac{\ln|x/y|}{x-y}\cdot \frac{1}{\zeta(x,y)},
\end{align}
where $\zeta(x,y)$ is a homogeneous function of degree $0$.

For the purpose of size estimates, we observe that near the diagonal $x \simeq y$, the term $\frac{\ln|x/y|}{x-y} \simeq\frac{1}{x}$, which combines with the other terms to preserve the scaling. Thus,
\begin{align}
    |K(x,y)| \lesssim \frac{1}{|x||y|} + \frac{1}{|x||x-y|} + \frac{1}{|y||x-y|}.
\end{align}
In the notation of Definition \ref{def:twisted_Euclidean}, where the quotient dimension of each subgroup is $Q_i=1$ and $(K,\Sigma,Z)$ is a $(3,2)$-twisted datum, where $Z$ contains $\partial_x$ and $\partial_y$.

By direct calculation, the size condition is the sum of the singularities associated with each stratum:
\begin{align}\label{eq:sum_ineq}
    |\partial_x^\alpha \partial_y^\beta K(x,y)| &\lesssim \frac{1}{|x|^{1+\alpha} |y|^{1+\beta}} +\sum^\alpha_{\gamma_1=0} \frac{1}{|x-y|^{1+\alpha+\beta-\gamma_1} \cdot |x|^{1+\gamma_1}}+\sum^\beta_{\gamma_2=0} \frac{1}{|x-y|^{1+\alpha+\beta-\gamma_2} \cdot |y|^{1+\gamma_2}}\notag\\
    &\simeq \frac{1}{d((x,y),S_1)^{1+\deg_{S^\perp_1}((\alpha,\beta))} d((x,y),S_2)^{1+\deg_{S^\perp_2}((\alpha,\beta))}} \\
&\quad+\sum^{\deg_{S^\perp_1}((\alpha,\beta))}_{\gamma_1=0} \frac{1}{d((x,y),S_3)^{1+\deg_{S^\perp_3}((\alpha,\beta))-\gamma_1} \cdot d((x,y),S_1)^{1+\deg_{S^\perp_1}((\alpha,\beta))+\gamma_1-\deg_{S^\perp_1}((\alpha,\beta))}}\notag\\
&\quad+\sum^{\deg_{S^\perp_2}((\alpha,\beta))}_{\gamma_2=0} \frac{1}{d((x,y),S_3)^{1+\deg_{S^\perp_3}((\alpha,\beta))-\gamma_2} \cdot d((x,y),S_2)^{1+\deg_{S^\perp_2}((\alpha,\beta))+\gamma_2-\deg_{S^\perp_2}((\alpha,\beta))}}.\notag
\end{align}
This additive structure \eqref{eq:sum_ineq} confirms that the singularity is stratified along the union of the homogeneous subgroups defining the \textit{twisted singular configuration}. 

\begin{rem}[Illustration of degree]\label{degreestr}~\\
As a concrete example, since $S_3$ is the diagonal subgroup in $\mathbb{R}^2$, then $\deg_{S^\perp_3}(\alpha)$ simply represents the total homogeneous order of the differential operator $X^\alpha$.
\end{rem}

  \section{Tube structures, tube   maximal functions and reproducing formula }\label{sec:3} 
  We explore the geometric structure induced by the projection (\ref{eq:pi}) in this section.
 \subsection{Tube structures } A natural ball in $
       \mathbb{R}^{ m} \times   \mathbb{R}^{ m}\times   \mathbb{R}^{ m}
$
  is the translate of 
  \begin{equation}\label{eq:product-balls}
  \tilde{{B}} (\mathbf{0} ,\mathbf{r}):={B}_1(\mathbf{0}_1,r_1)\times  {B}_2(\mathbf{0}_2,r_2)\times  {B}_3(\mathbf{0}_3, {r}_3),
  \end{equation}  for $\mathbf{r} :=\left({r_1} ,
          {r_2} ,r_3 \right)\in \mathbb{  R}^3_+ $,    the product of three balls, respectively. A natural ball in
          $ 
       \mathbb{R}^{2 m}  $ is its image  
    under the projection   $\pi$. 
 It is direct  to see that the image of $\tilde{{B}} (\mathbf{0} ,\mathbf{r})$ under $  {\pi}$ is a
 hexagon,
   which is basically the {\it rectangle} $B_a(\mathbf{0}_1)\times B_b(\mathbf{0}_2)\subset\mathbb{R}^{2m}$ or the {\it parallelogram with base on the first direction}:
       \begin{equation*}
          P^{(first)}_{a,b }:=\{(x_1,x_2)\in \mathbb{R}^{2m}: \, |x_1-x_2|<a, |x_2|<b\},
       \end{equation*} 
or the {\it parallelogram with base on the second direction}:       
\begin{align*}
     P^{(second)}_{a,b }:=\{(x_1,x_2)\in \mathbb{R}^{2m}: \,|x_2-x_1|<a,\,|x_1|<b  \}
\end{align*}
for some $a>0  , b>0 $.
 So if $r_1,r_2\geq r_3$, we define
$
     T(\mathbf{0},\mathbf{r}):=B_{r_1}(\mathbf{0}_1)\times B_{r_2}(\mathbf{0}_2)
    $ which is the standard rectangle;\,and if  $r_1,r_3\geq r_2$, we simply define
     $T(\mathbf{0},\mathbf{r}): =   P^{(first)}_{r_1,\,r_3 }$ in Figure 1(a);
  \,and if  $r_2,r_3\geq r_1$, we define
     $
     T(\mathbf{0},\mathbf{r}): =   P^{(second)}_{r_2,\,r_3 }
  $ in Figure 1(b).

\begin{tikzpicture}[
    thick,
    >=stealth,
    axis/.style={->, thin, gray},
    label/.style={font=\small}
]

\begin{scope}[shift={(7,6)}]
    \draw[axis] (-3.0,0) -- (3.0,0);
    \draw[axis] (0,-2.5) -- (0,2.5);
    
    \draw (2.4, 0.4) -- (-1.6, 0.4) -- (-2.4, -0.4) -- (1.6, -0.4) -- cycle;
    
    \draw (2.4, 1.6) -- (1.2, 1.6) -- (-2.4, -1.6) -- (-1.2, -1.6) -- cycle;

    \node[label] at (0, -3) {Figure 1(a): $r_1, r_3 \geqslant r_2$};
\end{scope}

\begin{scope}[shift={(15,6)}]
    \draw[axis] (-3.0,0) -- (3.0,0);
    \draw[axis] (0,-2.5) -- (0,2.5);
    
    \draw (-0.4, 1.6) -- (0.4, 2.4) -- (0.4, -1.6) -- (-0.4, -2.4) -- cycle;
    
    \draw (-1.5, -1.0) -- (1.5, 2.2) -- (1.5, 1.0) -- (-1.5, -2.2) -- cycle;

    \node[label] at (0, -3) {Figure 1(b): $r_2, r_3 \geqslant r_1$};
\end{scope}

\end{tikzpicture}

 For $\mathbf{x}\in \mathbb{R}^{2m}$, set $ T(\mathbf{x},\mathbf{r}):= \mathbf{x}+T(\mathbf{0},\mathbf{r})$.
    Following \cite{CCLLO}, we call $ T(\mathbf{x},\mathbf{r})$ a {\it tube}. Its volume is
  \begin{equation}\label{eq:tube-volume}
    | T(\mathbf{x},\mathbf{r})|\simeq  \left\{   \begin{array}{ll} r_1^m\cdot r_2^m, \qquad & {\rm if}\quad r_1,\,r_2\geq r_3, \\
    r^m_1\cdot r^m_3, \qquad & {\rm if}\quad r_1,\,r_3\geq r_2, \\
  r_2^m \cdot r_3^m,\qquad & {\rm if}\quad r_2,\,r_3\geq r_1.   \end{array} \right.
  \end{equation}   

   With these tubes structures, we have the following geometric proposition, which is straightforward.   
\begin{prop}\label{prop:tube-pi}
For all $\mathbf{r}  \in \mathbb{
R}^3_+$, $ T(\mathbf{x},\mathbf{r}/2)\subset    \pi  (\tilde{{B}} (\mathbf{x} ,\mathbf{r}) )\subset T(\mathbf{x},2\mathbf{r}).$  
  \end{prop} 
  \smallskip
  
 \subsection{Tube maximal function}
This subsection contributes to the mapping property of the tube maximal function. To prove the boundedness result of the maximal function, we involve   the {\it iterated maximal operator}.

  Define  the {\it iterated maximal operator}  by
 \begin{equation*}
    M_{iterated}  (f)(\mathbf{x})=\sup_{\mathbf{r}\in \mathbb{R}^3_+} \left |f* \chi_{\mathbf{r} }(\mathbf{x} )\right|.
 \end{equation*}

 \begin{prop}
       Let $j=1,2,3$. For $ \varphi^{
(j)}\in L^1( \mathbb{R}^{ m})$    and $f\in L^1(\mathbb{R}^{2m}) $, let $\varphi_{\mathbf{r} }$ be given as in (\ref{normaldil}), then
    \begin{equation*}\begin{split}
  f*\varphi_{\mathbf{r} } =     f*_1\varphi_{r_1}^{
(1)} * _{ 2} \varphi_{r_2}^{
(2)}  { *} _3  \varphi_{r_3}^{
(3)}
,\end{split}\end{equation*}
 where the operators $*_1$ and $*_2$ are convolutions with respected to the the first and second variable and $*_3$ is the operation defined by
 $$
 f*_3 \varphi_{r_3}^{
(3)}(\mathbf{x}):=\int_{\mathbb{R}^m}f(x_1-u,\,x_2-u)\cdot\varphi_{r_3}^{
(3)}(u)\,du.
 $$
    \end{prop}
\begin{proof}
This proposition follows from the direct computation and the property that
    $$    \widehat{g*_3\varphi^{(3)}}(\xi_1,\,\xi_2)=\widehat{g}(\xi_1,\,\xi_2)\cdot\widehat{\varphi^{(3)}}(\xi_1+\xi_2),\quad\forall g\in L^1(\mathbb{R}^{2m}).
    $$
    The proof is complete.
\end{proof}      

The following proposition implies that  $\chi_{\mathbf{r} } $ is essentially the normalized characteristic function of the tube
 $  T(\mathbf{0},\mathbf{r})$.

  \begin{prop} \label{prop:tube} For $ f \in L^1(\mathbb{R}^{2m})$ and $\mathbf{r}  \in \mathbb{  R}^3_+ $, we have
\begin{equation}\label{mmmmm}
    c_m\cdot|f|* \chi_{\frac 12\mathbf{r} }(\mathbf{x} ) \leq \frac {1}{| T(\mathbf{x},\mathbf{r})|}\int_{  T(\mathbf{x},\mathbf{r})}|f (\mathbf{y})|\,d\mathbf{y}\leq
     C_m\cdot|f|* \chi_{2 \mathbf{r} }(\mathbf{x} ).
\end{equation}
As a consequence, $M_{tube}f(\mathbf{x}) \simeq M_{iterated}  (f)(\mathbf{x})$ for almost every point in $\mathbb{R}^{2m}$.
    \end{prop}
    \begin{proof}
    By Proposition \ref{prop:tube-pi}, for all $\mathbf{r}\in\mathbb{R}^3_{+}$ and $\mathbf{x}\in\mathbb{R}^{2m}$, $ \pi  (\tilde{{B}} (\mathbf{x} ,\mathbf{r}/2) )\subset T(\mathbf{x},\mathbf{r})\subset \pi  (\tilde{{B}} (\mathbf{x} ,2\mathbf{r}) ).$
        
Firstly, assume that $r_1,\,r_2\geq r_3$. Then $T(\mathbf{x,\,\mathbf{r}})=B_{r_1}(x_1)\times B_{r_2}(x_2)$ and hence the average over the tube $T(\mathbf{x,\,\mathbf{r}})$ is
        \begin{align*}
            \frac 1{| T(\mathbf{x},\mathbf{r})|}\int_{  T(\mathbf{x},\mathbf{r})}|f (\mathbf{y})|d\mathbf{y}\simeq\frac{1}{r^m_1r^m_2}\int_{  \mathbb{R}^{2m}}|f (\mathbf{x}-\mathbf{y})|\cdot\chi_{B_{r_1}({0}_1)\times B_{r_2}({0}_2)}(\mathbf{y})d\mathbf{y}.
        \end{align*}
        On the other hand, for any constant $c>0$, the convolution
        \begin{align*}
            |f|* \chi_{c\cdot\mathbf{r} }(\mathbf{x} ):=\int_{\mathbb{R}^{2m}}|f (\mathbf{x}-\mathbf{y})|\cdot\chi_{c\cdot\mathbf{r} }(\mathbf{y})\,d\mathbf{y},
        \end{align*}
        where the normalized ball characterization function is 
        \begin{align*}
        \chi_{c\cdot\mathbf{r} }(\mathbf{y})=\frac{c^{-3m}}{r^m_1\cdot r^m_2}\left(\frac{1}{r^m_3}\int_{\mathbb{R}^{m}}\chi_{B({0},cr_1)}(u-y_1)\cdot\chi_{B({0},cr_2)}(u-y_2)\cdot\chi_{B({0},cr_3)}(u)\,du\right).   
        \end{align*}
        Since the normalized ball characterization function can be dominated by $$
        \frac{c^{-3m}}{r^m_1\cdot r^m_2}\left(\frac{1}{r^m_3}\left|{B({0},cr_3)}\right|\right)\simeq\frac{c^{-2m}}{r^m_1\cdot r^m_2},
        $$
        and the support is contained in $T(\mathbf{x,\,\mathbf{r}})$ whenever $c=\frac{1}{2}$, then we have
        \begin{align*}
          |f|* \chi_{\frac{1}{2}\cdot\mathbf{r} }(\mathbf{x} )\leq c_m\cdot \frac{1}{r^m_1\cdot r^m_2}\int_{\mathbb{R}^{2m}}|f (\mathbf{x}-\mathbf{y})|\cdot\chi_{T(\mathbf{x,\,\mathbf{r}})}(\mathbf{y})\,d\mathbf{y} \simeq\frac 1{| T(\mathbf{x},\mathbf{r})|}\int_{  T(\mathbf{x},\mathbf{r})}|f (\mathbf{y})|d\mathbf{y}.
        \end{align*}
       Also remark that for all $\mathbf{y}\in T(\mathbf{0,\,\mathbf{r}})$, $B({0},r_3)\subset {B(y_1,2r_1)}\cap{B(y_2,2r_2)}\cap{B(0,2r_3)}$. Hence, for the case of $r_1,r_2\geq r_3$, one has (\ref{mmmmm}).

Secondly, suppose that $r_1,\,r_3\geq r_2$. Then $T(\mathbf{0},\mathbf{r}): =   P^{(first)}_{r_1,\,r_3 }$; and hence the average over the tube $T(\mathbf{x,\,\mathbf{r}})$ is
        \begin{align*}
            \frac 1{| T(\mathbf{x},\mathbf{r})|}\int_{  T(\mathbf{x},\mathbf{r})}|f (\mathbf{y})|d\mathbf{y}\simeq\frac{1}{r^m_1r^m_3}\int_{  \mathbb{R}^{2m}}|f (\mathbf{x}-\mathbf{y})|\cdot\chi_{ P^{(first)}_{r_1,\,r_3 }}(\mathbf{y})\,d\mathbf{y}.
        \end{align*}
        On the other hand, for any constant $c>0$, the convolution
        \begin{align*}
            |f|* \chi_{c\cdot\mathbf{r} }(\mathbf{x} ):=\int_{\mathbb{R}^{2m}}|f (\mathbf{x}-\mathbf{y})|\cdot\chi_{c\cdot\mathbf{r} }(\mathbf{y})\,d\mathbf{y},
        \end{align*}
        where the normalized ball characterization function is 
        \begin{align*}
        \chi_{c\cdot\mathbf{r} }(\mathbf{y})=\frac{c^{-3m}}{r^m_1\cdot r^m_3}\left(\frac{1}{r^m_2}\int_{\mathbb{R}^{m}}\chi_{B({0},cr_1)}(u-y_1)\cdot\chi_{B({y}_2,cr_2)}(u)\cdot\chi_{B({0},cr_3)}(u)\,du\right).   
        \end{align*}
        Since the normalized ball characterization can be dominated by $$
        \frac{c^{-3m}}{r^m_1\cdot r^m_3}\left(\frac{1}{r^m_2}\left|{B(y_2,cr_2)}\right|\right)\simeq\frac{c^{-2m}}{r^m_1\cdot r^m_3},
        $$
        and the support is contained in $P^{(first)}_{r_1,\,r_3}$ whenever $c=\frac{1}{2}$, then we have
        \begin{align*}
          |f|* \chi_{\frac{1}{2}\cdot\mathbf{r} }(\mathbf{x} )\leq c_m\cdot \frac{1}{r^m_1\cdot r^m_3}\int_{\mathbb{R}^{2m}}|f (\mathbf{x}-\mathbf{y})|\cdot\chi_{P^{(first)}_{r_1,\,r_3}}(\mathbf{y})\,d\mathbf{y} \simeq\frac 1{| T(\mathbf{x},\mathbf{r})|}\int_{  T(\mathbf{x},\mathbf{r})}|f (\mathbf{y})|\,d\mathbf{y}.
        \end{align*}
       Also remark that for all $(y_1,\,y_2)\in P^{(first)}_{r_1,\,r_3}$, $B(y_2,r_2)\subset {B(y_1,2r_1)}\cap{B(y_2,2r_2)}\cap{B(0,2r_3)}$. Hence, for the case of $r_1,r_3\geq r_2$, one has (\ref{mmmmm}). Similarly, the proposition holds for $r_2, r_3\geq r_1$ and the proof is completed.
    \end{proof}
    \smallskip
    \subsubsection{$L^p$-boundedness of the tube maximal function: proof of Theorem \ref{thm:maximal}} 
\begin{proof}
    By Proposition \ref{prop:tube}, it is suffices to show that the iterated maximal function is $L^p(\mathbb{R}^{2m})$ bounded for all $1<p<\infty.$ Note that for all $f\in C^\infty_0(\mathbb{R}^{2m})$
    \begin{align*}
    |f|*_{\mathbb{R}^{2m}}\chi_{\mathbf{r} }(\mathbf{x})&=|f|*_{\mathbb{R}^{2m}}\left(\left(\chi_{r_1}\otimes\chi_{r_2}\right)*_{3}\chi_{r_3}\right)(\mathbf{x})=\left(\chi_{r_1}\otimes\chi_{r_2}\right)*_{\mathbb{R}^{2m}}\left(|f|*_{3}\chi_{r_3}\right)(\mathbf{x}),
  \end{align*}
  and for all $\mathbf{x}\in\mathbb{R}^{2m}$
  \begin{align*}
 \left(|f|*_{3}\chi_{r_3}\right)(\mathbf{x})&:=\frac{1}{|B_{r_3}(\mathbf{0})|}\int_{B_{r_3}(\mathbf{0})}   |f|(\mathbf{x}-(u,u)))\,du  \\
 &=\frac{1}{|B_{r_3}(\mathbf{0})\times B_{r_3}(\mathbf{0})|}\int_{B_{r_3}(\mathbf{0})} \int_{B_{r_3}(\mathbf{0})}  |f|(\mathbf{x}-\mathbf{u})\cdot\chi_{\{u_1=u_2\}}(\mathbf{u})\,du_1\,du_2\\
 &\leq\frac{1}{|B_{r_3}(\mathbf{0})\times B_{r_3}(\mathbf{0})|}\int_{B_{r_3}(\mathbf{0})\times B_{r_3}(\mathbf{0})} |f|(\mathbf{x}-\mathbf{u})\,d\mathbf{u}.
  \end{align*}
  Then we have that the iterated maximal function is dominated by the  strong maximal function associated with the standard rectangle on $\mathbb{R}^{2m}$ compose with the Hardy--Littlewood maximal function on $\mathbb{R}^{2m}$, that is 
  $$
  M_{iterated}\leq M_s\circ M_{HL}.
  $$
  Therefore, the $L^p(\mathbb{R}^{2m})$-boundedness of the tube maximal function $M_{tube}$ follows by the mapping property of the strong maximal function $M_s$ and the Hardy--Littlewood maximal function $M_{HL}$.
\end{proof}

 \subsubsection{Calder\'on
reproducing formula: proof of Theorem \ref{prop:reproducing}}

\begin{proof}
    For each $j=1,2,3$ and $r_j>0$, let $\phi^{(j)}\in\mathcal{S}(\mathbb{R}^m)$ be non-negative such that the Fourier support of each $\phi^{(j)}$ is in $\{\xi_j\in\mathbb{R}^m:\,\frac{1}{2}\leq|\xi_j|\leq2\}$ and for all $\xi_j\in\mathbb{R}^m-\{\mathbf{0}\}$,
    \begin{align*}
        \int_{\mathbb{R}_+}\widehat{\phi^{(j)}}(\delta_{r_j}\xi_j)\,\frac{dr_j}{r_j}=1.
    \end{align*}
    Let $N_j\geq1$ and let $\psi^{(j)}\in C^\infty_0(\mathbb{R}^m)$ be the function which has higher order cancellation property up to order $2N_j-1$ such that $\widehat{\psi^{(j)}}$ doesn't vanish on $\{\xi_j\in\mathbb{R}^m:\,\frac{1}{2}\leq|\xi_j|\leq2\}$. Then for all $\xi_j\neq 0$
    \begin{align*}
      \int_{\mathbb{R}_+}\widehat{\psi^{(j)}}(\delta_{r_j}\xi_j)\cdot\widehat{\varphi^{(j)}}(\delta_{r_j}\xi_j)\,\frac{dr_j}{r_j}&:=  \int_{\mathbb{R}_+}\widehat{\psi^{(j)}}(\delta_{r_j}\xi_j)\cdot\left(\frac{\widehat{\phi^{(j)}}}{\widehat{\psi^{(j)}}}\right)(\delta_{r_j}\xi_j)\,\frac{dr_j}{r_j}
      =\int_{\mathbb{R}_+}\widehat{\phi^{(j)}}(\delta_{r_j}\xi_j)\,\frac{dr_j}{r_j}=1,
    \end{align*}
    where $\varphi^{(j)}\in \mathcal{S}(\mathbb{R}^m)$ is the function with the property that $\widehat{\varphi^{(j)}}=\widehat{\phi^{(j)}}/\widehat{\psi^{(j)}}$. Note that $\varphi^{(j)}$ is of mean zero for each $j$.
    
   By taking the Fourier transform of $\psi_{\mathbf{r} }$ and $\varphi_{\mathbf{r} }$ for each $\mathbf{r}\in\mathbb{R}^3_+$ we have
   \begin{align*}
       \widehat{\psi_{\mathbf{r} }}(\xi_1,\,\xi_2)=\widehat{\psi^{(1)}}(\delta_{r_1}\xi_1)\times\widehat{\psi^{(2)}}(\delta_{r_2}\xi_2)\times\widehat{\psi^{(3)}}(\delta_{r_3}(\xi_1+\xi_2)),
   \end{align*}
   and
    \begin{align*}
       \widehat{\varphi_{\mathbf{r} }}(\xi_1,\,\xi_2)=\widehat{\varphi^{(1)}}(\delta_{r_1}\xi_1)\times\widehat{\varphi^{(2)}}(\delta_{r_2}\xi_2)\times\widehat{\varphi^{(3)}}(\delta_{r_3}(\xi_1+\xi_2)),
   \end{align*}
   which implies that for all $\xi_1,\xi_2\neq\mathbf{0}$ and $\xi_1+\xi_2\neq\mathbf{0}$
   \begin{align*}
       \int_{\mathbb{R}^3_+}\widehat{\psi_\mathbf{r}}(\xi_1,\,\xi_2)\times\widehat{\varphi_{\mathbf{r} }}(\xi_1,\,\xi_2)\,\frac { d\mathbf{r}}{\mathbf{r}}&=\prod^2_{j=1}\int_{\mathbb{R}_+}\widehat{\psi^{(j)}}(\delta_{r_j}\xi_j )\times\widehat{\varphi^{(j)}}(\delta_{r_j}\xi_j)\,\frac{dr_j}{r_j}\times\int_{\mathbb{R}_+}\left(\widehat{\psi^{(3)}}\times\widehat{\varphi^{(3)}}\right)(\delta_{r_3}(\xi_1+\xi_2))\frac{dr_3}{r_3}\\
       &=1.
   \end{align*}
   Therefore, for all $f\in L^1(\mathbb{R}^{2m})\cap L^2(\mathbb{R}^{2m})$
   \begin{align*}
       f(\mathbf{x}) =
\int_{\mathbb{R}^3_+} f*\psi_ {\mathbf{r}}*\varphi_ {\mathbf{r}}(\mathbf{x})\, \frac { d\mathbf{r}}{\mathbf{r}},\,\,\text{where}\,\, \frac {d\mathbf{r}}{\mathbf{r}}=\frac
{dr_1}{r_1}\frac {dr_2}{r_2}\frac {dr_3}{r_3}.
   \end{align*}
   For the existence of the function $\psi^{(j)}$. Consider a radial function $g\in C^\infty_0(\mathbb{R}^    m)$ supported on $\{x_j\in\mathbb{R}^m:\,|x_j|<2\}$ with $\widehat{g}(0)=1$ and let $\psi^{(j)}(\cdot):=\lambda^{m}\left(\triangle^{N_j}_jg\right)(\delta_\lambda(\cdot))$ for some large $\lambda>0$, in which $\triangle_j$ denotes the standard Laplacian in $\mathbb{R}^m$.
\end{proof}

\subsubsection{Laplacian on the twisted setting}
 Write $  \mathbf{x}_j=({x}^1_j,\ldots,{x}^m_j)\in\mathbb{R}^{ m}$.
Let $j=1,2$, and define that 
\begin{equation*}
    \triangle_{j}:=  - \sum_{k=1}^{m} \frac{ \partial^2}{\partial (x_{j}^{k})^2}\quad  \text{and}\quad{ \triangle}_{twist}:=  - \sum_{k=1}^{m} \left(\frac \partial {\partial x^k_{1} }+\frac  
    \partial {\partial x^k_{2} }\right)^2.
\end{equation*}

\section{Dyadic rectangles/ tubes on $\mathbb{R}^{2m}$ and the new covering lemma}\label{sec:4}
In Section \ref{sec:3}, we construct the system of tubes on $\mathbb{R}^{2m}$. To develop the twisted covering lemma as well as the atomic decomposition, we need to further classify the tubes into several types (see also Figure 1 (a) and 1 (b)).
\subsection{Miscellaneous types of dyadic rectangles/ tubes}
    There are five types of  dyadic rectangles/tubes.
    
   $\bullet$ A
       {\it type ${\rm I}$  dyadic  rectangle} is the   translate   of the standard  dyadic  rectangle  $[0,2^{j_1})^m\times[0,2^{j_2})^m $ under an element of the lattice
 $
   2^{  {j}_1  } \mathbb{Z}^{m } \times  2^{ j_2   }   \mathbb{Z}^m
$; 

$\bullet$ A   {\it type ${\rm I\!I}$  dyadic  rectangle}  is the  translate    of  the  slant  dyadic  rectangle (based on the first direction)
              \begin{equation}\label{eq:slant1}
     I\times_t J:=    \left\{ \mathbf{x} \in  \mathbb{R}^{2m}:\, \mathbf{x}_1-\mathbf{x}_2\in I,\mathbf{x}_2\in J\right\} ,\quad  
     \end{equation} 
     with $$I=[0,2^{j_1})^m, J=[0,2^{j_2})^m {\rm\ \ \  and\ \ \ }
     j_1\leq j_2
     $$
       under an element of the lattice
  \begin{equation*}
   2^{  {j}_1  } \mathbb{Z}^{m } \times_t  2^{ j_2   }   \mathbb{Z}^m:=\{(\mathbf{n}_1+\mathbf{n}_2, \mathbf{n}_2):\mathbf{n}_1 \in 2^{  {j}_1  } \mathbb{Z}^{m }, \,  \mathbf{n}_2\in  2^{ j_2   }   \mathbb{Z}^m \};    
 \end{equation*}

$\bullet$  A {\it type ${\rm I\!I\!I}$  dyadic  rectangle} is defined to be the same structure of the {\it type ${\rm I\!I}$  dyadic  rectangle} but with the constraint that $j_1>j_2.$

$\bullet$  A {\it type ${\rm I\!V}$  dyadic  rectangle} is the  translate    of  the  slant  dyadic  rectangle (based on the second direction)
              \begin{equation}\label{eq:slant11}
     I\,{}_t\times J:=    \left\{ \mathbf{x} \in  \mathbb{R}^{2m}:\,  \mathbf{x}_1 \in I, \mathbf{x}_2-\mathbf{x}_1 \in J\right\}   
     \end{equation} 
     with $$I=[0,2^{j_1})^m, J=[0,2^{j_2})^m {\rm \ \ \ and \ \ \ }
     j_1\leq j_2
     $$
       under an element of the lattice
  \begin{equation*}
   2^{  {j}_1  } \mathbb{Z}^{m } \,{}_t\times  2^{ j_2   }   \mathbb{Z}^m:=\{(\mathbf{n}_1, \mathbf{n}_1+\mathbf{n}_2):\,\mathbf{n}_1 \in 2^{  {j}_1  } \mathbb{Z}^{m },\,  \mathbf{n}_2\in  2^{ j_2   }   \mathbb{Z}^m \}.    
 \end{equation*} 

 $\bullet$  A {\it type ${\rm V}$  dyadic  rectangle} is defined to be the same structure of the {\it type ${\rm I\!V}$  dyadic  rectangle} but with the constraint that $j_1>j_2.$

  \begin{rem} A {\it  dyadic  rectangle   of scale $\mathbf{j}\in \mathbb{Z}^3$} is a type ${\rm I}$ rectangle as a translate of $[0,2^{j_1})^m\times[0,2^{j_2})^m $ if $j_1, j_2\geq  j_3  $; it is a type ${\rm I\!I}$ / ${\rm I\!I\!I}$ rectangle as a translate of
    $[0,2^{j_1})^m\times_t [0,2^{j_3})^m $ if    $ \ j_1 , j_3 >  j_2$;  it is a type ${\rm I\!V}$ / ${\rm V}$ rectangle as a translate of $[0,2^{j_2})^m_t\times[0,2^{j_3})^m $ if 
         $ \ j_2 , j_3 >  j_1$. Now denote by
  $ \mathfrak{R }_{\mathbf{j}}  
$ the set of  rectangles of scale $\mathbf{j} $
 and by $
   \mathfrak{R }:=\bigcup_{\mathbf{j}\in \mathbb{Z} ^3}\mathfrak{R }_{\mathbf{j}}.
$
  the set of all rectangles.  They constitute a partition $  \mathbb{R}^{2m} $ for each $\mathbf{j}$.
  \end{rem}

Regarding the five different types of rectangles/tubes, we consider the maximal ones in $\Omega$ and the related maximal functions.  
\begin{itemize}
\item \textbf{Standard maximal rectangles}:
For an open set $\Omega$ in $ \mathbb{R}^{2m}$ with finite measure, let $m^{\rm I}_1(\Omega)$ be the class of all dyadic tubes of type ${\rm I}$ (the standard rectangles), $I\times J\subset\Omega$, which are maximal in the direction along the $x$-axis. Dyadic tubes of in $m^{\rm I}_2(\Omega)$ are maximal along the $y$-axis.

    \item 
\textbf{Maximal tubes of type ${\rm I\!I}$}:
For an open set $\Omega$ in $ \mathbb{R}^{2m}$ with finite measure, let $m^{\rm I\!I}_1(\Omega)$ be the class of all dyadic tubes of type ${\rm I\!I}$, $I\times_t J\subset\Omega$, which are maximal in the direction along $x$-axis. Dyadic tubes of type ${\rm I\!I}$ in $m^{\rm I\!I}_2(\Omega)$ are maximal in the direction of $(1,1)$. Consider the tube maximal function with based on the first direction which is defined by
\begin{align*}
    M^{\rm I\!I}_{tube}(f)(\mathbf{x}):=\sup_{\mathbf{x}\in I\times_t J}\frac{1}{|I\times_t J|}\int_{I\times_t J}|f(\mathbf{y})|\,d\mathbf{y},
\end{align*}
for suitable function $f$. Note that the tube maximal function with based on the first direction is dominated by the tube maximal function $M_{tube}$, therefore $ M^{\rm I\!I}_{tube}$ is an $L^p(\mathbb{R}^{2m})$ bounded operator as well.

    \item 
\textbf{Maximal tubes of type ${\rm I\!I\!I}$, ${\rm I\!V}$ and ${\rm V}$}:
similarly, we can define $m^{j}_1(\Omega)$, $m^{j}_2(\Omega)$ and
$M^{j}_{tube}(f)(\mathbf{x})$ for $j= {\rm I\!I\!I}, {\rm I\!V}, {\rm V}$. 
\end{itemize}

\subsection{Covering Lemma adapted to the miscellaneous dyadic rectangles}
We now consider the covering lemmas. The first one is the standard one due to Journ\'e and Pipher \cite{J,P}.
\begin{lem}(Standard covering lemma)\label{lem:Journe standard}~\\
    Let $\Omega$
be an open subset of $ \mathbb{R}^{2m}$ with finite measure and $\kappa>0$. Then for each $I\times J\in m^{\rm I}_2(\Omega)$, there is an $\widehat{I}$ containing $I$ such that

\begin{align}
 \sum_{R= I \times J\in m^{\rm I}_2(\Omega)} |R|  \left(\frac {\ell (I)}{\ell (\widehat{I} )}\right)^\kappa \leq    C     |\Omega|  ,
 \end{align}
 for some constant $C$ independent of $\Omega$. Likewise, for each $I\times J\in m^{I}_1(\Omega)$, there is an $\widehat{J}$ containing $J$ such that
 \begin{align}
\sum_{R= I \times J\in m^{\rm I}_1(\Omega)} |R| \left(\frac {\ell (J)}{\ell (\widehat{J} )}\right)^\kappa \leq    C     |\Omega|  ,
 \end{align}
\end{lem}
\smallskip

\subsubsection{Twisted covering lemma}
The following two covering lemmas are new, which allows us to cover an open set with the rectangles with the same type. We highlight that the {\it type-keeping feature} in these covering lemmas is needed when verifying the atomic decomposition (\ref{second}).
\smallskip
\begin{lem}[Twisted covering lemma of tubes on type ${\rm I\!I}$ and ${\rm I\!I\!I}$]~\label{lem:Journe 1}\\
    Let $\Omega$
be an open subset of $ \mathbb{R}^{2m}$ with finite measure and $\kappa>0$. Then for each $I\times_t J\in m^{\rm I\!I}_2(\Omega)$, there is an $\widehat{I}$ containing $I$ such that

\begin{align}\label{Journe 1 I hat}
 \sum_{R= I \times_t J\in m^{\rm I\!I}_2(\Omega)} |R|  \left(\frac {\ell (I)}{\ell (\widehat{I} )}\right)^\kappa \leq    C     |\Omega|  ,
 \end{align}
 for some constant $C$ independent of $\Omega$. Likewise, for each $I\times_t J\in m^{\rm I\!I}_1(\Omega)$, there is an $\widehat{J}$ containing $J$ such that
 \begin{align}\label{Journe 1 J hat}
\sum_{R= I \times_t J\in m^{\rm I\!I}_1(\Omega)} |R| \left(\frac {\ell (J)}{\ell (\widehat{J} )}\right)^\kappa \leq    C     |\Omega|  ,
 \end{align}

 We also have the similar version for tubes in type ${\rm I\!I\!I}$.
\end{lem}

\begin{lem}[Twisted covering lemma of tubes on type ${\rm I\!V}$ and ${\rm V}$]\label{lem:Journe 2}~\\
    Let $\Omega$
be an open subset of $ \mathbb{R}^{2m}$ with finite measure and $\kappa>0$. Then for each $I \,_t\times J\in m^{\rm I\!V}_2(\Omega)$, there is an $\widehat{I}$ containing $I$ such that

\begin{align}
 \sum_{R= I\,_t \times J\in m^{\rm I\!V}_2(\Omega)} |R| \cdot \left(\frac {\ell (I)}{\ell (\widehat{I} )}\right)^\kappa \leq    C     |\Omega|  ,
 \end{align}
 for some constant $C$ independent of $\Omega$. Likewise, for each $I\,_t\times J\in m^{\rm I\!V}_1(\Omega)$, there is an $\widehat{J}$ containing $J$ such that
 \begin{align}
\sum_{R= I\,_t \times J\in m^{\rm I\!V}_1(\Omega)} |R|\cdot\left(\frac {\ell (J)}{\ell (\widehat{J} )}\right)^\kappa \leq    C     |\Omega|.
 \end{align}

  We also have the similar version for tubes in type ${\rm V}$.

\end{lem}

To end this section, we demonstrate the proof of Lemma \ref{lem:Journe 1}. Lemma \ref{lem:Journe 2} can be obtained by using the similar geometric argument. In the twisted case, there would be {\it oblique flag rectangles} that come into the structure (see \cite{CCLLO} for the oblique flag rectangle structures).
\begin{proof}
In the type of ${\rm I\!I}$, the tubes have the property that  
$|I|\leq |J|.$ From the standard Journ\'e's covering lemma, it suffices to prove \eqref{Journe 1 I hat}. 

For any dyadic interval $I$, we define $$E_{I}(\Omega)=\bigcup\{\mathfrak Q \subseteq\mathbb{R}^{m}:\ I\times_t \mathfrak Q\ \text{ is a {\it slant flag dyadic rectangle}, } I\times_t \mathfrak Q\subseteq\Omega\},$$
where the {\it slant flag dyadic rectangle} $I\times_t \mathfrak Q$ means that the volume of the dyadic cube $I$ is dominated by the volume of the dyadic cube $\mathfrak Q$. For simplicity, we may denote $E_I(\Omega)$ by $E_I$. For all $k\in\mathbb{N}$, we also define that 
$$A_{I,k}=\bigcup\{\mathfrak Q\subseteq\mathbb{R}^m:
I\times_t \mathfrak Q\in m^{\rm I\!I}_{2}(\Omega),\,\hat{I}=(I)_{k-1}\},$$
where $\hat{I}$ is the largest dyadic cube containing $I$ such that $|(\hat{I}\times_t R)\cap\Omega|>\frac{1}{2}|\hat{I}\times_t \mathfrak Q|$ and $(I)_{k-1}=2^{k-1}I$.

We claim that 
if $I$ is a dyadic interval, then for all $k\in\mathbb{N}$
\begin{align}\label{claim AIk}
    A_{I,k}\subseteq\{x:M\chi_{E_{I}\setminus E_{(I)_{k}}}(x)>\frac{1}{2}\}.
    \end{align}
To acquire this claim, we observe that for any $x\in A_{I,k}$, there exists a dyadic cube $\mathfrak Q$ satisfying $I\times_t \mathfrak Q\in m_{2}^{\rm I\!I}(\Omega),\hat{I}=(I)_{k-1}$ and $x\in \mathfrak Q$. We establish $\mathfrak Q\in E_{I}$ and $I\times_t E_{I}\subseteq\Omega$ from the definition of $E_{I}$; besides, by the definition of $\hat{I}=(I)_{k-1}$, we have
\[
\left|\left((I)_{k}\times_t \mathfrak Q\right)\cap\left( (I)_{k}\times_t E_{(I)_{k}}\right)\right|<\left|\left((I)_{k}\times_t \mathfrak Q\right)\cap\Omega\right|\leq\frac{1}{2}|(I)_{k}\times_t \mathfrak Q|.
\] This means that $|\mathfrak Q\cap E_{(I)_{k}}|<\frac{1}{2}|\mathfrak Q|$ and so $|\mathfrak Q\cap(E_{I}\setminus E_{(I)_{k}})|>\frac{1}{2}|\mathfrak Q|$, which implies 
that 
\eqref{claim AIk} holds.   

Next we argue that 
\begin{align}\label{claim sum EI without E2I}
\sum_{\substack{I\subset\mathbb{R}^m,\\\text{dyadic}}}|E_{I}\setminus E_{2I}|\cdot|I|\simeq|\Omega|.
\end{align}

For any $x\in\Omega$, there exists a dyadic slant of tube $I\times_t \mathfrak Q\subseteq\Omega$ such that $x\in I\times_t \mathfrak Q$ since $\Omega$ is open, and thus $x\in I\times_t E_{I}$. In light of the inclusion that $E_{I}\supseteq E_{2I}\supseteq\cdots\supseteq E_{(I)_{k}}$, there must be a largest $k\in\mathbb{N}$ such that $x\in I\times_t E_{(I)_{k}}$, which leads to $x\in (I)_{k}\times_t E_{(I)_{k}}$.

From the definition of $E_{I}$, one has $\bigcup(I\times_t(E_{I}\setminus E_{2I}))\subseteq\Omega$, and therefore $\Omega=\bigcup(I\times_t(E_{I}\setminus E_{2I}))$. Remark that if $I$ and $J$ are disjoint, then $I\times_t(E_{I}\setminus E_{2I})$ and $J\times_t(E_{J}\setminus E_{2J})$ are also disjoint; on the other hand, if $I$ and $J$ are not disjoint, then there must exist a $k\geq 0$ such that $J=(I)_{k}$ (assume that $|I|\leq|J|$). If $k=0$, then $I=J$. If $k\geq 1$, then $E_{J}=E_{(I)_{k}}\subseteq E_{(I)_{1}}$, so $E_{I}\setminus E_{2I}$ and $E_{J}\setminus E_{2J}$ are disjoint. Hence, \eqref{claim sum EI without E2I} is established.

Furthermore, there exists a constant $C$ depending only on $\kappa$ such that for any open bounded $\Omega\subseteq\mathbb{R}^{m}$,
\begin{align}\label{claim EI sum}
    \sum_{\substack{I\subset\mathbb{R}^m,\\\text{dyadic}}}|I|\left[\sum^\infty_{k=1}\left(\frac{|I|}{|(I)_{k}|}\right)^\kappa|E_{I}\setminus E_{(I)_{k}}|\right]\leq C|\Omega|,
\end{align}
for the argument of ($\ref{claim EI sum}$), see, for example, Proposition $B$ in \cite{P}.

We now split the dyadic tubes $R=I\times_t J\in m^{\rm I\!I}_2(\Omega)$ into two cases. Let $k_0$ be the non-negative integer such that $2^{k_0}|I|=|J|$.

$\bullet$ Case $1$:  $R=I\times_t J\in m^{\rm I\!I}_2(\Omega)$ and
$
|(I)_{k_0}\times_t J\cap\Omega|\leq\frac{1}{2}|(I)_{k_0}\times_t J|.$

We can enlarge $I$ to $\hat{I}=\hat{I}(J)$, which is the longest dyadic interval containing $I$ such that
\begin{align*}
 |\hat{I}|<|J|\quad\text{and}\quad   |\hat{I}\times_t J\cap\Omega|>\frac{1}{2}|\hat{I}\times_t J|.
\end{align*}
It follows that
\[
\Bigg|\bigcup_{I\times_t J\in m^{\rm I\!I}_2(\Omega),\text{ Case 1}}\hat{I}\times_t J\Bigg|\leq \Bigg|\bigcup_{I\times J\in m^{\rm I\!I}_2(\Omega)}\hat{I}\times_t J\Bigg|\leq|\widetilde{\Omega}|\leq  C|\Omega|,
\]
where $\widetilde{\Omega}=\{x:M_{tube}(\chi_{\Omega})(x)>\frac{1}{2}\}$ denotes the enlargement of the open set $\Omega$. Then 
\begin{align*}
\sum_{\substack{R\in m^{\rm I\!I}_2(\Omega),\\\text{ Case 1}}}|R|\cdot\left(\frac{|I|}{|\hat{I}|}\right)^\kappa 
&\leq \sum_{\substack{I\subset\mathbb{R}^m,\\\text{dyadic}}}\sum_{k\geq1}\sum_{\substack{J:I\times_t J\in m^{\rm I\!I}_2(\Omega),\\\hat{I}=(I)_{k-1},\\|\hat{I}|\le|J|}}|I||J|\left(\frac{|I|}{|(I)_{k-1}|}\right)^\kappa \\
&= \sum_{\substack{I\subset\mathbb{R}^m,\\\text{dyadic}}}|I|\sum_{k\geq1}2^{-k\kappa}\sum_{\substack{J:I\times_t J\in m^{\rm I\!I}_2(\Omega),\\\hat{I}=(I)_{k-1},\\|\hat{I}|\le|J|}}|J|.
\end{align*}
Since all the $J$'s in $\{J:I\times_t J\in m^{\rm I\!I}_2(\Omega),\hat{I}=(I)_{k-1}\}$ are disjoint,
we have $$\sum_{\substack{J:I\times_t J\in m^{\rm I\!I}_2(\Omega),\\\hat{I}=(I)_{k-1},\\|\hat{I}|\le|J|}}|J|=|A_{I,k},|$$ 
which leads to
\begin{align*}
\sum_{\substack{R\in m^{\rm I\!I}_2(\Omega),\\\text{ Case 1}}}|R|\left(\frac{|I|}{|\hat{I}|}\right)^\kappa 
&\leq \sum_{\substack{I\subset\mathbb{R}^m,\\\text{dyadic}}}|I|\sum_{k\geq1}2^{-k\kappa}|A_{I,k}| 
\leq C\sum_{\substack{I\subset\mathbb{R}^m,\\\text{dyadic}}}|I|\sum_{k\geq1}2^{-k\kappa}|E_{I}\setminus E_{(I)_{k}}| 
\leq C|\Omega|.
\end{align*}

$\bullet$ Case $2$:  $R=I\times_t J\in m^{\rm I\!I}_2(\Omega)$ and
$
|(I)_{k_0}\times_t J\cap\Omega|>\frac{1}{2}|(I)_{k_0}\times_t J|.$

We will further enlarge 
$(I)_{k_0}\times_t J$ to
$$  (I)_{k_0+\ell}\times_t (J)_{\ell},\quad\ell=1,2,\ldots $$
and let $\hat{I}=(I)_{k_0+l}$ with the largest $l$ such that
$$|(I)_{k_0+l}\times_t(J)_l\cap\Omega|>\frac{1}{2}|(I)_{k_0+l}\times_t(J)_l|.$$
It follows that
\[
\Bigg|\bigcup_{\substack{I\times_t J\in m^{\rm I\!I}_2(\Omega),\\\text{ Case 2}}}\hat{I}\times_t J\Bigg|
\leq 
\Bigg|\bigcup_{\substack{I\times_t J\in m^{\rm I\!I}_2(\Omega),\\\text{ Case 2}}}(I)_{k_0+l}\times_t (J)_l\Bigg|\leq|\widetilde{\Omega}|
\leq C|\Omega|
\]
and
\begin{align*}
    \sum_{\substack{R\in m^{\rm I\!I}_2(\Omega),\\\text{ Case 2}}}|R|\left(\frac{|I|}{|\hat{I}|}\right)^\kappa &=\sum_{\substack{I:\exists J,\\  I\times_t J\subset\Omega}} \sum_{\substack{J:I\times_t J\in m^{\rm I\!I}_2(\Omega),\\\text{ Case 2}}}\left(\frac{|I|}{|\hat{I}|}\right)^\kappa |I||J|\\
    &=\sum_{\substack{I:\exists J,\\  I\times_t J\subset\Omega}} \sum_{k_0\ge 1}\sum_{l\ge 1} \sum_{\substack{J:I\times_t J\in m^{\rm I\!I}_2(\Omega),\\ |J|=2^{k_0}|I|,\\ \hat{I}(J)=(I)_{k_0+l}}}\left(\frac{|I|}{|\hat{I}|}\right)^\kappa |I||J|\\
    &=\sum_{k_0\ge 1}\sum_{l\ge 1}  2^{-(k_o+l)\kappa}\sum_{\substack{I:\exists J,\\  I\times_t J\subset\Omega}} \sum_{k_0\ge 1}\sum_{\substack{J:I\times_t J\in m^{\rm I\!I}_2(\Omega),\\ |J|=2^{k_0}|I|,\\ \hat{I}(J)=(I)_{k_0+l}}}|I||J|.
\end{align*}

Let $m(\tilde\Omega)$ be the set of the largest slant dyadic cubes contained in $\tilde\Omega$. If $\hat{I}(J)=(I)_{k_0+l}$, then $(I)_{k_0}\times_t J\subset\tilde\Omega$ is a dyadic tube and $((I)_{k_0}\times_t J)_l\in m(\tilde\Omega)$ by the defining condition of $\hat{I}$. Then
$$\sum_{\substack{J:I\times_t J\in m^{\rm I\!I}_2(\Omega),\\ |J|=2^{k_0}|I|,\\ \hat{I}(J)=(I)_{k_0+l}}}|I||J|\lesssim\sum_{Q\in m(\tilde{\Omega})}\sum_{\substack{I\times_t J:\\ Q=((I)_{k_0}\times_t J)_l}}|I\times J|.$$

For any slant dyadic cube $Q\subset\mathbb{R}^{2m}$ and $n_1,n_2\ge 0$, since there are $2^{n_1+n_2}$ dyadic rectangles $I\times_t J$ satisfying $Q=(I)_{n_1}\times_t (J)_{n_2}$, and all the rectangles have the same volume $|I\times_t J|=2^{-(n_1+n_2)}|Q|$, we then have
$$|Q|=\sum_{Q=(I)_{n_1}\times_t (J)_{n_2}}|I\times J|,$$
which implies that
\begin{align*}
    \sum_{\substack{R\in m^{\rm I\!I}_2(\Omega),\\\text{ Case 2}}}|R|\left(\frac{|I|}{|\hat{I}|}\right)^\kappa 
    &\lesssim\sum_{k_0\ge 1}\sum_{l\ge 1}  2^{-(k_o+l)\kappa}
    \sum_{Q\in m(\tilde\Omega)}
    \sum_{\substack{I,J:\\Q=(I)_{k_0+l}\times_t J_l}}|I||J|\\
    &\le\sum_{k_0\ge 1}\sum_{l\ge 1}  2^{-(k_o+l)\kappa}
    \sum_{Q\in m(\tilde\Omega)}|Q|\\
    &\le \sum_{k_0\ge 1}\sum_{l\ge 1}  2^{-(k_o+l)\kappa}\cdot C|\Omega|,
\end{align*}
in which the last inequality follows from the fact that
\[\sum_{Q\in m(\tilde{\Omega})}|Q|\leq|\tilde{\Omega}|\le C|\Omega|.\] Therefore, 
$$   \sum_{\substack{R\in m^{\rm I\!I}_2(\Omega),\\\text{ Case 2}}}|R|\left(\frac{|I|}{|\hat{I}|}\right)^\kappa \le C|\Omega|.$$
Combining these two cases above, we have \eqref{Journe 1 I hat} and the proof is complete.
\end{proof}

\section{Hardy  spaces: Littlewood--Paley theory and atomic decomposition}\label{sec:5}
For the convenience of the readers, we recall  the defintion of the {\it Littlewood--Paley function} of $f$
\begin{align*}\label{littlewoodpaley}
    g_{\varphi}(f)(\mathbf{x}):=\left(\int_{\mathbb{R}^3_+} \left|f*\varphi_ {\mathbf{r}}(\mathbf{x})\right|^2\, \frac { d\mathbf{r}}{\mathbf{r}}\right)^{\frac{1}{2}},\,\,\text{where}\,\, \frac {d\mathbf{r}}{\mathbf{r}}=\frac
{dr_1}{r_1}\frac {dr_2}{r_2}\frac {dr_3}{r_3};
\end{align*}
and $\varphi_j's$ are Schwartz functions on $\mathbb{R}^m$ which has mean value zero for $j=1,2,3.$
\subsubsection{$L^p$-equivalence of the Littlewood--Paley function: proof of Theorem \ref{6999}}

\begin{proof}
Let $f\in L^p(\mathbb{R}^{2m})$. We first show that the $L^p$-norm of the  {\it Littlewood--Paley function} of $f$ is dominated by the $L^p$-norm of the function $f$.  Consider the function $\mathscr{H}$ defined on $\mathbb{R}^{2m}$ which is given by
\begin{align*}
   \mathscr{H}(\mathbf{x}):=f*\left(\varphi^{(1)}_{r_1}\otimes\varphi^{(2)}_{r_2}\right)(\mathbf{x})
\end{align*}
and the $L^2$-space $L^2(\mathbb{R}^{2}_+,\,\frac{dr_1dr_2}{r_1r_2})$. It is direct to see that for almost all $\mathbf{x}\in\mathbb{R}^{2m}$, $\mathscr{H}(\mathbf{x})\in L^2(\mathbb{R}^{2}_+,\,\frac{dr_1dr_2}{r_1r_2})$ and
\begin{align*}
    \|g_{\varphi}(f)\|^p_{L^p(\mathbb{R}^{2m})}&:=\int_{\mathbb{R}^{2m}}\left(\int_{\mathbb{R}^{3}_+}\left|f*\left(\varphi^{(1)}_{r_1}\otimes\varphi^{(2)}_{r_2}*_3\varphi^{(3)}_{r_3}\right)(\mathbf{x})\right|
    ^2\frac { d\mathbf{r}}{\mathbf{r}}\right)^{\frac{p}{2}}\,d\mathbf{x}\\
    &=\int_{\mathbb{R}^{2m}}\left(\int_{\mathbb{R}_+}\left\|\left(\mathscr{H}*_3\varphi^{(3)}_{r_3}\right)(\mathbf{x})\right\|_{L^2(\mathbb{R}^{2}_+,\,\frac{dr_1dr_2}{r_1r_2})}
    ^2\frac { dr_3}{r_3}\right)^{\frac{p}{2}}\,d\mathbf{x}.
\end{align*}
By the vector-valued Littlewood--Paley inequality, we have 
\begin{align*}
    \|g_{\varphi}(f)\|^p_{L^p(\mathbb{R}^{2m})}\lesssim\int_{\mathbb{R}^{2m}}\left\|\mathscr{H}(\mathbf{x})\right\|_{L^2(\mathbb{R}^{2}_+,\,\frac{dr_1dr_2}{r_1r_2})}
    ^p\,d\mathbf{x},
\end{align*}
which leads to $$ \|g_{\varphi}(f)\|^p_{L^p(\mathbb{R}^{2m})}\lesssim  \|f\|^p_{L^p(\mathbb{R}^{2m})}$$ if one applies the Littlewood--Paley inequality to $\mathcal{H}$. Now, we show the other direction by duality. Let $\|h\|_{L^q(\mathbb{R}^{2m})}=1$, then by the reproducing formula
\begin{align*}
    \langle f,\,h\rangle&=\int_{\mathbb{R}^3_+}\langle f*\psi_ {\mathbf{r}}*\varphi_ {\mathbf{r}},\,h\rangle\, \frac { d\mathbf{r}}{\mathbf{r}}=\int_{\mathbb{R}^{2m}\times\mathbb{R}^3_+}f*\psi_ {\mathbf{r}}*\varphi_ {\mathbf{r}}(\mathbf{x})\cdot\ h(\mathbf{x})\, \frac { d\mathbf{r}}{\mathbf{r}}\,d\mathbf{x}=\int_{\mathbb{R}^{2m}\times\mathbb{R}^3_+}f*\varphi_ {\mathbf{r}}(\mathbf{y})\cdot\ h*\widetilde{\psi_ {\mathbf{r}}}(\mathbf{y})\, \frac { d\mathbf{r}}{\mathbf{r}}\,d\mathbf{y},
\end{align*}
where $\widetilde{\psi_ {\mathbf{r}}}(\cdot):=\psi_ {\mathbf{r}}(-\cdot)$. By the Cauchy--Schwartz and the H\"older inequality respectively, we have
\begin{align*}
\left|\langle f,\,h\rangle\right|&\leq\|g_{\varphi}(f)\|_{L^p(\mathbb{R}^{2m})}\cdot\|g_{\widetilde{\psi}}(h)\|_{L^q(\mathbb{R}^{2m})}\lesssim\|g_{\varphi}(f)\|_{L^p(\mathbb{R}^{2m})},
\end{align*}
where the last inequality follows from the $L^p$-boundedness of the {\it Littlewood--Paley function} of $g$. So, the theorem is completed by the duality of the $L^p(\mathbb{R}^{2m})$ space.
\end{proof}
\smallskip

\subsubsection{Non-tangential region, Lusin area function and proof of Theorem \ref{thm S g equiv}}
 Next, we recall the non-tangential region and the Lusin area function adapted to the twisted setting.
 For $\mathbf{x} \in \mathbb{R}^{2m}$,
   the {\it nontangential region} $\Gamma (\mathbf{x})$ is defined by
  $\{(\mathbf{x}',\mathbf{r})\in \mathbb{R}^{2m} \times\mathbb{R}_+^3:\,\mathbf{x}'\in T(\mathbf{x}, \mathbf{r}) \}
 $ and 
  the {\it Lusin-Littlewood--Paley area function} is given by
\begin{align*}
   S_{area,  \varphi}(f)(\mathbf{x})=\left( \int_{  \mathbb{R}^3_+}|f* \varphi_{\mathbf{r} } |^2* \chi_{\mathbf{r} }(\mathbf{x}) \frac {  
    d\mathbf{r}}{\mathbf{r}}\right)^{\frac 12},
\end{align*}
for  $f \in L
^1
(\mathbb{R}^{2m})$.

We then have the Theorem \ref{thm S g equiv}. Again, for the convenience of the readers, we recall it here. 
\begin{thm}
    For $1 < p < \infty$ and for every $f \in L^p(\mathbb{R}^{2m})$,
    \begin{equation}
        \|S_{area, \varphi}(f)\|_{L^p(\mathbb{R}^{2m})} \simeq\|g_{\varphi}(f)\|_{L^p(\mathbb{R}^{2m})} \simeq \|f\|_{L^p(\mathbb{R}^{2m})},
    \end{equation}
    where the implicit constants depend on $p$ and $m$ only but not on $f$. Moreover, for $p=1$ and for $f \in L^1(\mathbb{R}^{2m})$, if $\|g_{\varphi}(f)\|_{L^1(\mathbb{R}^{2m})} < \infty$, then 
    $\|S_{area,\varphi}(f)\|_{L^1(\mathbb{R}^{2m})} \lesssim \|g_{\varphi}(f)\|_{L^1(\mathbb{R}^{2m})}$; conversely, if $\|S_{area, \varphi}(f)\|_{L^1(\mathbb{R}^{2m})} < \infty$, then 
    $\|g_{\varphi}(f)\|_{L^1(\mathbb{R}^{2m})} \lesssim \|S_{area,\varphi}(f)\|_{L^1(\mathbb{R}^{2m})}$, where the implicit constants depend on $m$ only but not on $f$.
\end{thm}

\begin{proof}
    The proof of these equivalences relies on well-established techniques in multiparameter Littlewood-Paley theory. The argument proceeds by employing the continuous and discrete Calderón reproducing formulas to decompose the functions, followed by the application of almost orthogonality estimates. The comparison between the continuous area integral $S_{area, \varphi}$ and the discrete square function $g_{\varphi}$ is achieved via Plancherel--Polya inequalities adapted to this setting.
    
    This general framework originated in the work of \cite{H,HS}. For the specific extensions to multiparameter product and flag structures relevant to our context, we refer the reader to \cite{HLLW,HLL,HLS}. The adaptation of these methods to the present setting requires only standard modifications; therefore, we omit the detailed exposition.
\end{proof}

The rest of the section focuses on
the equivalent characterization of $H^1_{tw}(\mathbb{R}^{2m})$ via twisted atomic decomposition. 

\subsection{Tent space and the twisted atom}
We start with the tent space adapted to the dyadic tubes structure.
\begin{defn}[Tent structure]~\\
For all $\mathbf{j}\in\mathbb{Z}^3$, let $R\in\mathfrak{R}_j$ be a dyadic rectangle and define the tent with base $R$ by 
    \begin{align*}
        \mathcal{T}(R)&:=R\times[2^{j_1},\,2^{j_1+1})\times[2^{j_2},\,2^{j_2+1})\times[2^{j_3},\,2^{j_3+1})
        \subset \mathbb{R}^{2m}\times\mathbb{R}_+^3,
    \end{align*}
    where $\mathbb{R}^3_+:=(\mathbb{R}_+)^3$.
\end{defn}
A geometric fact about the tent is that the family of tents forms a partition of $\mathbb{R}^{2m}\times\mathbb{R}^3_+$.
\begin{prop}\label{tent}
We have the decomposition of the upper-half space $\mathbb{R}^{2m}\times\mathbb{R}_+^3$ into tents.
    $$    \mathbb{R}^{2m}\times\mathbb{R}_+^3=\bigsqcup_{R\in\mathfrak{R}} \mathcal{T}(R).
    $$
\end{prop}

\begin{proof}
    Let $\mathbf{j},\,\mathbf{j}'\in\mathbb{Z}^3$ be distinct, then for all $R\in\mathfrak{R}_\mathbf{j}$ and $R'\in\mathfrak{R}_{\mathbf{j}'}$, we have $\mathcal{T}(R)\cap \mathcal{T}(R')=\emptyset.$ On the other hand, fix any $\mathbf{j}\in\mathbb{Z}^3$ and let $R\neq\,R'\in \mathfrak{R}_\mathbf{j}$, then $R\cap R'=\emptyset$, since $\mathfrak{R}_\mathbf{j}$ forms a tiling on $\mathbb{R}^{2m}$. Thus, we have the union of all tent is a disjoint union subset of $\mathbb{R}^{2m}\times\mathbb{R}_+^3$.

    Moreover, given any point $(\mathbf{x},\,\mathbf{t})$ in $\mathbb{R}^{2m}\times\mathbb{R}^3_+$. There is an unique $\mathbf{j}\in\mathbb{Z}^3$ such that
    $\mathbf{t}\in[2^{j_1},\,2^{j_1+1})\times[2^{j_2},\,2^{j_2+1})\times[2^{j_3},\,2^{j_3+1})$. Besides, since $\mathfrak{R}_\mathbf{j}$ forms a tiling on $\mathbb{R}^{2m}$, then one can find one $R\in\mathfrak{R}_\mathbf{j}$ such that $\mathbf{x}\in R$, which completes the proof.
\end{proof}
\smallskip

We recall the definition of the twisted atom of each type and the twisted atomic Hardy space. The notion of the {\it enlargement} is needed.

Let $R=c_R+T(0,2^j)$ be a dyadic rectangle, the {\it $\sigma$-enlargement} $R^*$ of the rectangle $R$ is defined to be $R^*:=c_R+T(0,2^{j+\sigma})$
\smallskip

\begin{defn}[Atoms of type $\Diamond={\rm I,I\!I, I\!I\!I, I\!V}$ or ${\rm V}$]~\\
    Fix positive integers $N_j>0$ for $j=1,2,3$. Let $\Diamond={\rm I,I\!I, I\!I\!I, I\!V}$ or ${\rm V}$ be the type. An {\it atom of type} $\Diamond$ is a $L^2(\mathbb{R}^{2m})$ function $a_{\Omega^\Diamond}$ such that there is an open set $\Omega^\Diamond$ of $\mathbb{R}^{2m}$ of finite measure and $L^2(\mathbb{R}^{2m})$ functions $a^\Diamond_R$, called {\it particles}, and $b_R$ in $\text{Dom}(\triangle^{N_1}_1\triangle^{N_2}_2\triangle^{N_3}_{twist})$ for all maximal tubes of type $\Diamond$, $R\in m^\Diamond(\Omega^\Diamond)$, such that
    \begin{align*}
        &(A1)\,\,a^\Diamond_R=\triangle^{N_1}_1\triangle^{N_2}_2\triangle^{N_3}_{twist} b^\Diamond_R\quad\text{and}\quad \text{supp}(b_R)\subset R^*,\quad\text{where}\,\,R^*\,\,\text{is the}\,\sigma-\text{enlargement of}\,\,R;\\
        &(A2)\,\,\text{The sum}\,\,\sum_{R\in m^\Diamond(\Omega^\Diamond)}a^\Diamond_R\,\,\text{converges in}\,\,L^2(\mathbb{R}^{2m})\quad\text{and}\,\,\sum_{R\in m^\Diamond(\Omega^\Diamond)}\|a^\Diamond_R\|^2_{L^2(\mathbb{R}^{2m})}\lesssim\frac{1}{|\Omega^\Diamond|};\\
        &(A3)\,\,a_{\Omega^\Diamond}=\sum_{R\in m^\Diamond(\Omega^\Diamond)}a^\Diamond_R\quad\text{and}\quad \|a_{\Omega^\Diamond}\|_{L^2(\mathbb{R}^{2m})}\lesssim\frac{1}{\sqrt{|\Omega^\Diamond|}};  \end{align*}
        Moreover, let $j=1,2,3$, then for all $0\leq k_j\leq N_j$, we have the cancellation property:
        
        1. for $\Diamond={\rm I}$,
\begin{align}\label{I-1}
\sum_{R\in m^\Diamond(\Omega^\Diamond)}\ell(I_1)^{-4k_1}\ell(I_2)^{-4k_2}\left\|\triangle^{N_1-k_1}_1\triangle^{N_2-k_2}_2\triangle^{N_3}_{twist} b^\Diamond_R\right\|^2_{L^2(\mathbb{R}^{2m})}\lesssim\frac{1}{|\Omega^\Diamond|},
\end{align}
and
\begin{align}\label{I-2}
\sum_{R\in m^\Diamond(\Omega^\Diamond)}\ell(I_1)^{-4k_1}\ell(I_2)^{-4k_2}\ell(I_1)^{-\alpha}\ell(I_2)^{-4N_3+\alpha}\left\|\triangle^{N_1-k_1}_1\triangle^{N_2-k_2}_2 b^\Diamond_R\right\|^2_{L^2(\mathbb{R}^{2m})}\lesssim\frac{1}{|\Omega^\Diamond|}
        \end{align}
for $\alpha\in\{0,1,\cdots, 4N_3\}$; 

\smallskip
        2. for $\Diamond={\rm I\!I, I\!I\!I}$, the cancellation property becomes
  \begin{align}\label{II-1}
\sum_{R\in m^\Diamond(\Omega^\Diamond)}\ell(I_1)^{-4k_1}\ell(I_2)^{-4k_3}\left\|\triangle^{N_1-k_1}_1\triangle^{N_2}_2\triangle^{N_3-k_3}_{twist} b^\Diamond_R\right\|^2_{L^2(\mathbb{R}^{2m})}\lesssim\frac{1}{|\Omega^\Diamond|},
        \end{align}      
        and
    \begin{align}\label{II-2}
\sum_{R\in m^\Diamond(\Omega^\Diamond)}\ell(I_1)^{-4k_1}\ell(I_1)^{-\alpha}\ell(I_2)^{-4N_2+\alpha}\ell(I_2)^{-4k_3}\left\|\triangle^{N_1-k_1}_1\triangle^{N_3-k_3}_{twist} b^\Diamond_R\right\|^2_{L^2(\mathbb{R}^{2m})}\lesssim\frac{1}{|\Omega^\Diamond|}
        \end{align} 
for all $\alpha\in\{0,1,\cdots, 4N_2\}$;

        \smallskip
 
       3. for $\Diamond={\rm I\!V,V}$,
 \begin{align}\label{IV-1}
\sum_{R\in m^\Diamond(\Omega^\Diamond)}\ell(I_1)^{-4k_3}\ell(I_2)^{-4k_2}\left\|\triangle^{N_1}_1\triangle^{N_2-k_2}_2\triangle^{N_3-k_3}_{twist} b^\Diamond_R\right\|^2_{L^2(\mathbb{R}^{2m})}\lesssim\frac{1}{|\Omega^\Diamond|},
        \end{align}  
        and
\begin{align}\label{IV-2}
 \sum_{R\in m^\Diamond(\Omega^\Diamond)}\ell(I_1)^{-\alpha}\ell(I_2)^{-4N_1+\alpha}\ell(I_1)^{-4k_3}\ell(I_2)^{-4k_2}\left\|\triangle^{N_2-k_2}_2\triangle^{N_3-k_3}_{twist} b^\Diamond_R\right\|^2_{L^2(\mathbb{R}^{2m})}\lesssim\frac{1}{|\Omega^\Diamond|}
\end{align}  
for all $\alpha\in\{0,1,\cdots, 4N_1\}$.
\end{defn}
\smallskip

\subsection{From the area function to the atomic decomposition} Now, we prove that the area function Hardy space is contained in the twisted atomic Hardy space, that is the inclusion $ H^1_{area,\varphi}( \mathbb{R}^{2m})\subset H^1_{Tw,atom}( \mathbb{R}^{2m})$.
\begin{proof}
    Let $f\in L^2(\mathbb{R}^{2m})\cap H^1_{area}( \mathbb{R}^{2m})$ be given and $k\in\mathbb{Z}$. Define the sets\begin{eqnarray*}
   \Omega_k&:=&\{\mathbf{x}\in \mathbb{R}^{2m}: S_{area,\,\varphi}(f)(\mathbf{x})>2^k \},\\
    \mathcal{R}^*_k&:=&\Big\{R'\in\mathfrak{R }:\
       |(R')^*\cap\Omega_{k+1}|\leq\frac{|(R')^*|}{2^{\sigma m+1}}<|(R')^*\cap \Omega_{k}| \Big\},\\
    \widetilde{\Omega}_k&:=&\Big\{\mathbf{x}\in \mathbb{R}^{2m}:{M}_{tube}(\chi_{\Omega_k})(\mathbf{x})>\frac{1}{2^{3\sigma m+1}}\Big\},
\end{eqnarray*}
where ${M}_{tube}$ denotes the tube maximal function and $(R')^*$ is the $\sigma$-enlargement of the dyadic rectangle $R'$. Note that the $L^2$-boundedness of ${M}_{tube}$ and $S_{area,\,\varphi}$ gives that the measure of the open set $ \widetilde{\Omega}_k$ is of finite measure; and, from definition,  $\Omega_{k+1}\subset\Omega_k$ for all $k\in\mathbb Z$.

We first claim that with the notion of enlargement of the dyadic rectangles, one can have that for all dyadic rectangles $R'\in\mathcal{R}^*_k$, the $\sigma$-enlargement $(R')^*$ is contained in the set $\widetilde{\Omega}_k$ for all $k\in\mathbb{Z}.$ To see this, suppose that $R'=\mathbf{x}+T(\mathbf{0},\,2^{\mathbf{j}})$. Then the $\sigma$-enlargement $(R')^*$ is $\mathbf{x}+T(\mathbf{0},\,2^{\mathbf{j}+\sigma})$ and hence $$
(R')^*\subset\mathbf{y}+T(\mathbf{0},\,2^{\mathbf{j}+2\sigma}),\quad\forall\,\mathbf{y}\in (R')^*.
$$
As a consequence, for all $\mathbf{y}\in (R')^*$ and $R'\in\mathcal{R}^*_k$, 
\begin{align}\label{Typeee}
    M_{tube}(\chi_{\Omega_k})(\mathbf{y})
&\geq\frac{|\Omega_k\cap\left(\mathbf{y}+T(\mathbf{0},\,2^{\mathbf{j}+2\sigma}\right)|}{|\mathbf{y}+T(\mathbf{0},\,2^{\mathbf{j}+2\sigma})|}\notag
    \geq\frac{|\Omega_k\cap (R')^*|}{|\mathbf{y}+T(\mathbf{0},\,2^{\mathbf{j}+2\sigma})|}\notag>\frac{|T(\mathbf{0},\,2^{\mathbf{j}+\sigma})|}{|T(\mathbf{0},\,2^{\mathbf{j}+2\sigma})|}\cdot\frac{1}{2^{\sigma m+1}}\\
    &=\frac{1}{2^{3\sigma m+1}},
\end{align}
which leads to $(R')^*\subset\widetilde{\Omega}_k$ for all $R'\in\mathcal{R}^*_k$ and the claim is completed. 
\smallskip

Return to the atomic decomposition. Let $\Diamond={\rm I,I\!I, I\!I\!I, I\!V}$ or ${\rm V}$ be the type and define that
\[\widetilde{\Omega}^\Diamond_k:=\Big\{\mathbf{x}\in \mathbb{R}^{2m}:{M}^{\Diamond}_{tube}(\chi_{\Omega_k})(\mathbf{x})>\frac{1}{2^{3\sigma m+1}}\Big\},\quad\forall\,k\in\mathbb{Z},
\]
then by the reproducing formula (\ref{eq:reproducing}) and the structure of tent, that is Proposition \ref{tent}, we have
\begin{align*}
    f(\mathbf{x})&= 
\int_{\mathbb{R}^3_+} f*\psi_ {\mathbf{r}}*\varphi_ {\mathbf{r}}(\mathbf{x})\, \frac { d\mathbf{r}}{\mathbf{r}}
=\sum_{R'\in\mathfrak
{R}} 
\int_{\mathcal{T}(R')} \psi_ {\mathbf{r}}(\mathbf{x}-\mathbf{y})\cdot f*\varphi_ {\mathbf{r}}(\mathbf{y})\,d\mathbf{y} \,\frac { d\mathbf{r}}{\mathbf{r}}\\
&=\sum_{k\in\mathbb{Z}}\sum_{R'\in\mathcal
{R}^*_k} 
\int_{\mathcal{T}(R')} \psi_ {\mathbf{r}}(\mathbf{x}-\mathbf{y})\cdot f*\varphi_ {\mathbf{r}}(\mathbf{y})\,d\mathbf{y} \,\frac { d\mathbf{r}}{\mathbf{r}}\\
&=\sum_{k\in\mathbb{Z}}\Bigg(\sum_{\substack{R'\in\mathcal
{R}^*_k,\\R': \text{type}\, {\rm }I }}+\sum_{\substack{R'\in\mathcal
{R}^*_k,\\R': \text{type}\, {\rm I\!I} }}+ \sum_{\substack{R'\in\mathcal
{R}^*_k,\\R': \text{type}\, {\rm I\!I\!I} }}+\sum_{\substack{R'\in\mathcal
{R}^*_k,\\R': \text{type}\, {\rm I\!V} }}+\sum_{\substack{R'\in\mathcal
{R}^*_k,\\R': \text{type}\, {\rm V}}}\Bigg)
\int_{\mathcal{T}(R')} \psi_ {\mathbf{r}}(\mathbf{x}-\mathbf{y})\cdot f*\varphi_ {\mathbf{r}}(\mathbf{y})\,d\mathbf{y} \,\frac { d\mathbf{r}}{\mathbf{r}}\\
&=:\sum_{k\in\mathbb{Z}}\lambda^{\rm I}_k\cdot {a^{\rm I}_k(\mathbf{x})}+\sum_{k\in\mathbb{Z}}\lambda^{\rm I\!I}_k\cdot {a^{\rm I\!I}_k(\mathbf{x})}+\sum_{k\in\mathbb{Z}}\lambda^{\rm I\!I\!I}_k\cdot {a^{\rm I\!I\!I}_k(\mathbf{x})}+\sum_{k\in\mathbb{Z}}\lambda^{\rm I\!V}_k\cdot {a^{\rm I\!V}_k(\mathbf{x})}+\sum_{k\in\mathbb{Z}}\lambda^{\rm V}_k\cdot {a^{V}_k(\mathbf{x})}
,
\end{align*}
where for $\Diamond={\rm I,I\!I, I\!I\!I, I\!V}$ or ${\rm V}$,
$$
\lambda^ \Diamond_k:={\sqrt{|\widetilde{\Omega}^\Diamond_k|}}\cdot\Bigg\|\bigg(\sum_{\substack{R'\in\mathcal{R}^*_k,\\ R': \text{type}\, \Diamond}}\int\left|f*\varphi_ {\mathbf{r}}\right|^2(\cdot) \cdot\chi_{\mathcal{T}(R')}(\cdot,\,\mathbf{r}) \,\frac { d\mathbf{r}}{\mathbf{r}}\bigg)^{\frac{1}{2}}\Bigg\|_{L^2(\mathbb{R}^{2m})},
$$
and
\begin{align}\label{atom}
    a^ \Diamond_k(\mathbf{x})
:=\frac{1}{\lambda^{\Diamond}_k}\cdot\sum_{\substack{R'\in\mathcal{R}^*_k,\\ R': \text{type}\, \Diamond}} 
\int_{\mathcal{T}(R')} \psi_ {\mathbf{r}}(\mathbf{x}-\mathbf{y})\cdot f*\varphi_ {\mathbf{r}}(\mathbf{y})\,d\mathbf{y} \,\frac { d\mathbf{r}}{\mathbf{r}}.\end{align}
Note that by the same argument as in (\ref{Typeee}), we have for all dyadic rectangles $R'\in\mathcal{R}^*_k$ which is of the type $\Diamond$, the $\sigma$-enlargement $R^*$ is contained in the set $\widetilde{\Omega}^\Diamond_k$ for all $k\in\mathbb{Z}$ and $\Diamond={\rm I,I\!I, I\!I\!I, I\!V}$ or ${\rm V}$.

\smallskip
\noindent$\bullet$ {\rm\it\ The\ support\ of\ the\ atom\ of\ each\ type}:
\smallskip
    
Since the function $\psi^{(j)}$ has compact support for each $j=1,2,3$, then the support of the function \begin{align*}
    \int_{\mathcal{T}(R')} \psi_ {\mathbf{r}}(\mathbf{x}-\mathbf{y})\cdot f*\varphi_ {\mathbf{r}}(\mathbf{y})\,d\mathbf{y} \,\frac { d\mathbf{r}}{\mathbf{r}}
\end{align*}
is contained in 
\begin{align*}
    R'+T(0,\,\mathbf{r})\subset (R')^*\subset\widetilde{\Omega}^\Diamond_k.
\end{align*}
Thereby, for all integer $k$ and for each atom of type $\Diamond={\rm I,I\!I, I\!I\!I, I\!V}$ or ${\rm V}$,
\begin{align}\label{supportofatom}
\text{supp}(a^{\Diamond}_k)\subset\bigcup_{\substack{R'\in\mathcal{R}^*_k,\\ R': \text{type}\, \Diamond}}(R')^*\subset\widetilde{\Omega}^\Diamond_k.
\end{align}

\smallskip
\noindent$\bullet$ {\rm\it\  The\ size\ control\ of\ the\ atom}:
\smallskip

Let $\|h\|_{L^2(\mathbb{R}^{2m})}=1$, then for each $k\in\mathbb{Z}$ and for each type $\Diamond={\rm I,I\!I, I\!I\!I, I\!V}$ or ${\rm V}$,
\begin{align*}
    \langle a^{\Diamond}_k,\,h\rangle&=\frac{1}{\lambda^{\Diamond}_k}\sum_{\substack{R'\in\mathcal{R}^*_k,\\ R': \text{type}\, \Diamond}} \int_{\mathbb{R}^{2m}\times\mathcal{T}(R')}\psi_ {\mathbf{r}}(\mathbf{x}-\mathbf{y})\cdot f*\varphi_ {\mathbf{r}}(\mathbf{y})\cdot h(\mathbf{x})\,d\mathbf{x}\, \frac { d\mathbf{r}}{\mathbf{r}}\,d\mathbf{y}\\
    &=\frac{1}{\lambda^{\Diamond}_k}\sum_{\substack{R'\in\mathcal{R}^*_k,\\ R': \text{type}\, \Diamond}} \int_{\mathcal{T}(R')}f*\varphi_ {\mathbf{r}}(\mathbf{y})\cdot h*\widetilde{\psi_ {\mathbf{r}}}(\mathbf{y})\, \frac { d\mathbf{r}}{\mathbf{r}}\,d\mathbf{y},
\end{align*}
where $\widetilde{\psi_ {\mathbf{r}}}(\mathbf{x}):=\psi_ {\mathbf{r}}(-\mathbf{x})$. By H\"older's inequality to the variable $\mathbf{y}$ and the mapping property of Littlewood--Paley function, we have the $L^2$-inner product of $a^\Diamond_k$ and $h$ is bounded above by
\begin{align*}
  &\frac{1}{\lambda^{\Diamond}_k}\int_{\mathbb{R}^{2m}}\Bigg(\sum_{\substack{R'\in\mathcal{R}^*_k,\\ R': \text{type}\, \Diamond}} \int_{\mathbb{R}^3_+}\left|f*\varphi_ {\mathbf{r}}(\mathbf{y})\right|^2\cdot\chi_{\mathcal{T}(R')}(\mathbf{y},\,\mathbf{r})\, \frac { d\mathbf{r}}{\mathbf{r}}\Bigg)^{\frac{1}{2}}\cdot g_{\widetilde{\psi}}(h)(\mathbf{y})\,d\mathbf{y}\notag\\
  &\leq \frac{1}{\lambda^{\Diamond}_k}\ \Bigg\|\bigg(\sum_{\substack{R'\in\mathcal{R}^*_k,\\ R': \text{type}\, \Diamond}}\int\left|f*\varphi_ {\mathbf{r}}\right|^2(\cdot) \cdot\chi_{\mathcal{T}(R')}(\cdot,\,\mathbf{r}) \,\frac { d\mathbf{r}}{\mathbf{r}}\bigg)^{\frac{1}{2}}\Bigg\|_{L^2(\mathbb{R}^{2m})}\left\|g_{\widetilde{\psi}}(h)\right\|_{L^2(\mathbb{R}^{2m})}\notag\\
  &\lesssim\frac{1}{\sqrt{|\widetilde{\Omega}^\Diamond_k}|},
\end{align*}
which implies that for all $k\in\mathbb{Z}$ and for each type $\Diamond={\rm I,I\!I, I\!I\!I, I\!V}$ or ${\rm V}$,
\begin{align}\label{sizeofatom}
    \|a^\Diamond_k\|_{L^2(\mathbb{R}^{2m})}\lesssim\frac{1}{\sqrt{|\widetilde{\Omega}^\Diamond_k}|}.
\end{align}

\noindent$\bullet$ {\it Further decomposition--from an atom to sum of particles:} 
\smallskip

We can further decompose the atom of type $\Diamond$ in the following way. For each dyadic rectangle $R'\in\mathcal{R}^*_k$ which is of type $\Diamond$, the $\sigma$-enlargement of the dyadic rectangle, $(R')^*$, is contained in the set $\widetilde{\Omega}^\Diamond_k$. Define the set $m^\Diamond(\widetilde{\Omega}^\Diamond_k)$ to be the set of maximal tubes which is of the same type $\Diamond$ contained in $\widetilde{\Omega}^\Diamond_k$, then
\begin{align*}
    a^\Diamond_k(\mathbf{x})
&:=\frac{1}{\lambda^\Diamond_k}\cdot\sum_{\substack{R'\in\mathcal{R}^*_k,\\ R': \text{type}\, \Diamond}} 
\int_{\mathcal{T}(R')} \psi_ {\mathbf{r}}(\mathbf{x}-\mathbf{y})\cdot f*\varphi_ {\mathbf{r}}(\mathbf{y})\,d\mathbf{y} \,\frac { d\mathbf{r}}{\mathbf{r}}\\
&=\sum_{R\in m^\Diamond(\widetilde{\Omega}^\Diamond_k)}\Bigg(\frac{1}{\lambda^\Diamond_k}\cdot\sum_{\substack{R'\in\mathcal
{R}^*_k\\ R': \text{type}\, \Diamond \\R'\subset R}} 
\int_{\mathcal{T}(R')} \psi_ {\mathbf{r}}(\mathbf{x}-\mathbf{y})\cdot f*\varphi_ {\mathbf{r}}(\mathbf{y})\,d\mathbf{y} \,\frac { d\mathbf{r}}{\mathbf{r}}\Bigg)=:{\sum_{R\in m^{\Diamond}(\widetilde{\Omega}^\Diamond_k)} \underset{\text{particle}}{a^\Diamond_{k,\,R}(\mathbf{x})}}.
\end{align*}
Then each $a^\Diamond_{k,\,R}$ is supported on the $\sigma$-enlargement of the maximal tube $R$, which is of the type $\Diamond$; and being parallel to (\ref{sizeofatom}), one can also have
\begin{align}\label{sumofl2}
    \sum_{R\in m^\Diamond(\widetilde{\Omega}^\Diamond_k)} \|{a^\Diamond_{k,\,R}}\|^2_{L^2(\mathbb{R}^{2m})}\lesssim\frac{1}{|\widetilde{\Omega}^\Diamond_k|},\quad\text{for all} \,\,k\in\mathbb{Z}.
\end{align}

Besides, due to the explicit construction of the function $\psi$, we can rewrite each particle $a^\Diamond_{k,\,R}$ of type $\Diamond$ as the following: let $j=1,2,3$ and $\ddot{\Psi^{(j)}}$ be the Schwartz function defined on $\mathbb{R}^m$ with $\widehat{\ddot{\Psi^{(j)}}}(0)=0$. Define $\Psi^{(j)}\in \mathcal{S}(\mathbb{R}^m)$ by $\Psi^{(j)}:=\triangle^{N_j}_j\ddot{\Psi^{(j)}}$ and let take $\psi^{(j)}$ to be such that$$
\widehat{\psi^{(j)}}(r_j\cdot\xi_j)=r^{-2N_j}_j\cdot\widehat{\triangle^{N_j}_j\Psi^{(j)}}(r_j\cdot\xi_j),\quad\forall r_j>0,
$$
then by the Plancherel identity
\begin{align*}
\psi_{
\mathbf{r}
}(\mathbf{x}-\mathbf{y})
&=r^{-2N_1}_{1}r^{-2N_2}_{2}r^{-2N_3}_{3}\left[\left((\triangle^{N_1}_1\Psi^{(1)})_{r_1}\otimes(\triangle^{N_2}_2\Psi^{(2)})_{r_2}\right)*_3(\triangle^{N_3}_{3}\Psi^{(3)})_{r_3}\right](\mathbf{x}-\mathbf{y})\\
&=\triangle^{N_1}_1\triangle^{N_2}_2\triangle^{N_3}_{twist}\circ\left(\Psi_{\mathbf{r}}\right)(\mathbf{x}-\mathbf{y}).
\end{align*}
Then for each type $\Diamond$, if we define 
\begin{align*}
b^\Diamond_{k,\,R}(\mathbf{x}):= \frac{1}{\lambda^\Diamond_k}\sum_{\substack{R'\in\mathcal
{R}^*_k\\ R': \text{type}\, \Diamond \\ R'\subset R}} 
\int_{\mathcal{T}(R')} \Psi_ {\mathbf{r}}(\mathbf{x}-\mathbf{y})\cdot f*\varphi_ {\mathbf{r}}(\mathbf{y})\,d\mathbf{y} \,\frac { d\mathbf{r}}{\mathbf{r}}\notag
\end{align*}
and remark that the support of $b^\Diamond_{k,\,R}$ is contained in the $\sigma$-enlargement of the tube $R$, we can have
\begin{align*}
a^\Diamond_{k,\,R}(\mathbf{x})=\triangle^{N_1}_1\triangle^{N_2}_2\triangle^{N_3}_{twist}b^\Diamond_{k,\,R}(\mathbf{x}).
\end{align*}
Moreover, for all $j=1,2,3$, we let $0\leq k_j\leq N_j$. Then
\begin{align*}
\left(\triangle^{k_1}_1\triangle^{k_2}_2\triangle^{k_3}_{twist}b^\Diamond_{k,\,R}\right)(\mathbf x)
&=\frac{1}{\lambda^\Diamond_k}\sum_{\substack{R'\in\mathcal
{R}^*_k\\ R': \text{type}\, \Diamond \\ R'\subset R}} 
\int_{\mathcal{T}(R')} f*\varphi_ {\mathbf{r}}(\mathbf{y}) \ r^{-2k_1}_1r_2^{-2k_2}r_3^{-2k_3}\\
&\qquad\times\bigg[\left(\left(\triangle^{k_1}_1\Psi^{(1)}\right)_{r_1}\otimes\left(\triangle^{k_2}_2\Psi^{(2)}\right)_{r_2}\right)*_3\left(\triangle^{k_3}_{3}\Psi^{(3)}\right)_{r_3} \bigg]\,(\mathbf{x}-\mathbf{y})\,d\mathbf{y} \,\frac { d\mathbf{r}}{\mathbf{r}}.
\end{align*}
Thus, if $\Diamond={\rm I}$, then we have $R=I_1\times I_2$, for some intervals $I_1,I_2$, and $r_1\leq\ell(I_1)$, $r_2\leq\ell(I_2)$ and $r_3\leq\min\{r_1,r_2\}$. So, one can combine this geometric fact with the $L^2$-property of the particle to get (\ref{I-1}) and (\ref{I-2}). Furthermore,  if $\Diamond={\rm I\!I, I\!I\!I}$, then we have $R=I_1\times_t I_2$, for some intervals $I_1,I_2$, and $r_1\leq\ell(I_1)$, $r_3\leq\ell(I_2)$ and $r_2\leq\min\{r_1,r_3\}$; and if $\Diamond={\rm I\!V, V}$, then we have $R=I_1\,{}_t\times I_2$, for some intervals $I_1,I_2$, and $r_3\leq\ell(I_1)$, $r_2\leq\ell(I_2)$ and $r_1\leq\min\{r_2,r_3\}$. Hence we can deduce the cancellation properties (\ref{II-1})
---(\ref{IV-2}).

To be more concise, we now provide the details for verification of (\ref{II-2}), and skip the details for the other cases since they are similar.\\
Let $\boldsymbol{C}:=r^{2k_1}_1\ell(I_1)^{-2k_1}r_2^{2N_2}\left[\ell(I_1)^{-\alpha/2}\ell(I_2)^{-2N_2+\alpha/2}\right]r_3^{2k_3}\ell(I_2)^{-2k_3}$ be the constant and consider 
\begin{align}\label{harrrrrddd}
&\ell(I_1)^{-4k_1}\ell(I_1)^{-\alpha}\ell(I_2)^{-4N_2+\alpha}\ell(I_2)^{-4k_3}\left\|\triangle^{N_1-k_1}_1\triangle^{N_3-k_3}_{twist} b^\Diamond_{k,R}\right\|^2_{L^2(\mathbb{R}^{2m})}\notag\\
&=:\Bigg\|\frac{1}{\lambda^\Diamond_k}\sum_{\substack{R'\in\mathcal
{R}^*_k\\ R': \text{type}\, \Diamond \\ R'\subset R}} 
\int_{\mathcal{T}(R')} f*\varphi_ {\mathbf{r}}(\mathbf{y})\,\cdot\circledS(\cdot-\mathbf{y}) \,d\mathbf{y} \,\frac { d\mathbf{r}}{\mathbf{r}}\Bigg\|^2_{L^2(\mathbb{R}^{2m})},
\end{align}
where $\circledS$ is the function defined to be
\begin{align*}
\boldsymbol{C}\times r^{-2N_1}_1r^{-2N_2}_2r_3^{-2N_3}\bigg[\left(\left(\triangle^{N_1-k_1}_1\Psi^{(1)}\right)_{r_1}\otimes\Psi^{(2)}_{r_2}\right)*_3\left(\triangle^{N_3-k_3}_{3}\Psi^{(3)}\right)_{r_3}\bigg].
\end{align*}
From the construction of $\Psi^{(j)}$, we also have
\begin{align*}
&r^{-2N_1}_1r^{-2N_2}_2r_3^{-2N_3}\bigg[\left(\left(\triangle^{N_1-k_1}_1\Psi^{(1)}\right)_{r_1}\otimes\Psi^{(2)}_{r_2}\right)*_3\left(\triangle^{N_3-k_3}_{3}\Psi^{(3)}\right)_{r_3}\bigg]\notag\\
    &=r^{-2N_1}_1r^{-2N_2}_2r_3^{-2N_3}\bigg[\left(\left( \triangle^{N_1}_1\circ\triangle^{N_1-k_1}_1\ddot{\Psi^{(1)}}\right)_{r_1}\otimes\left(\triangle^{N_2}_2\circ\ddot{\Psi^{(2)}}\right)_{r_2}\right)*_3\left(\triangle^{N_3}_{3}\circ\triangle^{N_3-k_3}_3\ddot{\Psi^{(3)}}\right)_{r_3}\bigg]\\
   &=\triangle^{N_1}_1\triangle^{N_2}_2\triangle^{N_3}_{twist}\bigg[\left(\left( \triangle^{N_1-k_1}_1\ddot{\Psi^{(1)}}\right)_{r_1}\otimes\ddot{\Psi^{(2)}}_{r_2}\right)*_3\left(\triangle^{N_3-k_3}_3\ddot{\Psi^{(3)}}\right)_{r_3}\bigg]\\
   &=:\circledast,
\end{align*}
which implies that $
\circledS=\boldsymbol{C}\cdot\circledast$.
Therefore, to estimate (\ref{harrrrrddd}),
it sufficed to estimate that for all $h\in L^2(\mathbb{R}^{2m})$ with the normalized $L^2$-norm
\begin{align*}
   &\frac{1}{\lambda^\Diamond_k}\sum_{\substack{R'\in\mathcal
{R}^*_k\\ R': \text{type}\, \Diamond \\ R'\subset R}} 
\int_{\mathbb{R}^{2m}\times\mathcal{T}(R')} f*\varphi_ {\mathbf{r}}(\mathbf{y})\,\cdot\circledS(\mathbf{x}-\mathbf{y})\cdot h(\mathbf{x})\,d\mathbf{x} \,d\mathbf{y} \,\frac { d\mathbf{r}}{\mathbf{r}}\\
&=\frac{1}{\lambda^\Diamond_k}\sum_{\substack{R'\in\mathcal
{R}^*_k\\ R': \text{type}\, \Diamond \\ R'\subset R}} 
\int_{\mathcal{T}(R')} f*\varphi_ {\mathbf{r}}(\mathbf{y})\cdot\boldsymbol{C}\cdot \left(h*\widetilde{\circledast}\right)(\mathbf y)\,d\mathbf{y} \,\frac { d\mathbf{r}}{\mathbf{r}},
\end{align*}
where $\widetilde{\circledast}(\cdot):=\circledast(-\cdot)$.\\
For $\Diamond={\rm I\!I, I\!I\!I}$, from the construction, the geometric structure gives that
\begin{align*}
    r_1\leq\ell(I_1),r_3\leq\ell(I_2)\quad\text{and}\quad r_2\leq\min\{\ell(I_1),\ell(I_2)\},
\end{align*} 
and hence $\boldsymbol{C}\leq1$.
By H\"older's inequality to the variable $\mathbf{y}$ and the mapping property of Littlewood--Paley function, (\ref{harrrrrddd}) can be further bounded by
\begin{align*}
  &\frac{1}{\lambda^{\Diamond}_k}\int_{\mathbb{R}^{2m}}\Bigg(\sum_{\substack{R'\in\mathcal{R}^*_k,\\ R': \text{type}\, \Diamond}} \int_{\mathbb{R}^3_+}\left|f*\varphi_ {\mathbf{r}}(\mathbf{y})\right|^2\cdot\boldsymbol{C}^2\cdot\chi_{\mathcal{T}(R')}(\mathbf{y},\,\mathbf{r})\, \frac { d\mathbf{r}}{\mathbf{r}}\Bigg)^{\frac{1}{2}}\cdot g_{\widetilde{\circledast}}(h)(\mathbf{y})\,d\mathbf{y}\notag\\
  &\leq \frac{1}{\lambda^{\Diamond}_k}\ \Bigg\|\bigg(\sum_{\substack{R'\in\mathcal{R}^*_k,\\ R': \text{type}\, \Diamond}}\int\left|f*\varphi_ {\mathbf{r}}\right|^2(\cdot) \cdot\chi_{\mathcal{T}(R')}(\cdot,\,\mathbf r) \,\frac { d\mathbf{r}}{\mathbf{r}}\bigg)^{\frac{1}{2}}\Bigg\|_{L^2(\mathbb{R}^{2m})}\left\|g_{\widetilde{\circledast}}(h)\right\|_{L^2(\mathbb{R}^{2m})}\notag\\
  &\lesssim\frac{1}{\sqrt{|\widetilde{\Omega}^\Diamond_k}|},
\end{align*}
which gives (\ref{II-2}) as desired.
\smallskip

To complete the proof, it remains to verify that for each type $\Diamond={\rm I,I\!I, I\!I\!I, I\!V}$ or ${\rm V}$:
$$
\sum_k|\lambda^\Diamond_k|\lesssim\|S_{area,  \varphi}(f)\|_{L
^1
(\mathbb{R}^{2m})}.
$$
For all dyadic rectangles of type $\Diamond$, $R'\in\mathcal{R}^*_k$, if  $(\mathbf{x},\mathbf{r})\in \mathcal T (R')$, then $T(\mathbf{x},\mathbf{r})\subset (R')^*\subset\widetilde{\Omega}^\Diamond_k$ and $|T(\mathbf{x},\mathbf{r})\cap \Omega_{k+1}|\leq|(R')^*\cap\Omega_{k+1}|<{|(R')^*|}/{2^{\sigma m+1}}$. Therefore,
\begin{align*}
\frac{\left|\left(\widetilde{\Omega}^
    \Diamond_k-\Omega_{k+1}\right)\cap T(\mathbf{x},\,\mathbf{r})\right|}{|T(\mathbf{x},\,\mathbf{r})|}>1-\frac{1}{2^{\sigma m+1}}\cdot\frac{|R^*|}{|T(\mathbf{x},\,\mathbf{r})|}\geq\frac{1}{2},
\end{align*}
which leads to
\begin{align*}
\sum_{\substack{R'\in\mathcal{R}^*_k,\\ R': \text{type}\, \Diamond}}\int_{\mathcal{T}(R')}\left|f*\varphi_ {\mathbf{r}}\right|^2(\mathbf{x})  \,\frac { d\mathbf{r}}{\mathbf{r}}\,d\mathbf{x}
&\leq 2 \cdot\sum_{\substack{R'\in\mathcal{R}^*_k,\\ R': \text{type}\, \Diamond}}\int_{\mathcal{T}(R')}\left|f*\varphi_ {\mathbf{r}}\right|^2(\mathbf{x})\cdot\frac{\left|\left(\widetilde{\Omega}^\Diamond_k-\Omega_{k+1}\right)\cap T(\mathbf{x},\,\mathbf{r})\right|}{|T(\mathbf{x},\,\mathbf{r})|}  \,\frac { d\mathbf{r}}{\mathbf{r}}\,d\mathbf{x}\\
&\leq 2 \cdot\int_{\widetilde{\Omega}^\Diamond_k-\Omega_{k+1}}\int_{\mathbb{R}^{2m}\times\mathbb{R}^3_{+}}\left|f*\varphi_ {\mathbf{r}}\right|^2(\mathbf{x})\cdot\frac{\chi_{T(\mathbf{0},\,\mathbf{r})}(\mathbf{y}-\mathbf{x})}{|T(\mathbf{0},\,\mathbf{r})|}  \,\frac { d\mathbf{r}}{\mathbf{r}}\,d\mathbf{x}\,d\mathbf{y}\\
&\lesssim\int_{\widetilde{\Omega}_k-\Omega_{k+1}}\int_{\mathbb{R}^3_{+}}\left|f*\varphi_ {\mathbf{r}}\right|^2*\chi_\mathbf{r}(\mathbf{y})  \,\frac { d\mathbf{r}}{\mathbf{r}}\,d\mathbf{y}\\
&=\int_{\widetilde{\Omega}_k-\Omega_{k+1}}S^2_{area,  \varphi}(f)(\mathbf{y})\,d\mathbf{y}.
\end{align*}
Then by the layer-cake formula, we have
\begin{align*}
\sum_k|\lambda^\Diamond_k|
&=\sum_k\sqrt{|\widetilde{\Omega}^\Diamond_k|}\cdot\Bigg(\sum_{\substack{R'\in\mathcal{R}^*_k,\\ R': \text{type}\, \Diamond}}\int_{\mathcal{T}(R')}\left|f*\varphi_ {\mathbf{r}}\right|^2(\mathbf{x})  \,\frac { d\mathbf{r}}{\mathbf{r}}\,d\mathbf{x}\Bigg)^{\frac{1}{2}}
\\
&\lesssim\sum_k\sqrt{|\widetilde{\Omega}_k|}\cdot\left(\int_{\widetilde{\Omega}_k-\Omega_{k+1}}S^2_{area, \varphi}(f)(\mathbf{y})\,d\mathbf{y}\right)^{1/2}\\
&\lesssim\sum_k{|\widetilde{\Omega}_k|}\cdot2^k\\
&\lesssim\|S_{area,  \varphi}(f)\|_{L^1(\mathbb{R}^{2m})}.
\end{align*}
The proof is complete.
\end{proof}

\smallskip

\section{Characterization of the atomic Hardy space: the inclusion $H^1_{Tw,atom}( \mathbb{R}^{2m}) \subset H^1_{area,\varphi}( \mathbb{R}^{2m})$    }\label{sec:6}  
This section contributes to the other direction of the equivalence of the Hardy space. We begin with the Littlewood--Paley inequality.
\begin{lem}[Littlewood--Paley inequality]~\\
    For any $L^2(\mathbb{R}^{2m})$-integrable function $f$,
    \begin{align*}
        \int_{0}^{+\infty }
 \left\| \left(f(\cdot,y_2)*_1\varphi^{(1)}_{ {r}_1 }\right)(y_1) \right\|^2_{L^2(\mathbb{R}^{ m},\,dy_1)}   \frac {d{r_1}}{{r}_1} & \lesssim  \left\| f(\mathbf{y}) \right\|^2_{L^2(\mathbb{R}^{ m} ,\,dy_1)};
    \end{align*}
    and
    \begin{align*}
        \int_{0}^{+\infty }
\left \| \left(f(y_1,\cdot)*_2\varphi^{(2)}_{ {r}_2 }\right)(y_2) \right\|^2_{L^2(\mathbb{R}^{ m},\,dy_2 )}   \frac {d{r_2}}{{r}_2} & \lesssim\left  \| f(\mathbf{y}) \right\|^2_{L^2(\mathbb{R}^{ m},\,dy_2 )} 
    \end{align*}
   and  
    \begin{align*}
        \int_{0}^{+\infty }
 \left\| f*_3\varphi^{(3)}_{ {r}_3 } \right\|^2_{L^2(\mathbb{R}^{2m} )}   \frac { d {r_3}}{ {r}_3} & \lesssim  \left\| f \right\|^2_{L^2(\mathbb{R}^{2m} )}.
    \end{align*}
\end{lem}

\begin{proof}
    By Plancheral identity, we have
    \begin{align*}
              \int_{0}^{+\infty }
 \| \left(f(\cdot,y_2)*_1\varphi^{(1)}_{ {r}_1 }\right)(y_1) \|^2_{L^2(\mathbb{R}^{ m},\,dy_1)}   \frac {d{r_1}}{{r}_1}&=  \int_{0}^{+\infty }
 \| \widehat{f}(\xi_1,y_2)\cdot\widehat{\varphi^{(1)}_{ {r}_1 }}(\xi_1) \|^2_{L^2(\mathbb{R}^{ m},\,d\xi_1)}   \frac {d{r_1}}{{r}_1}\\
 &=\int_{\mathbb{R}^m}
\left|\widehat{f}(\xi_1,y_2)\right|^2\int^\infty_0|\widehat{\varphi^{(1)}}(r_1\xi_1)|^2\frac{dr_1}{r_1}\,d\xi_1\\
&\lesssim \| f(\mathbf{y}) \|^2_{L^2(\mathbb{R}^{ m} ,\,dy_1)}.
    \end{align*}
    Besides, if we take the Fourier transform to the space variable $\mathbf{y}$, then we have
    \begin{align*}
        \int_{0}^{+\infty }
 \| f*_3\varphi^{(3)}_{ {r}_3 } \|^2_{L^2(\mathbb{R}^{2m} )}   \frac { d {r_3}}{ {r}_3}&= \int_{0}^{+\infty }
 \| \widehat{f}(\xi)\cdot\widehat{\varphi^{(3)}}(r_3(\xi_1+\xi_2)) \|^2_{L^2(\mathbb{R}^{2m} )}   \frac { d {r_3}}{ {r}_3}\\
 &= \int_{\mathbb{R}^{2m}}\left|\widehat{f}(\xi)\right|^2\int_{0}^{+\infty }
 \left| \widehat{\varphi^{(3)}}(r_3(\xi_1+\xi_2)) \right|^2  \frac { d {r_3}}{ {r}_3}\,d\xi\\
 & \lesssim  \| f \|^2_{L^2(\mathbb{R}^{2m} )}.
    \end{align*}
The proof is complete.
\end{proof}

To prove the main theorem, we establish the following auxiliary result on the decomposition of the non-compact 
Schwartz functions $\varphi$ into building blocks. These blocks possess compact support, $C^\infty$ smoothness, and crucially, they inherit the higher-order cancellation moments from the wavelets.

\begin{lem}\label{lemma-decomposition}
Fix the dimension $d$, a decay parameter $\gamma > g$, a dilation parameter $\overline{C} > 1$, and a moment integer $M \ge 0$.
Let $\varphi \in \mathcal{S}(\mathbb{R}^d)$ be a Schwartz function such that:
\begin{equation}\label{eqn:psi_moments}
    \int_{\mathbb{R}^d} x^\beta \varphi(x) \, dx = 0 \quad \text{for all } |\beta| \le M.
\end{equation}
For any scale $j \in \mathbb{Z}$, define the dilated function $\varphi_j(x) := 2^{-jd} \varphi(2^{-j}x)$.

Then, there exists a sequence of functions  $\{ \varphi_{\ell, j} \}_{\ell=0}^\infty$ such that we have the pointwise and $L^q$ decomposition:
\begin{equation}\label{eqn:decomp_simple}
    \varphi_j(x) = \sum_{\ell=0}^\infty (\overline{C} 2^\ell)^{-\gamma} \varphi_{\ell, j}(x).
\end{equation}
The function $\varphi_{\ell, j}$ satisfy the following properties, where $R_{\ell, j} := \overline{C} 2^\ell 2^j$ denotes the support radius:

\begin{itemize}
    \item[(i)] {\rm (Compact support)}
    $$ \supp \varphi_{\ell, j} \subset B(0, 2 R_{\ell, j}). $$

    \item[(ii)] {\rm (Size and smoothness)} For every multi-index $\mu$, there exists a constant $C_{\mu} > 0$ (independent of $j, \ell$) such that:
    $$ |\partial^\mu \varphi_{\ell, j}(x)| \leq C_{\mu} (R_{\ell, j})^{-|\mu|} \frac{1}{|B(0, R_{\ell, j})|}. $$

    \item[(iii)] {\rm (Higher order cancellation)} For all $|\beta| \le M$:
    $$ \int_{\mathbb{R}^d} x^\beta \varphi_{\ell, j}(x) \, dx = 0. $$
\end{itemize}
\end{lem}

\begin{proof}
By the rescaling argument, it suffices to prove the lemma for the base case $j=0$.

Let $\varphi \in \mathcal{S}(\mathbb R^d)$ with vanishing moments up to order $M$.Take $R_\ell = \overline{C} 2^\ell$ and let $\{h_\ell\}$ be the smooth cut-off functions such that $\supp h_\ell \subset B(0, R_\ell)$ and $h_\ell(x) = 1$ for $x \in B(0, R_\ell/4)$.
Define the pieces $\Lambda_\ell(x)$ by
$$ \Lambda_0 = h_0 \varphi, \quad\text{and}\quad \Lambda_\ell = (h_\ell - h_{\ell-1})\cdot\varphi,$$ where $\ell\geq 1$.
Note that for each $\ell\geq1$, $\supp \Lambda_\ell \subset \{ x : R_\ell/8 \le |x| \le R_\ell \}$.

Let $\beta$ be a multi-index with $|\beta| \le M$. The $\beta$-moment of the $\ell$-th piece is defined by
$$ m_{\ell, \beta} := \int_{\mathbb R^d} x^\beta \Lambda_\ell(x) \, dx, $$
and, by the vanishing moment condition of $\varphi$ , $\sum_{\ell=0}^\infty m_{\ell, \beta}= 0$.
We also define the partial sum $S_{\ell, \beta}$ by
$$ S_{\ell, \beta} := \sum_{k=0}^\ell m_{k, \beta} = - \sum_{k=\ell+1}^\infty m_{k, \beta}. $$
For $\ell \ge 1$, the integration of $m_{\ell, \beta}$ is over the annulus $A_\ell \simeq B(0, R_\ell) \setminus B(0, R_{\ell-1})$; and on this domain, $|x| \simeq R_\ell$, which leads to
\begin{align*}
    |m_{\ell, \beta}| \le \int_{A_\ell} |x|^{|\beta|} |\varphi(x)| \, dx 
    &\le C_N \int_{R_{\ell-1} \le |x| \le R_\ell} |x|^{|\beta|} (1+|x|)^{-N} \, dx \\
    &\le C'_N R_\ell^{|\beta| - N + d}.
\end{align*}
Since $R_\ell = \overline{C} 2^\ell$, the term $R_\ell^{|\beta|-N+d}$ decays like $(2^{|\beta|-N+d})^\ell$. By choosing $N$ sufficiently large (specifically $N > d + M + \gamma$), we ensure rapid exponential decay; and the tail sum $S_{\ell, \beta}$ is dominated by the first term in the summation (geometric series dominance):
\begin{equation}\label{eq:S_decay}
    |S_{\ell, \beta}| \le \sum_{k=\ell+1}^\infty |m_{k, \beta}| \lesssim \sum_{k=\ell+1}^\infty R_k^{|\beta|-N+d} \lesssim R_\ell^{|\beta|-N+d}.
\end{equation}

We now construct functions $\theta_{\ell, \beta}$ supported in $B(0, R_\ell)$ that are biorthogonal to polynomials and satisfy specific derivative bounds.
Let $\mathcal{P}_M$ be the space of polynomials on $\mathbb R^d$ of degree at most $M$.
There exists a set of smooth functions $\{\Theta_\nu\}_{|\nu|\le M}$ supported in the unit ball $B(0,1)$ that are biorthogonal to the monomial basis $\{x^\beta\}_{|\beta|\le M}$, that is 
$$ \int_{B(0,1)} z^\beta \Theta_\nu(z) \, dz = \delta_{\beta, \nu}. $$
Now, define the scaled functions for scale $R_\ell$ by
$ \theta_{\ell, \nu}(x) := R_\ell^{-d - |\nu|} \Theta_\nu \left( {x}/{R_\ell} \right).$ 
Since $\supp \Theta_\nu \subset B(0,1)$, we have $\supp \theta_{\ell, \nu} \subset B(0, R_\ell)$.
For each $|\beta|\le M$, calculate the moments of these scaled functions against any monomial $x^\beta$, one has that
\begin{align}\label{8877}
    \int_{\mathbb R^d} x^\beta \theta_{\ell, \nu}(x) \, dx= R_\ell^{|\beta| - |\nu|} \delta_{\beta, \nu}.
\end{align}
In particular, if $\beta \neq \nu$, (\ref{8877}) is $0$, while if $\beta = \nu$, the scaling factor in (\ref{8877}) is $R_\ell^{0} = 1$. Thus,
\begin{equation}\label{eq:theta_ortho}
    \int_{\mathbb R^d} x^\beta \theta_{\ell, \nu}(x) \, dx = \delta_{\beta, \nu}.
\end{equation}
For any multi-index $\mu$, we also have
$$ \partial^\mu_x \left[ \Theta_\nu \left( \frac{x}{R_\ell} \right) \right] = R_\ell^{-|\mu|} (\partial^\mu \Theta_\nu) \left( \frac{x}{R_\ell} \right), $$
which implies that
\begin{equation}\label{eq:theta_deriv}
    |\partial^\mu \theta_{\ell, \nu}(x)| = R_\ell^{-d-|\nu|} | \partial^\mu \Theta_\nu(\cdot)| R_\ell^{-|\mu|} \lesssim R_\ell^{-d-|\nu|-|\mu|}.
\end{equation}

We now construct $\widetilde{\Lambda}_\ell$ with vanishing moments up to order $M$. Define $S_{-1, \beta} = 0$. We construct the corrected building block:
$$ \widetilde{\Lambda}_\ell(x) := \Lambda_\ell(x) + \sum_{|\nu|\le M} S_{\ell-1, \nu} \theta_{\ell, \nu}(x) - \sum_{|\nu|\le M} S_{\ell, \nu} \theta_{\ell+1, \nu}(x). $$

Let $|\beta| \le M$. A straightforward computation shows that 
$$ \int x^\beta \widetilde{\Lambda}_\ell = m_{\ell, \beta} + S_{\ell-1, \beta} - S_{\ell, \beta},$$
and thus, from the definition,
$$ \int x^\beta \widetilde{\Lambda}_\ell = S_{\ell, \beta} - S_{\ell, \beta} = 0. $$

Now it suffices to verify that $\varphi_{\ell, 0} = (\overline{C} 2^\ell)^\gamma \widetilde{\Lambda}_\ell$ satisfies the size and smoothness condition (ii).
We need to show for every $\mu$, it holds that
$$ |\partial^\mu \varphi_{\ell, 0}(x)| \le C_\mu R_\ell^{-|\mu|} \frac{1}{|B(0, R_\ell)|} \simeq R_\ell^{-|\mu|-d};$$
equivalently, we need to show that
$$ |\partial^\mu \widetilde{\Lambda}_\ell(x)| \lesssim (\overline{C} 2^\ell)^{-\gamma} R_\ell^{-|\mu|-d}. $$
Using the decay of $\varphi$ with exponent $K$, where $K$ is arbitrary large, 
$$ |\partial^\mu \Lambda_\ell(x)| \le C R_\ell^{-|\mu|} (1+R_\ell)^{-K} \simeq R_\ell^{-|\mu|-d} \cdot R_\ell^{d-K}. $$
Since $R_\ell = \overline{C} 2^\ell$, we need to require $R_\ell^{d-K} \lesssim (\overline{C} 2^\ell)^{-\gamma}$. It turns out that we need the requirement that $d-K \le -\gamma$, or $K \ge d + \gamma$. This restriction on $K$ can be achieved due to that the function $\varphi $ is a Schwartz function.\\
Consider the  correction term $Q_{\ell}(x) = \sum_\nu S_{\ell, \nu} \theta_{\ell+1, \nu}(x)$.
Using the bound for $S_{\ell, \nu}$ from  \eqref{eq:S_decay} (with parameter $K$) and the bound for $\theta$ from \eqref{eq:theta_deriv}, 
\begin{align*}
    |\partial^\mu Q_\ell(x)| \le \sum_{|\nu|\le M} |S_{\ell, \nu}| \cdot |\partial^\mu \theta_{\ell+1, \nu}(x)| &\lesssim \sum_{|\nu|\le M} R_\ell^{|\nu| - K + d} \cdot (R_{\ell+1})^{-d-|\nu|-|\mu|} 
    \simeq R_\ell^{-K - |\mu|}.
\end{align*}
Rewriting this in the form of the normalization $R_\ell^{-|\mu|-d}$, we have
$ |\partial^\mu Q_\ell(x)| \simeq R_\ell^{-|\mu|-d} \cdot R_\ell^{d-K}. $ Again, we have $R_\ell^{d-K} \lesssim (\overline{C} 2^\ell)^{-\gamma}$,  provide that we choose $K \ge d+\gamma$.

Both $\Lambda_\ell$ and the correction terms decay sufficiently fast in $\ell$ to absorb the scaling factor $(\overline{C} 2^\ell)^\gamma$ while maintaining the required derivative bounds scaled by $R_\ell$. Thus, $\varphi_{\ell, 0}$ is the desired function.
For the general case $\varphi_j$, we simply define $\varphi_{\ell, j}(x) := 2^{-jd} \varphi_{\ell, 0}(2^{-j}x)$. The properties (support, size, derivatives, and cancellation) scale naturally to satisfy the lemma statement.
\end{proof}

\subsection{From the atomic decomposition to the area function: proof of Theorem \ref{thm:S-atom} }

\begin{proof}
To prove this inclusion, it suffices to show that the atom of each type is in the area function Hardy space. Besides, from Lemma \ref{lemma-decomposition}, we can assume that the function $\varphi$ in the definition of $  S_{area,\varphi}$ is compactly supported.

Let $\Diamond$ be the type of tubes and let $\Omega^\Diamond\subset\mathbb{R}^{2m}$ be an open set. Suppose that  $a_{\Omega^\Diamond}$ is an atom of type $\Diamond$ supported on $\Omega^\Diamond$ such that $$
a_{\Omega^\Diamond}=\sum_{R\in m^\Diamond(\Omega^\Diamond)}a_R,
$$
where $m^\Diamond(\Omega^\Diamond)$ denotes the set of all maximal tube of type $\Diamond$ contained in the open set $\Omega^\Diamond$. 

$ \clubsuit$ We first consider the case that $\Diamond=I$. For each $R=I_1\times I_2\in m^{\rm I}(\Omega^{\rm I})$, define $\widehat{I}_1\subset\mathbb{R}^m$ to be the maximal dyadic interval such that $I_1\subset\widehat{I}_1$ and
\begin{align*}
    {\widehat{I}_1\times I_2\subset\widetilde{\Omega^{\rm I}}:=\bigg\{\mathbf{x}\in\mathbb{R}^{2m}:\,M^I_{tube}(\chi_{\Omega^{\rm I}})(\mathbf{x})>\frac{1}{2}\bigg\}.}
\end{align*}
Next, define $\widehat{I}_2\subset\mathbb{R}^m$ to be the maximal dyadic interval such that $I_2\subset\widehat{I}_2$ and
\begin{align*}
  {\widehat{I}_1\times \widehat{I}_2\subset\widetilde{\widetilde{\Omega^{\rm I}}}:=\bigg\{\mathbf{x}\in\mathbb{R}^{2m}:\,M^I_{tube}(\chi_{\widetilde{\Omega^{\rm I}}})(\mathbf{x})>\frac{1}{2}\bigg\}.}
\end{align*}
To complete the inclusion, it suffices to show that the $L^1$-norm of $  S_{area,\varphi}(a_{\Omega^{\rm I}})$ is uniformly bounded. Note that if we let {$\widehat{R}$ be the $100$-fold dilation of $ \widehat{I}_1\times \widehat{I}_2$ concentric with $ \widehat{I}_1\times \widehat{I}_2$,} then 
\begin{align}\label{ESTI}
    \|S_{area,\varphi}(a_{\Omega^{\rm I}})\|_{L^1(\mathbb{R}^{2m})}&\leq\int_{\left(\cup_{R\in m^{\rm I}(\Omega^{\rm I})}\widehat{R}\right)}S_{area,\varphi}(a_{\Omega^{\rm I}})(\mathbf{x})\,d\mathbf{x}+\int_{\left(\cup_{R\in m^{\rm I}(\Omega^{\rm I})}\widehat{R}\right)^c}S_{area,\varphi}(a_{\Omega^{\rm I}})(\mathbf{x})\,d\mathbf{x}\notag\\    &\lesssim\left|\left(\cup_{R\in m^{\rm I}(\Omega^{\rm I})}\widehat{R}\right)\right|^{\frac{1}{2}}\cdot\|a_{\Omega^{\rm I}}\|_{L^2}+\sum_{R\in m^{\rm I}(\Omega^{\rm I})}\int_{\left(\cup_{R\in m^{\rm I}(\Omega^{\rm I})}\widehat{R}\right)^c}S_{area,\varphi}(a_R)(\mathbf{x})\,d\mathbf{x}\notag\\    &\leq\left|\left(\cup_{R\in m^{\rm I}(\Omega^{\rm I})}\widehat{R}\right)\right|^{\frac{1}{2}}\cdot\|{a_{\Omega^{\rm I}}}\|_{L^2}+\sum_{R\in m^{\rm I}(\Omega^{\rm I})}\int_{\left(\widehat{R}\right)^c}S_{area,\varphi}(a_R)(\mathbf{x})\,d\mathbf{x}\notag\\
    &\lesssim 1+\sum_{\substack{R\in m^{\rm I}(\Omega^{\rm I})\\R=I_1\times I_2 }}\int_{\left(\widehat{R}\right)^c}S_{area,\,\varphi}(a_R)(\mathbf{x})\,d\mathbf{x},
\end{align}
where the last inequality follows by the $L^2$-boundedness property of the tube maximal function and the atom of type ${\rm I}$.

The tube $R=I_1 \times I_2$ is the standard dyadic rectangle  in $\mathbb{R}^{m}\times\mathbb{R}^m$. Then in this case
\begin{align}\label{EST1-1}
  &\sum_{\substack{R\in m^{\rm I}(\Omega^{\rm I})\\R=I_1\times I_2 }}\int_{\left(\widehat{R}\right)^c}S_{area,\,\varphi}(a_R)(\mathbf{x})\,d\mathbf{x}\notag\\ &\leq \sum_{\substack{R\in m^{\rm I}(\Omega^{\rm I})\\R=I_1\times I_2 }} \int_{(100\widehat{I_1})^c\times(100I_2)}S_{area,\,\varphi}(a_R)(\mathbf{x})\,d\mathbf{x}+\sum_{\substack{R\in m^{\rm I}(\Omega^{\rm I})\\R=I_1\times I_2 }} \int_{(100\widehat{I_1})^c\times (100I_2)^c}S_{area,\,\varphi}(a_R)(\mathbf{x})\,d\mathbf{x}\notag\\
&\quad+\sum_{\substack{R\in m^{\rm I}(\Omega^{\rm I})\\R=I_1\times I_2 }}\int_{(100{I_1})^c\times(100\widehat{I_2})^c}\ S_{area,\,\varphi}(a_R)(\mathbf{x})\,d\mathbf{x}+\sum_{\substack{R\in m^{\rm I}(\Omega^{\rm I})\\R=I_1\times I_2 }}\int_{(100{I_1})\times(100\widehat{I_2})^c} S_{area,\,\varphi}(a_R)(\mathbf{x})\,d\mathbf{x}\notag\\
  &=:\sum_{\substack{R\in m^{\rm I}(\Omega^{\rm I})\\R=I_1\times I_2 }}(\mathcal{A})+\sum_{\substack{R\in m^{\rm I}(\Omega^{\rm I})\\R=I_1\times I_2 }}(\mathcal{B})+\sum_{\substack{R\in m^{\rm I}(\Omega^{\rm I})\\R=I_1\times I_2 }}(\mathcal{C})+\sum_{\substack{R\in m^{\rm I}(\Omega^{\rm I})\\R=I_1\times I_2 }}(\mathcal{D}).
\end{align}
By symmetry, we only need to estimate the case $(\mathcal{C})$ and the case $(\mathcal{D})$. We start with the case $(\mathcal{D})$.

$\bullet$ {\it Case $(\mathcal{D})$}:
 By applying H\"older's inequality to the first variable $x_1$, we have
 \begin{align}\label{Foundfactori}
     (\mathcal{D})&\lesssim |I_1|^{\frac{1}{2}}\int_{(100\widehat{I_2})^c}\left(\int_{100 I_1}S^2_{area,\,\varphi}(a_R)(\mathbf{x})\,dx_1\right)^{\frac{1}{2}}\,d{x_2}\notag\\
     &=|I_1|^{\frac{1}{2}}\int_{(100\widehat{I_2})^c}\left(\int_{ \mathbb{R}^{2m}\times \mathbb{R}^3_+}\int_{ 100I_1}|a_R* \varphi_{\mathbf{r} } |^2(\mathbf{y})\cdot \chi_{\mathbf{r} }(\mathbf{x}-\mathbf{y}) \,dx_1\,\frac {  
    d\mathbf{r}}{\mathbf{r}}\,d\mathbf{y}\right)^{\frac{1}{2}}\,d{x_2}.
 \end{align}
By integrating the variable $x_1$ with respect to the cone structure $ \chi_{\mathbf{r} }$ and apply Littlewood--Paley inequality to eliminate $\varphi^{(1)}_{r_1}$, 
 \begin{align*}
&\int_{ \mathbb{R}^{2m}\times \mathbb{R}^3_+}\int_{ 100I_1}|a_R* \varphi_{\mathbf{r} } |^2(\mathbf{y})\cdot \chi_{\mathbf{r} }(\mathbf{x}-\mathbf{y}) \,dx_1\,\frac {  
    d\mathbf{r}}{\mathbf{r}}\,d\mathbf{y}\\
&\lesssim\int_{ \mathbb{R}^{m}\times \mathbb{R}^2_+}\left(\chi_{B_{r_2}(0)}*\chi_{B_{r_3}(0)}\right
)(x_2-y_2)\cdot
\left\|a_R*_2 \varphi^{(2)}_{{r}_2 }*_3\varphi^{(3)}_{r_3}(\mathbf{y})\right\|^2_{L^2(\mathbb{R}^m,dy_1)}\,\frac { dr_2dr_3}{r^{m+1}_2r^{m+1}_3}\,d{y}_2\\
&=:\widetilde{(\mathcal{D}1)}+\widetilde{(\mathcal{D}2)}+\widetilde{(\mathcal{D}3)},
    \end{align*} in which 
\begin{align*}
\widetilde{(\mathcal{D}1)}:=\int_{ \mathbb{R}^{m}}\int_{\mathbb{R}_+}\int^\infty_{\frac{|x_2-c_{I_2}|}{4}}\left(\chi_{B_{r_2}(0)}*\chi_{B_{r_3}(0)}\right
)(x_2-y_2)\cdot
\left\|a_R*_2 \varphi^{(2)}_{{r}_2 }*_3\varphi^{(3)}_{r_3}(\mathbf{y})\right\|^2_{L^2(\mathbb{R}^m,dy_1)}\,\frac { dr_2dr_3}{r^{m+1}_2r^{m+1}_3}\,d{y}_2,
 \end{align*}
 \begin{align*}
\widetilde{(\mathcal{D}2)}:=\int_{ \mathbb{R}^{m}}\int_{\mathbb{R}_+}\int^{\frac{|x_2-c_{I_2}|}{4}}_{\ell_2}\left(\chi_{B_{r_2}(0)}*\chi_{B_{r_3}(0)}\right
)(x_2-y_2)\cdot
\left\|a_R*_2 \varphi^{(2)}_{{r}_2 }*_3\varphi^{(3)}_{r_3}(\mathbf{y})\right\|^2_{L^2(\mathbb{R}^m,dy_1)}\,\frac { dr_2dr_3}{r^{m+1}_2r^{m+1}_3}\,d{y}_2,
\end{align*}
and 
\begin{align*}
\widetilde{(\mathcal{D}3)}:=\int_{\mathbb{R}^{m}}\int_{\mathbb{R}_+}\int^{\ell_2 }_0\left(\chi_{B_{r_2}(0)}*\chi_{B_{r_3}(0)}\right
)(x_2-y_2)\cdot
\left\|a_R*_2 \varphi^{(2)}_{{r}_2 }*_3\varphi^{(3)}_{r_3}(\mathbf{y})\right\|^2_{L^2(\mathbb{R}^m,dy_1)}\,\frac { dr_2dr_3}{r^{m+1}_2r^{m+1}_3}\,d{y}_2.
\end{align*}
Therefore, from (\ref{Foundfactori}) and the definition of $\widetilde{(\mathcal{D}1)}$, $\widetilde{(\mathcal{D}2)}$ and $\widetilde{(\mathcal{D}3)}$,
 \begin{align*}
     (\mathcal{D})&\lesssim |I_1|^{\frac{1}{2}}\int_{(100\widehat{I_2})^c}\left(\int_{100 I_1}S^2_{area,\,\varphi}(a_R)(\mathbf{x})\,dx_1\right)^{\frac{1}{2}}\,d{x_2}\notag\\
     &\lesssim|I_1|^{\frac{1}{2}}\int_{(100\widehat{I_2})^c}\left(\widetilde{(\mathcal{D}1)}\right)^{\frac{1}{2}}\,d{x_2}+|I_1|^{\frac{1}{2}}\int_{(100\widehat{I_2})^c}\left(\widetilde{(\mathcal{D}2)}\right)^{\frac{1}{2}}\,d{x_2}+|I_1|^{\frac{1}{2}}\int_{(100\widehat{I_2})^c}\left(\widetilde{(\mathcal{D}3)}\right)^{\frac{1}{2}}\,d{x_2}.
 \end{align*}
\smallskip

$\blacksquare$ {\it Estimate on the term $\widetilde{(\mathcal{D}1)}$}:
Applying the trivial estimate to the convolution $\chi_{B_{r_2}(0)}*\chi_{B_{r_3}(0)}$ and the Littlewood--Paley inequality to eliminate, we have the term $\widetilde{(\mathcal{D}1)}$ is dominated by
\begin{align*}
&\int_{\mathbb{R}_+}\int^\infty_{\frac{|x_2-c_{I_2}|}{4}}\left[\int_{\mathbb{R}^{2m}}\left|a_R*_2 \varphi^{(2)}_{{r}_2 }*_3\varphi^{(3)}_{r_3}(\mathbf{y})\right|^2\,d\mathbf{y}\right]\,\frac{dr_2}{r^{m+1}_2}\frac { dr_3}{r_3}
\lesssim\int^\infty_{\frac{|x_2-c_{I_2}|}{4}}\left\|a_R*_2\varphi^{(2)}_{r_2}\right\|^2_{L^2(\mathbb{R}^{2m})}\,\frac{ dr_2}{r^{m+1}_2}. 
\end{align*}
On the other hand, by applying Young's convolution inequality, 
  \begin{align*}
\widetilde{(\mathcal{D}1)}
&\lesssim\int^\infty_{\frac{|x_2-c_{I_2}|}{4}}r_2^{-4N_2}\left\|\left({ {\triangle}_1^{ N_1 } {\triangle}_{twist} ^{ N_{3}}b_{R}}\right)\right\|^2_{L^2(\mathbb{R}^{2m})} \frac {dr_2}{r^{m+1}_2}
\simeq |x_2-c_{I_2}|^{-4N_2-m}\left\|\left({ {\triangle}_1^{ N_1 } {\triangle}_{twist} ^{ N_{3}}b_{R}}\right)\right\|^2_{L^2(\mathbb{R}^{2m})}. 
\end{align*}
To continue, let $E_j:=\delta_{2^j}\left(100I_2\right)-\delta_{2^{j-1}}\left(100I_2\right)$, then there is $j_0\in\mathbb{N}$ such that \begin{align}\label{dyadicdecomp}
(100\widehat{I_2})^c=\bigcup_{j\geq j_0}E_j\quad\text{and}\quad2^{j_0}={\ell(\widehat{I_2})\over\ell(I_2)},
\end{align}
and thus the dyadic decomposition of the domain $(100
\widehat{I_2})^c$ gives that
\begin{align*}
  |I_1|^{\frac{1}{2}}\int_{(100\widehat{I_2})^c}\left(\widetilde{(\mathcal{D}1)}\right)^{\frac{1}{2}}\,d{x_2}
    &\lesssim |I_1|^{\frac{1}{2}}\sum_{j\geq j_0}\int_{E_j}\left(  |x_2-c_{I_2}|^{-4N_2-m}\left\|{\triangle}_1^{ N_1 } {\triangle}_{twist} ^{ N_{3}}b_{R}\right\|^2_{L^2(\mathbb{R}^{2m})}    \right)^{\frac{1}{2}}\,d{x_2}\notag\\
  & \simeq \bigg({\ell(\widehat{I_2})\over\ell(I_2)}\bigg)^{\frac{m}{2}-2N_2}\cdot|R|^{\frac{1}{2}} \cdot\left\|{ {\triangle}_1^{ N_1 } {\triangle}_{twist} ^{ N_{3}}b_{R}}\right\|_{L^2(\mathbb{R}^{2m})}\cdot\ell^{-2N_2}_2. 
\end{align*}

$\blacksquare$ {\it Estimate on the term $\widetilde{(\mathcal{D}3)}$}:
$\widetilde{(\mathcal{D}3)}$ can be further decompose as
\begin{align*}
&\int_{\mathbb{R}^{m}}\int^{\ell_2 }_0\int^{\ell_2 }_0 \left(\chi_{B_{r_2}(0)}*\chi_{B_{r_3}(0)}\right
)(x_2-y_2)\cdot
\left\|a_R*_2 \varphi^{(2)}_{{r}_2 }*_3\varphi^{(3)}_{r_3}(\mathbf{y})\right\|^2_{L^2(\mathbb{R}^m,dy_1)}\,\frac { dr_2dr_3}{r^{m+1}_2r^{m+1}_3}\,d{y}_2\\
&\quad+\int_{\mathbb{R}^{m}} \int^\infty_{\ell_2 } \int^{\ell_2 }_0 \left(\chi_{B_{r_2}(0)}*\chi_{B_{r_3}(0)}\right
)(x_2-y_2)\cdot
\left\|a_R*_2 \varphi^{(2)}_{{r}_2 }*_3\varphi^{(3)}_{r_3}(\mathbf{y})\right\|^2_{L^2(\mathbb{R}^m,dy_1)}\,\frac { dr_2dr_3}{r^{m+1}_2r^{m+1}_3}\,d{y}_2\\
&=:\widetilde{(\mathcal{D}31)}+\widetilde{(\mathcal{D}32)}.
\end{align*}
\begin{itemize}
{\item \it Estimate on $\widetilde{(\mathcal{D}31)}$}:
\end{itemize}  
We analysis the convolution $a_R*_2 \varphi^{(2)}_{{r}_2 }*_3\varphi^{(3)}_{r_3}(\mathbf{y})$.
Since that
\begin{align*}
  \left| a_R*_2 \varphi^{(2)}_{{r}_2 }*_3\varphi^{(3)}_{r_3}(\mathbf{y})\right|&=\left|\int_{\mathbb{R}^m} a_R*_3\varphi^{(3)}_{r_3}(y_1,y_2-v)\cdot  \varphi^{(2)}_{{r}_2 }(v)\,dv\right|,
\end{align*}
then the support condition of the function $\varphi$ implies that for all $0\leq r_2,r_3\leq \ell_2$, we have $ |y_2-v-c_{I_2}|\leq\frac{3\ell_2}{2}$; and  the cone structure gives that 
 $|x_2-y_2|<2\ell_2$, then we have \begin{align*}
r_2\geq|v|&\geq|x_2-c_{I_2}|-|v+c_{I_2}-y_2|-|y_2-x_2|
     \geq|x_2-c_{I_2}|-\frac{7\ell_2}{2},
 \end{align*} 
 which contradicts to $x_2\in(100\widehat{I_2})^c$ and hence the term $\widetilde{(\mathcal{D}31)}$ has no contribution.
\begin{itemize}
      {\item \it Estimate on $\widetilde{(\mathcal{D}32)}$}:
   \end{itemize}
We further decompose $\widetilde{(\mathcal{D}32)}$ into
\begin{align*}
&\int_{\mathbb{R}^{m}}\int^\infty_{\frac{|x_2-c_{I_2}|}{4}} \int^{\ell_2 }_0 \left(\chi_{B_{r_2}(0)}*\chi_{B_{r_3}(0)}\right
)(x_2-y_2)\cdot
\left\|a_R*_2 \varphi^{(2)}_{{r}_2 }*_3\varphi^{(3)}_{r_3}(\mathbf{y})\right\|^2_{L^2(\mathbb{R}^m,dy_1)}\,\frac { dr_2dr_3}{r^{m+1}_2r^{m+1}_3}\,d{y}_2\\
&\quad+\int_{\mathbb{R}^{m}}\int^{\frac{|x_2-c_{I_2}|}{4}}_{\ell_2 } \int^{\ell_2 }_0 \left(\chi_{B_{r_2}(0)}*\chi_{B_{r_3}(0)}\right
)(x_2-y_2)\cdot
\left\|a_R*_2 \varphi^{(2)}_{{r}_2 }*_3\varphi^{(3)}_{r_3}(\mathbf{y})\right\|^2_{L^2(\mathbb{R}^m,dy_1)}\,\frac { dr_2dr_3}{r^{m+1}_2r^{m+1}_3}\,d{y}_2\\
&=:I(\widetilde{(\mathcal{D}32)})+II(\widetilde{(\mathcal{D}32)}).
\end{align*}
For the term $I(\widetilde{(\mathcal{D}32)})$, by applying the trivial estimate to the convolution $\chi_{B_{r_2}(0)}*\chi_{B_{r_3}(0)}$, Littlewood--Paley's inequality to eliminate $\varphi^{(2)}_{{r}_2 }$ and Young's convolution inequality, this term is dominated by
\begin{align*}
\int^\infty_{\frac{|x_2-c_{I_2}|}{4}}  \left\|a_R*_3 \varphi^{(3)}_{{r}_3 }\right\|^2_{L^2(\mathbb{R}^{2m})}\,\frac { dr_3}{r^{m+1}_3}\lesssim|x_2-c_{I_2}|^{-(m+4N_3)}\left\|\triangle^{N_1}_1\triangle^{N_2}_2b_{R}\right\|^2_{L^2(\mathbb{R}^{2m})}.
\end{align*}
Similar to the estimate on $\widetilde{(\mathcal{D}1)}$, by adapting the dyadic decomposition (\ref{dyadicdecomp}) to $(100\widehat{I_2})$, we have
   \begin{align*}
    &|I_1|^{\frac{1}{2}}\int_{(100\widehat{I_2})^c}\left(I(\widetilde{(\mathcal{D}32)})\right)^{\frac{1}{2}}\,d{x_2}\lesssim \bigg({\ell(\widehat{I_2})\over\ell(I_2)}\bigg)^{\frac{m}{2}-2N_3}\cdot|R|^{\frac{1}{2}}\cdot \ell^{-2N_3}_2\left\|\triangle^{N_1}_1\triangle^{N_2}_2b_{R}\right\|_{L^2(\mathbb{R}^{2m})}.
\end{align*}

For the term $II(\widetilde{(\mathcal{D}32)})$, similar to $\widetilde{(\mathcal{D}31)}$, 
by analyzing the support condition from cone and the functions $\varphi^{(2)}_{{r}_2 }$ and $\varphi^{(3)}_{{r}_3}$, this term would vanish.
\smallskip

 $\blacksquare$ {\it Estimate on the term $\widetilde{(\mathcal{D}2)}$}: We  split this term into $\widetilde{(\mathcal{D}21)}$, $\widetilde{(\mathcal{D}22)}$ and $\widetilde{(\mathcal{D}23)}$.
\begin{align*}
\widetilde{(\mathcal{D}2)}&=\int_{\mathbb{R}^{m}}\int_{\mathbb{R}_+}\int^{\frac{|x_2-c_{I_2}|}{4}}_{\ell_2}\left(\chi_{B_{r_2}(0)}*\chi_{B_{r_3}(0)}\right
)(x_2-y_2)\cdot
\left\|a_R*_2 \varphi^{(2)}_{{r}_2 }*_3\varphi^{(3)}_{r_3}(\mathbf{y})\right\|^2_{L^2(\mathbb{R}^m,dy_1)}\,\frac { dr_2dr_3}{r^{m+1}_2r^{m+1}_3}\,d{y}_2\\
&=\int_{ \mathbb{R}^{m}}\int^{\ell_2}_0\int^{\frac{|x_2-c_{I_2}|}{4}}_{\ell_2}\left(\chi_{B_{r_2}(0)}*\chi_{B_{r_3}(0)}\right
)(x_2-y_2)\cdot
\left\|a_R*_2 \varphi^{(2)}_{{r}_2 }*_3\varphi^{(3)}_{r_3}(\mathbf{y})\right\|^2_{L^2(\mathbb{R}^m,dy_1)}\,\frac { dr_2dr_3}{r^{m+1}_2r^{m+1}_3}\,d{y}_2\\
&\quad+\int_{ \mathbb{R}^{m}}\int^{\frac{|x_2-c_{I_2}|}{4}}_{\ell_2}\int^{\frac{|x_2-c_{I_2}|}{4}}_{\ell_2}\left(\chi_{B_{r_2}(0)}*\chi_{B_{r_3}(0)}\right
)(x_2-y_2)\cdot
\left\|a_R*_2 \varphi^{(2)}_{{r}_2 }*_3\varphi^{(3)}_{r_3}(\mathbf{y})\right\|^2_{L^2(\mathbb{R}^m,dy_1)}\,\frac { dr_2dr_3}{r^{m+1}_2r^{m+1}_3}\,d{y}_2\\
&\quad+\int_{ \mathbb{R}^{m}}\int^\infty_{\frac{|x_2-c_{I_2}|}{4} }\int^{\frac{|x_2-c_{I_2}|}{4}}_{\ell_2}\left(\chi_{B_{r_2}(0)}*\chi_{B_{r_3}(0)}\right
)(x_2-y_2)\cdot
\left\|a_R*_2 \varphi^{(2)}_{{r}_2 }*_3\varphi^{(3)}_{r_3}(\mathbf{y})\right\|^2_{L^2(\mathbb{R}^m,dy_1)}\,\frac { dr_2dr_3}{r^{m+1}_2r^{m+1}_3}\,d{y}_2\\
&=:\widetilde{(\mathcal{D}21)}+\widetilde{(\mathcal{D}22)}+\widetilde{(\mathcal{D}23)}.
\end{align*}
By investigating the support condition given by cone structure and the functions $\varphi^{(2)}_{{r}_2 }$ and $\varphi^{(3)}_{{r}_3}$, one can show that $\widetilde{(\mathcal{D}21)}=0$; besides, being parallel to the estimate on $I(\widetilde{(\mathcal{D}32)})$, we can verify that
\begin{align*}
    &|I_1|^{\frac{1}{2}}\int_{(100\widehat{I_2})^c}\left(\widetilde{(\mathcal{D}23)}\right)^{\frac{1}{2}}\,d{x_2}\lesssim \bigg({\ell(\widehat{I_2})\over\ell(I_2)}\bigg)^{\frac{m}{2}-2N_2}\cdot|R|^{\frac{1}{2}}\cdot \ell^{-2N_2}_2\left\|\triangle^{N_1}_1\triangle^{N_2}_2b_{R}\right\|^2_{L^2(\mathbb{R}^{2m})}.
\end{align*}

Now, it remains to estimate the term $\widetilde{(\mathcal{D}22)}$. We will excavate the pointwise estimate given by the functions $\varphi^{(2)}_{{r}_2 }$ and $\varphi^{(3)}_{{r}_3}$.
    Since  $\ell_2\leq r_2,\,r_3\leq\frac{|x_2-c_{I_2}|}{4}$ and 
   \begin{align*}
a_R*_2 \varphi^{(2)}_{{r}_2 }*_3\varphi^{(3)}_{r_3}(\mathbf{y})=r^{-2N_2}_2
\int_{\mathbb{R}^m}\left(\triangle^{N_1}_1\triangle^{N_3}_3b_R*_3\varphi^{(3)}_{r_3}\right)(y_1,y_2-v)\cdot (\triangle^{N_2}_2\varphi^{(2)})_{r_2}(v)\,dv,
   \end{align*}
then we have $|y_2-v-c_{I_2}|\leq\frac{\ell_2}{2}+r_3$ and $|v|\leq\frac{|x_2-c_{I_2}|}{4}$. Besides, the cone structure gives that $|x_2-y_2|\leq \frac{|x_2-c_{I_2}|}{2}$. Therefore, we have
\begin{align*}
|v|&\geq |x_2-c_{I_2}|-|y_2-x_2|-|v-y_2+c_{I_2}|
       \geq\frac{6}{25}|x_2-c_{I_2}|,
 \end{align*}   
 where the last inequality follows from the fact that
$|x_2-c_{I_2}|\geq50\cdot\ell_2$, and hence we have $|v|\simeq |x_2-c_{I_2}|$. As a result, one has that
\begin{align*}
&\left|a_R*_2 \varphi^{(2)}_{{r}_2 }*_3\varphi^{(3)}_{r_3}(\mathbf{y})\right|^2\lesssim\frac{r^{2(M_2-m {-2N_2})}_2}{\left(r_2+|x_2-c_{I_2}|\right)^{2M_2}}\left(\int_{|v|\simeq|x_2-c_{I_2}|}\left|\triangle^{N_1}_1\triangle^{N_3}_3b_R*_3\varphi^{(3)}_{r_3}(y_1,y_2-v)\right| \,dv\right)^2\\
&=\frac{r^{2(M_2-m{-2N_2})}_2}{\left(r_2+|x_2-c_{I_2}|\right)^{2M_2}}\cdot r^{-4N_3}_3\left(\int_{|v|\simeq|x_2-c_{I_2}|}\left|\triangle^{N_1}_1 b_R*_3\left(\triangle^{N_3}_3\varphi^{(3)}\right)_{r_3}(y_1,y_2-v)\right| \,dv\right)^2\\
&=\frac{r^{2(M_2-m{-2N_2})}_2}{\left(r_2+|x_2-c_{I_2}|\right)^{2M_2}}\cdot r^{-4N_3}_3\left(\int_{|v|\simeq|x_2-c_{I_2}|}\left|\int_{\mathbb{R}^m}\triangle^{N_1}_1 b_R(y_1-z,y_2-v-z)\cdot\left(\triangle^{N_3}_3\varphi^{(3)}\right)_{r_3}(z)\,dz\right| \,dv\right)^2,
\end{align*}
for some $M_2> m$, which will be decided later.

Next, we investigate the support condition of $z$. Since that $|v|\simeq|x_2-c_{I_2}|$, $|y_2-v-z-c_{I_2}|\leq\frac{\ell_2}{2}$ and the cone structure gives that $|x_2-y_2|\leq\frac{|x_2-c_{I_2}|}{2}$, we then have
\begin{align*}
|z|&\geq|x_2-c_{I_2}|-|y_2-v-z-c_{I_2}|-|y_2-x_2|-|v|\geq\frac{6}{25}|x_2-c_{I_2}|,
\end{align*}
which leads to $|z|\simeq |x_2-c_{I_2}|$. Therefore, the Poisson type upper bound and the cancellation (the corresponding  decay) lead to
\begin{align*}
\left|a_R*_2 \varphi^{(2)}_{{r}_2 }*_3\varphi^{(3)}_{r_3}(\mathbf{y})\right|^2&\lesssim \frac{r^{2(M_2-m{-2N_2})}_2}{\left(r_2+|x_2-c_{I_2}|\right)^{2M_2}}\cdot \frac{r^{2(M_3-2N_3)}_3}{\left(r_3+|x_2-c_{I_2}|\right)^{2m+2M_3}}\\
&\quad\times\left(\int_{\substack{|v|\simeq|x_2-c_{I_2}|\\|z|\simeq|x_2-c_{I_2}|}}\left|\triangle^{N_1}_1b_R(y_1-z,y_2-v-z)\right|\,dz \,dv\right)^2,
\end{align*}
for some $M_3>0$. Thus, by plug the pointwise estimate, trivial estimate to the cone structure $\chi_{B_{r_2}(0)}*\chi_{B_{r_3}(0)}$ and then using Minkowski's inequality, we have
\begin{align*}
\widetilde{(\mathcal{D}22)}&\lesssim\int^{\frac{|x_2-c_{I_2}|}{4}}_{\ell_2}\int^{\frac{|x_2-c_{I_2}|}{4}}_{\ell_2} \left(\frac{r^{2(M_2-m{-2N_2})}_2}{\left(r_2+|x_2-c_{I_2}|\right)^{2M_2}}\cdot \frac{r^{2M_3-4N_3}_3}{\left(r_3+|x_2-c_{I_2}|\right)^{2m+2M_3}}\right)\\
&\qquad\times\Bigg[\int_{\substack{|v|\simeq|x_2-c_{I_2}|\\|z|\simeq|x_2-c_{I_2}|}}\left\|\triangle^{N_1}_1b_R\right\|_{L^2(\mathbb{R}^{2m})}\,dv\,dz\Bigg]^2\,\frac{ dr_2dr_3}{r_2r^{m+1}_3}\\
&\simeq \int^{\frac{|x_2-c_{I_2}|}{4}}_{\ell_2}\int^{\frac{|x_2-c_{I_2}|}{4}}_{\ell_2}  \left(\frac{r^{2(M_2-m-2N_2)}_2}{\left(r_2+|x_2-c_{I_2}|\right)^{2M_2}}\cdot \frac{r^{2M_3-4N_3}_3}{\left(r_3+|x_2-c_{I_2}|\right)^{2m+2M_3}}\right)\\
&\qquad\times\left\|\triangle^{N_1}_1b_R\right\|^2_{L^2(\mathbb{R}^{2m})}|x_2-c_{I_2}|^{4m}\,\frac{ dr_2dr_3}{r_2r^{m+1}_3}\\
&{\lesssim\int^{\frac{|x_2-c_{I_2}|}{4}}_{\ell_2}\int^{\frac{|x_2-c_{I_2}|}{4}}_{\ell_2} \left(\frac{\ell_2^{2(M_2-m-2N_2)}}{\left(\ell_2+|x_2-c_{I_2}|\right)^{2M_2}}\cdot \frac{\ell_3^{2M_3-4N_3}}{\left(\ell_2+|x_2-c_{I_2}|\right)^{2m+2M_3}}\right)}\\
&{\qquad\times\left\|\triangle^{N_1}_1b_R\right\|^2_{L^2(\mathbb{R}^{2m})}|x_2-c_{I_2}|^{4m}\,\frac{ dr_2dr_3}{r_2r^{m+1}_3}},
\end{align*}
where { we choose $M_2,M_3$ such that $0<M_2-m<2N_2$ and
$m<M_3<2N_3.$}

 To wrap up,  by our choice of $M_2$ and $M_3$ above, we obtain that $M_2+M_3>2m$. Then by adapting the dyadic decomposition (\ref{dyadicdecomp}), we obtain that 
 \begin{align*}
    &|I_1|^{\frac{1}{2}}\int_{(100\widehat{I_2})^c}\left(\widetilde{(\mathcal{D}22)}\right)^{\frac{1}{2}}\,d{x_2}\notag\\
    &\lesssim|I_1|^{\frac{1}{2}} {\ell_2^{-2N_2}\ell_2^{-2N_3}}\left\|\triangle^{N_1}_1b_R\right\|_{L^2(\mathbb{R}^{2m})}\sum_{j\geq j_0}{(2^j\ell_2)^{3m}}\Bigg[\frac{\ell^{2(M_2-m)}_2}{\left(2^j\ell_2\right)^{2M_2}}\times \frac{\ell^{2M_3-m}_2}{\left(2^j\ell_2\right)^{2m+2M_3}}\Bigg]^{\frac{1}{2}}\notag\\
      &=|R|^{\frac{1}{2}}\cdot\left({\ell^{-2N_2}_2\ell^{-2N_3}_2}\left\|\triangle^{N_1}_1b_R\right\|_{L^2(\mathbb{R}^{2m})}\right)\sum_{j\geq j_0}{2^{j(2m-M_2-M_3)}}\notag\\
        &=|R|^{\frac{1}{2}}\cdot\left({\ell^{-2N_2}_2\ell^{-2N_3}_2}\left\|\triangle^{N_1}_1b_R\right\|_{L^2(\mathbb{R}^{2m})}\right)\cdot \bigg({\ell(\widehat{I_2})\over\ell(I_2)}\bigg)^{-(M_2+M_3-2m)}.
\end{align*}
As a consequence, by combining the covering Lemma for type ${\rm I}$ and the property of the atoms, we complete the estimate on $(\mathcal{D})$. 
\medskip

$\bullet$ {\it Case $(\mathcal{C})$}:
 To estimate the term $(\mathcal{C})$, we first apply the dyadic decompose to the integral area and then applying H\"older's inequality to get 
\begin{align*}
(\mathcal{C})&= \int_{(100{I_1})^c\times (100\widehat{I_2})^c}\ S_{area,\,\varphi}(a_R)(\mathbf{x})\,d\mathbf{x}
\leq\sum^\infty_{i=6}\sum^\infty_{j=j_0}\int_{\substack{|x_1-c_{I_1}|\simeq2^i\ell_1\\|x_2-c_{I_2}|\simeq2^j\ell_2}}\ S_{area,\,\varphi}(a_R)(\mathbf{x})\,d\mathbf{x}\notag\\
&\lesssim |R|^{\frac{1}{2}}\sum^\infty_{i=6}\sum^\infty_{j=j_0}2^{\frac{(i+j)m}{2}}\cdot\left(\int_{\substack{|x_1-c_{I_1}|\simeq2^i\ell_1\\|x_2-c_{I_2}|\simeq2^j\ell_2}}\left|S_{area,\,\varphi}(a_R)(\mathbf{x})\right|^2\,d\mathbf{x}\right)^{\frac{1}{2}},
\end{align*}
where $j_0\in\mathbb{N}$ is such that $2^{j_0m}|I_2|\simeq{|\widehat{I_2}|}{}$.
 Besides, the $L^2$-norm of $S_{area,\,\varphi}(a_R)(\mathbf{x})$ over the dyadic region can be further decomposed as
 \begin{align*}
&\int_{\substack{|x_1-c_{I_1}|\simeq2^i\ell_1\\|x_2-c_{I_2}|\simeq2^j\ell_2}}\left|S_{area,\,\varphi}(a_R)(\mathbf{x})\right|^2\,d\mathbf{x}\\
&=\int_{\substack{|x_1-c_{I_1}|\simeq2^i\ell_1\\|x_2-c_{I_2}|\simeq2^j\ell_2}}\int_{ \mathbb{R}^{2m}}\left(\int^{\ell_1}_0\int^
     {\ell_2}_0+\int^{\ell_1}_0\int^\infty_{\ell_2}+\int^\infty_{\ell_1}\int^
{\ell_2}_0+\int^\infty_{\ell_1}\int^\infty_
{\ell_2}\right)\int^\infty_0 \\
&\quad\quad|a_R* \varphi_{\mathbf{r} } |^2(\mathbf{y})\cdot\chi_{\mathbf{r} }(\mathbf{x}-\mathbf{y})\,\frac{dr_3 dr_2dr_1}{r_3r_2r_1}\,d\mathbf{y}\,d\mathbf{x}\\ 
    &=:(\mathcal{C}1)+{(\mathcal{C}2)+(\mathcal{C}3)}+(\mathcal{C}4).
 \end{align*}
\smallskip
$\blacksquare$ {\it Estimate on the term $(\mathcal{C}1)$}:

To estimate this term, we need to further split the range of $r_3$. 
 \begin{align*}
     (\mathcal{C}1) &=\int_{\substack{|x_1-c_{I_1}|\simeq2^i\ell_1\\|x_2-c_{I_2}|\simeq2^j\ell_2}}\int_{ \mathbb{R}^{2m}}\int^{\ell_1}_0\int^
     {\ell_2}_0\left(\int^{{{\ell_2}}}_0+ \int^{\frac{1}{4}\max\{|x_1-c_{I_1}|,\,|x_2-c_{I_2}|\} }_{{\ell_2}}+\int^\infty_{\frac{1}{4}\max\{|x_1-c_{I_1}|,\,|x_2-c_{I_2}|\}}\right)\\
     &\quad\quad|a_R* \varphi_{\mathbf{r} } |^2(\mathbf{y})\cdot\chi_{\mathbf{r} }(\mathbf{x}-\mathbf{y})\,\frac{ dr_3dr_2dr_1}{r_3r_2r_1}\,d\mathbf{y}\,d\mathbf{x}\\
    &=:(\mathcal{C}11)+(\mathcal{C}12)+(\mathcal{C}13).
 \end{align*}
 \begin{itemize}
     \item {\it Estimate on $(\mathcal{C}11)$}:
 \end{itemize}
 In this case, we have
 \begin{align*}
    (\mathcal{C}11)&=\int_{\substack{|x_1-c_{I_1}|\simeq2^i\ell_1\\|x_2-c_{I_2}|\simeq2^j\ell_2}}\int_{ \mathbb{R}^{2m}}\int^{\ell_1}_0\int^
     {\ell_2}_0\int^{{\ell_2}}_0 \left|\int_{\mathbb{R}^{{2m}}}\left(a_R*_3\varphi^{(3)}_{r_3}\right)(\mathbf{w})\left(\varphi^{(1)}_{r_1}\otimes\varphi^{(2)}_{r_2}\right) (y_1-w_1,y_2-w_2) \,d\mathbf{w}\right|^2\\
     &\quad\times\left[\int_{\mathbb{R}^{m}} \left(\chi_{B_{r_1}(0)}\otimes\chi_{B_{r_2}(0)}\right)(\mathbf{x}-\mathbf{y}-(z,z))\cdot\chi_{B_{r_3}(0)}(z)\,dz \right]\,\frac {  
    dr_3dr_2dr_1}{r^{m+1}_3r^{m+1}_2r^{m+1}_1}\,d\mathbf{y}\,d\mathbf{x}.
 \end{align*}
 
 We investigate the support condition given by the centralized support of $\varphi^{(2)}_2$. The support condition of $\left(a_R*_3\varphi^{(3)}_{r_3}\right)$ and $\varphi^{(2)}_{r_2}$ reveals that
 \begin{align*}
     |c_{I_2}-w_2|\leq \frac{3\ell_2}{2},\, |y_2-w_2|\leq \ell_2\implies|y_2-c_{I_2}|\leq\frac{5\ell_2}{2}.
 \end{align*}
 On the other hand, the support condition inherited from the cone structure gives that
 $|x_2-y_2|\leq 2\ell_2$. Then 
 \begin{align*}
     \ell_2\geq|y_2-w_2|&\geq|x_2-c_{I_2}|-   |c_{I_2}-w_2|-|y_2-x_2|
     \geq |x_2-c_{I_2}|-\frac{7\ell_2}{2}.
 \end{align*}
 which contradicts to the position of $x_2$. Therefore, the term $(\mathcal{C}11)$ vanishes.

\begin{itemize}
     \item {\it Estimate on $(\mathcal{C}12)$}:
 \end{itemize}
 Now, we turn to estimate term $(\mathcal{C}12)$, which is equal to
\begin{align*}
     &\int_{\substack{|x_1-c_{I_1}|\simeq2^i\ell_1\\|x_2-c_{I_2}|\simeq2^j\ell_2}}\int_{ \mathbb{R}^{2m}}\int^{\ell_1}_0\int^
     {\ell_2}_0\int^{\frac{1}{4}\max\{|x_1-c_{I_1}|,\,|x_2-c_{I_2}|\} }_{{\ell_2}} \left|\left(a_R*_3\varphi^{(3)}_{r_3}\right)(\mathbf{w})\left(\varphi^{(1)}_{r_1}\otimes\varphi^{(2)}_{r_2}\right) (y_1-w_1,y_2-w_2) \,d\mathbf{w}\right|^2\\
     &\quad\times\left[\int_{\mathbb{R}^{m}} \left(\chi_{B_{r_1}(0)}\otimes\chi_{B_{r_2}(0)}\right)(\mathbf{x}-\mathbf{y}-(z,z))\cdot\chi_{B_{r_3}(0)}(z)\,dz \right]\,\frac {  
    dr_3dr_2dr_1}{r^{m+1}_3r^{m+1}_2r^{m+1}_1}\,d\mathbf{y}\,d\mathbf{x}.
\end{align*}

Suppose that $|x_2-c_{I_2}|\geq|x_1-c_{I_1}|$, then being parallel to $(\mathcal{C}11)$, we can conclude that $(\mathcal{C}12)=0$. Now, assume that $|x_2-c_{I_2}|<|x_1-c_{I_1}|$. We now investigate the support condition given by the first variable.
By the cone structure, we have
$|x_1-y_1|\leq \ell_1+\frac{|x_1-c_{I_1}|}{4}$; moreover, the centralized condition gives that $|w_1-c_{I_1}|\leq \frac{\ell_1}{2}+\frac{|x_1-c_{I_1}|}{4}$. Therefore, we have that 
\begin{align*}
\ell_1\geq|y_1-w_1|&\geq|x_1-c_{I_1}|- |x_1-y_1|-|w_1-c_{I_1}|
   \geq \frac{|x_1-c_{I_1}|}{2}-\frac{3\ell_1}{2},
\end{align*}
combine with the facts that $\ell_1\leq\frac{|x_1-c_{I_1}|}{50}$,
 which is obviously a contradiction. So, in either case, $(\mathcal{C}12)=0$.
 
\begin{itemize}
     \item {\it Estimate on $(\mathcal{C}13)$}:
 \end{itemize}
For the term $(\mathcal{C}13)$. we interchange the integral and apply Littlewood--Paley inequality to eliminate the functions $\varphi^{(1)}_{r_1}$ and $\varphi^{(2)}_{r_2}$, and thus $(\mathcal{C}13)$ is dominated by
\begin{align*}
    \int^\infty_{\frac{1}{4}\max\{2^i\ell_1,\,2^j\ell_2\}}\left\|a_R*_3\varphi^{(3)}_{r_3} \right\|^2_{L^2(\mathbb{R}^{2m})}\,\frac{dr_3}{r_3},
\end{align*}
   which leads to
   \begin{align*}
    (\mathcal{C}13)&\lesssim \int^\infty_{\frac{1}{4}\max\{2^i\ell_1,\,2^j\ell_2\}} \left\|\left({ {\triangle}_1^{ N_1 } \triangle^{N_2}_{2}b_{R}}\right)\right\|^2_{L^2(\mathbb{R}^{m})}\,\frac{dr_3}{r^{4N_3+1}_3}
    \simeq \max\{2^i\ell_1,\,2^j\ell_2\}^{-4N_3}\left\|\left({ {\triangle}_1^{ N_1 } \triangle^{N_2}_{2}b_{R}}\right)\right\|^2_{L^2(\mathbb{R}^{2m})}.
   \end{align*}
   To wrap up, if we take $N_3=m$, then we have
   \begin{align*}
       |R|^{\frac{1}{2}}\sum^\infty_{i=6}\sum^\infty_{j=j_0}2^{\frac{(i+j)m}{2}}\cdot(\mathcal{C}13)^{\frac{1}{2}}
        &\lesssim |R|^{\frac{1}{2}}\sum^\infty_{i=6}\sum^\infty_{j=j_0}2^{\frac{(i+j)m}{2}}\cdot(2^i\ell_1)^{-N_3}\cdot(2^j\ell_2)^{-N_3}\left\|\left({ {\triangle}_1^{ N_1 } \triangle^{N_2}_{2}b_{R}}\right)\right\|_{L^2(\mathbb{R}^{2m})}\\
       &\leq |R|^{\frac{1}{2}}\cdot(\ell_1\ell_2)^{-N_3}\left\|\left({ {\triangle}_1^{ N_1 } \triangle^{N_2}_{2}b_{R}}\right)\right\|_{L^2(\mathbb{R}^{2m})}\cdot\bigg({\ell(\widehat{I_2})\over\ell(I_2)}\bigg)^{-\frac{m}{2}}.
   \end{align*}
\smallskip

$\blacksquare$ {\it Estimate on the terms $(\mathcal{C}2)$ and $(\mathcal{C}3)$}:

For the terms $(\mathcal{C}2)$ and $(\mathcal{C}3)$. We only need to, by symmetry, estimate $(\mathcal{C}2)$. By further splitting the range of $r_2$, we have that
\begin{align*}
(\mathcal{C}2)
&=\int_{\substack{|x_1-c_{I_1}|\simeq2^i\ell_1\\|x_2-c_{I_2}|\simeq2^j\ell_2}}\int_{ \mathbb{R}^{2m}}\int^{\ell_1}_0\left(\int^{\frac{|x_2-c_{I_2}|}{4}}_{\ell_2}+\int^\infty_{\frac{|x_2-c_{I_2}|}{4}}\right)\int^\infty_0 |a_R* \varphi_{\mathbf{r} } |^2(\mathbf{y})\cdot\chi_{\mathbf{r} }(\mathbf{x}-\mathbf{y})\,\frac{ dr_3dr_2dr_1}{r_3r_2r_1}\,d\mathbf{y}\,d\mathbf{x}\\ 
  &=:(\mathcal{C}21)+(\mathcal{C}22).
\end{align*}

\begin{itemize}
     \item {\it Estimate on $(\mathcal{C}21)$}:
 \end{itemize}
By further decompose with respect to $r_3$, one has that
\begin{align*}
(\mathcal{C}21)&=\int_{\substack{|x_1-c_{I_1}|\simeq2^i\ell_1\\|x_2-c_{I_2}|\simeq2^j\ell_2}}\int_{ \mathbb{R}^{2m}}\int^{\ell_1}_0\int^{\frac{|x_2-c_{I_2}|}{4}}_{\ell_2}\left(\int^{\max\big\{\frac{|x_1-c_{I_1}|}{4},\frac{|x_2-c_{I_2}|}{8}\big\}}_0+\int^\infty_{\max\big\{\frac{|x_1-c_{I_1}|}{4},\frac{|x_2-c_{I_2}|}{8}\big\}}\right)\\
&\quad\quad|a_R* \varphi_{\mathbf{r} } |^2(\mathbf{y})\cdot\chi_{\mathbf{r} }(\mathbf{x}-\mathbf{y})\,\frac{ dr_3dr_2dr_1}{r_3r_2r_1}\,d\mathbf{y}\,d\mathbf{x}\\
&=I(\mathcal{C}21)+II(\mathcal{C}21).
\end{align*}
We first estimate the term $I(\mathcal{C}21)$. Similar to the term $(\mathcal{C}12)$, one can verify that this term would vanish.
 
Next, we estimate the term $II(\mathcal{C}21)$. Base on the similar method as presented in estimation of $(\mathcal{C}13)$, 
\begin{align*}
       |R|^{\frac{1}{2}}\sum^\infty_{i=6}\sum^\infty_{j=j_0}2^{\frac{(i+j)m}{2}}\cdot\left(II(\mathcal{C}21)\right)^{\frac{1}{2}}
       &\simeq  |R|^{\frac{1}{2}}\cdot\left(\ell_1\ell_2\right)^{-N_3}\left\|\triangle^{N_1}_1\triangle^{N_2}_{2}b_R\right\|_{L^2(\mathbb{R}^{2m})}\cdot\bigg({\ell(\widehat{I_2})\over\ell(I_2)}\bigg)^{-(N_3-\frac{m}{2})},
   \end{align*}
  provided that $2N_3>m$. 
\smallskip
\
\begin{itemize}
     \item {\it Estimate on $(\mathcal{C}22)$}:
 \end{itemize}
 By further decompose with respect to $r_3$, one has that
\begin{align*}
(\mathcal{C}22)&=\int_{\substack{|x_1-c_{I_1}|\simeq2^i\ell_1\\|x_2-c_{I_2}|\simeq2^j\ell_2}}\int_{ \mathbb{R}^{2m}}\int^{\ell_1}_0\int^\infty_{\frac{|x_2-c_{I_2}|}{4}}\left(\int^{\frac{|x_1-c_{I_1}|}{4}}_0+\int^\infty_{\frac{|x_1-c_{I_1}|}{4}}\right)\\
&\qquad|a_R* \varphi_{\mathbf{r} } |^2(\mathbf{y})\cdot\chi_{\mathbf{r} }(\mathbf{x}-\mathbf{y})\,\frac{ dr_3dr_2dr_1}{r_3r_2r_1}\,d\mathbf{y}\,d\mathbf{x}\\
&=I(\mathcal{C}22)+II(\mathcal{C}22).
\end{align*}
Similar to the estimate of $(\mathcal{C}12)$, we can investigate the support condition given by the first variable and conclude that the term $I(\mathcal{C}22)$ would be equal to zero. As for the estimate of the term $II(\mathcal{C}22)$, interchanging the integral and apply Littlewood--Paley's inequality to eliminate $\varphi^{(1)}_{r_1}$, we have $II(\mathcal{C}22)$ is bounded above by
\begin{align*}
 \int^\infty_{\frac{2^j\ell_2}{4}}\int^\infty_{\frac{2^i\ell_1}{4}} \left\|a_R*_2\varphi^{(2)}_{r_2}*_3 \varphi^{(3)}_{{r}_3 } \right\|^2_{L^2(\mathbb{R}^{2m})}\,\frac {dr_3dr_2}{r_3r_2}.
\end{align*}
To continue, by applying Young's convolution inequality, we know that
\begin{align*}
 II(\mathcal{C}22)\lesssim\int^\infty_{\frac{2^j\ell_2}{4}}\int^\infty_{\frac{2^i\ell_1}{4}} \left\|\triangle^{N_1}_1b_R\right\|^2_{L^2(\mathbb{R}^{2m})}\,\frac {dr_3dr_2}{r^{4N_3+1}_3r^{4N_2+1}_2},
\end{align*}
and hence to wrap up, take $N_2=N_3=m$, we have
\begin{align*}
       |R|^{\frac{1}{2}}\sum^\infty_{i=6}\sum^\infty_{j=j_0}2^{\frac{(i+j)m}{2}}\cdot\left(II(\mathcal{C}22)\right)^{\frac{1}{2}}
       &\lesssim  |R|^{\frac{1}{2}}\sum^\infty_{i=6}\sum^\infty_{j=j_0}2^{{i(\frac{m}{2}-N_2)}} 2^{{j(\frac{m}{2}-N_3)}} \left(\ell_1\ell_2\right)^{-2m}\left\|\triangle^{N_1}_1b_R\right\|_{L^2(\mathbb{R}^{2m})}\\
       &\simeq  |R|^{\frac{1}{2}}\cdot\left(\ell^{-2N_2}_1\ell^{-2N_2}_2\right)\left\|\triangle^{N_1}_1b_R\right\|_{L^2(\mathbb{R}^{2m})}\cdot\bigg({\ell(\widehat{I_2})\over\ell(I_2)}\bigg)^{-\frac{m}{2}}.
   \end{align*}
  So, we complete the estimate on $(\mathcal{C}2)$.
\smallskip

$\blacksquare$ {\it Estimate on the term $(\mathcal{C}4)$}:

 Consider that the following decomposition
 \begin{align*}
(\mathcal{C}4)&=\int_{\substack{|x_1-c_{I_1}|\simeq2^i\ell_1\\|x_2-c_{I_2}|\simeq2^j\ell_2}}\int_{ \mathbb{R}^{2m}}\Bigg(\int^\infty_{\frac{|x_1-c_{I_1}|}{4}}\int^\infty_
     {\frac{|x_2-c_{I_2}|}{4}}+\int^\infty_{\frac{|x_1-c_{I_1}|}{4}}\int^{\frac{|x_2-c_{I_2}|}{4}}_{\ell_2}+\int^{\frac{|x_1-c_{I_1}|}{4}}_{\ell_1}\int^\infty_
     {\frac{|x_2-c_{I_2}|}{4}}\\
     &\quad+\int^{\frac{|x_1-c_{I_1}|}{4}}_{\ell_1}\int^
     {\frac{|x_2-c_{I_2}|}{4}}_{\ell_2} \Bigg)\int^\infty_0|a_R* \varphi_{\mathbf{r} } |^2(\mathbf{y})\cdot\chi_{\mathbf{r} }(\mathbf{x}-\mathbf{y})\,\frac{ dr_3dr_2dr_1}{r_3r_2r_1}\,d\mathbf{y}\,d\mathbf{x}\\
    &=:(\mathcal{C}41)+(\mathcal{C}42)+(\mathcal{C}43)+(\mathcal{C}44).
 \end{align*}
 
\begin{itemize}
     \item {\it Estimate on $(\mathcal{C}41)$}:
 \end{itemize}
 By interchanging the integral and applying Littlewood--Paley inequality to eliminate the function $\varphi^{(3)}_{r_3}$, 
 \begin{align*}
     (\mathcal{C}41)
&\lesssim\int^\infty_{\frac{2^i\ell_1}{4}}\int^\infty_
     {\frac{2^j\ell_2}{4}}\left\|a_R*\left(\varphi^{(1)}_{r_1 }\otimes \varphi^{(2)}_{r_2 }\right)  \right\|^2_{L^2(\mathbb{R}^{2m})}\,\frac{dr_2dr_1}{r_2r_1}.
 \end{align*}
 Then if we apply the cancellation property from $\triangle^{N_1}_{1}$ and  $\triangle^{N_2}_{2}$ and
 take $N_1=N_2=m$, then 
  \begin{align*}
&|R|^{\frac{1}{2}}\sum^\infty_{i=6}\sum^\infty_{j=j_0}2^{\frac{(i+j)m}{2}}\cdot\left((\mathcal{C}41)\right)^{\frac{1}{2}}\simeq|R|^{\frac{1}{2}}\cdot \ell_1^{-2N_1}\ell_2^{-2N_2}\cdot\left\|\triangle^{N_3}_{twist}b_R\right\|_{L^2(\mathbb{R}^{2m})}\cdot\bigg({\ell(\widehat{I_2})\over\ell(I_2)}\bigg)^{-\frac{3m}{2}}.
   \end{align*}

\begin{itemize}
\item {\it Estimate on $(\mathcal{C}42)$}:
\end{itemize}
 We further split this term $(\mathcal{C}42)$.
 \begin{align*}
(\mathcal{C}42)&=\int_{\substack{|x_1-c_{I_1}|\simeq2^i\ell_1\\|x_2-c_{I_2}|\simeq2^j\ell_2}}\int_{ \mathbb{R}^{2m}}\int^\infty_{\frac{|x_1-c_{I_1}|}{4}}\int^{\frac{|x_2-c_{I_2}|}{4}}_{\ell_2}\left(\int^{\ell_2}_0+\int^{\frac{|x_2-c_{I_2}|}{4} }_{\ell_2}+\int^\infty_{\frac{|x_2-c_{I_2}|}{4}}\right) \\
&\qquad|a_R* \varphi_{\mathbf{r} } |^2(\mathbf{y})\cdot\chi_{\mathbf{r} }(\mathbf{x}-\mathbf{y})\,\frac{ dr_3dr_2dr_1}{r_3r_2r_1}\,d\mathbf{y}\,d\mathbf{x}\\
&=:I(\mathcal{C}42)+II(\mathcal{C}42)+III(\mathcal{C}42).
 \end{align*}
By investigating the second variable in $I(\mathcal{C}42)$, we then conclude that this term vanishes. For the estimate of $III(\mathcal{C}42)$, we can use the similar idea in estimating $(C41)$ to deduce that
\begin{align*}
    III(\mathcal{C}42)\lesssim \left(2^i\ell_1\right)^{-4N_1}\left(2^j\ell_2\right)^{-4N_2}\left\|{\triangle}_{twist}^{ N_3 } b_{R}\right\|^2_{L^2(\mathbb{R}^{2m})},
\end{align*}
and hence to wrap up, take $N_1=N_2=m$ and one has
\begin{align*}
|R|^{\frac{1}{2}}\sum^\infty_{i=6}\sum^\infty_{j=j_0}2^{\frac{(i+j)m}{2}}\cdot\left(III(\mathcal{C}42)\right)^{\frac{1}{2}}
\lesssim|R|^{\frac{1}{2}}\cdot\ell_1^{-2N_1}\ell_2^{-2N_2}\left\|{\triangle}_{twist}^{ N_3 } b_{R}\right\|_{L^2(\mathbb{R}^{2m})}\cdot\bigg({\ell(\widehat{I_2})\over\ell(I_2)}\bigg)^{-\frac{3m}{2}}.
\end{align*}

Now we estimate the term $II(\mathcal{C}42)$. We first interchanging of integral, use the cancellation property of $\triangle^{N_1}_1$, and apply Young's convolution inequality to eliminate the factor $\varphi^{(1)}_{r_1}$, then
\begin{align}\label{tttttttttttt1}
II(\mathcal{C}42)&\lesssim\int_{|x_2-c_{I_2}|\simeq2^j\ell_2}\int_{ \mathbb{R}^{m}}\int^{\frac{|x_2-c_{I_2}|}{4}}_{\ell_2}\int^{\frac{|x_2-c_{I_2}|}{4} }_{\ell_2}\int^\infty_{\frac{2^i\ell_1}{4}}\left\|a_R*\varphi_{\mathbf{r}}(y_1,\,y_2)\right\|^2_{L^2(\mathbb{R}^{m},\,d{y}_1)}\,\frac{dr_1}{r_1}\notag\\
  &\qquad\times\left(\chi_{B_{r_2}(0)}*\chi_{B_{r_3}(0)}\right)(x_2-y_2)\,\frac {  
    dr_3dr_2}{r^{m+1}_3r^{m+1}_2}\,d{y}_2\,dx_2\notag\\
      &\lesssim\int_{|x_2-c_{I_2}|\simeq2^j\ell_2}\int_{ \mathbb{R}^{m}}\int^{\frac{|x_2-c_{I_2}|}{4}}_{\ell_2}\int^{\frac{|x_2-c_{I_2}|}{4} }_{\ell_2}\Bigg[\int^\infty_{\frac{2^i\ell_1}{4}}\left\|\left({ {\triangle}_2^{ N_2 } \triangle^{N_3}_{twist}b_{R}}*_2\varphi^{(2)}_{r_2}*_3\varphi^{(3)}_{r_3}\right)(\mathbf{y})\right\|^2_{L^2(\mathbb{R}^{m},\,d{y}_1)}\notag\\
      &\qquad\frac{dr_1}{r^{1+4N_1}_1}\Bigg]\times\left(\chi_{B_{r_2}(0)}*\chi_{B_{r_3}(0)}\right)(x_2-y_2)\,\frac {  
    dr_3dr_2}{r^{m+1}_3r^{m+1}_2}\,d{y}_2\,dx_2\notag\\
    &\simeq(2^i\ell_1)^{-4N_1}\int_{|x_2-c_{I_2}|\simeq2^j\ell_2}\int_{ \mathbb{R}^{2m}}\int^{\frac{|x_2-c_{I_2}|}{4}}_{\ell_2}\int^{\frac{|x_2-c_{I_2}|}{4} }_{\ell_2}\left| \left({ {\triangle}_2^{ N_2 } \triangle^{N_3}_{twist}b_{R}}*_2\varphi^{(2)}_{r_2}*_3\varphi^{(3)}_{r_3}\right)(\mathbf{y}) \right|^2\notag\\
      &\qquad\times\left(\chi_{B_{r_2}(0)}*\chi_{B_{r_3}(0)}\right)(x_2-y_2)\,\frac {  
    dr_2dr_3}{r^{m+1}_2r^{m+1}_3}\,d\mathbf{y}\,dx_2.
\end{align}
To estimate (\ref{tttttttttttt1}), we will utilize the support condition from the cone and dig out the pointwise estimate. Since
\begin{align*}
    \left| \left({ {\triangle}_2^{ N_2 } \triangle^{N_3}_{twist}b_{R}}*_2\varphi^{(2)}_{r_2}*_3\varphi^{(3)}_{r_3}\right)(\mathbf{y}) \right|^2=\left|\int_{\mathbb{R}^m} \left( {\triangle}_2^{ N_2 } \triangle^{N_3}_{twist}b_{R}*_3\varphi^{(3)}_{r_3}\right)(y_1,y_2-v)\cdot\varphi^{(2)}_{r_2}(v)\,dv\right|^2,
\end{align*}
we obtain that $|y_2-v-c_{I_2}|\leq \frac{\ell_2}{2}+ \frac{|x_2-c_{I_2}|}{4}$ and $|v|\leq\frac{|x_2-c_{I_2}|}{4}$; and from the cone structure in (\ref{tttttttttttt1}), we have that $|x_2-y_2|\leq \frac{|x_2-c_{I_2}|}{2}$, which leads to
\begin{align*}
    |v|&\geq |x_2-c_{I_2}|-|y_2-v-c_{I_2}|-|y_2-x_2|
    \geq \frac{|x_2-c_{I_2}|}{4}-\frac{\ell_2}{2}
    \geq\frac{6|x_2-c_{I_2}|}{25},
\end{align*}
where the last inequality can be deduced from the support condition of $x_2$, and hence $|v|\simeq|x_2-c_{I_2}|$. Thus, we can deduce that 
\begin{align*}
    &\left| \left({ {\triangle}_2^{ N_2 } \triangle^{N_3}_{twist}b_{R}}*_2\varphi^{(2)}_{r_2}*_3\varphi^{(3)}_{r_3}\right)(\mathbf{y}) \right|^2\\&=r^{-4N_2}_2\left|\int_{\mathbb{R}^m} \left( \triangle^{N_3}_{twist}b_{R}*_3\varphi^{(3)}_{r_3}\right)(y_1,y_2-v)\cdot\left({\triangle}_2^{ N_2 } \varphi^{(2)}\right)_{r_2}(v)\,dv\right|^2\\
    &\lesssim\frac{r^{2M_2-4N_2}_2}{\left(r_2+|x_2-c_{I_2}|\right)^{2m+2M_2}}\cdot \left(\int_{|v|\simeq|x_2-c_{I_2}|} \Big|\left( \triangle^{N_3}_{twist}b_{R}*_3\varphi^{(3)}_{r_3}\right)(y_1,y_2-v)\Big|\,dv\right)^2
\end{align*}
for some $M_2>0$. So, $II(\mathcal{C}42)$ is bounded above by
\begin{align*}
      &(2^i\ell_1)^{-4N_1}\int_{|x_2-c_{I_2}|\simeq2^j\ell_2}\int_{ \mathbb{R}^{2m}}\int^{\frac{|x_2-c_{I_2}|}{4}}_{\ell_2}\int^{\frac{|x_2-c_{I_2}|}{4} }_{\ell_2}\frac{r^{2M_2-4N_2}_2}{\left(r_2+|x_2-c_{I_2}|\right)^{2m+2M_2}}\\
&\times\left(\int_{|v|\simeq|x_2-c_{I_2}|} \Big|\left( \triangle^{N_3}_{twist}b_{R}*_3\varphi^{(3)}_{r_3}\right)(y_1,y_2-v)\Big|\,dv\right)^2\cdot\left(\chi_{B_{r_2}(0)}*\chi_{B_{r_3}(0)}\right)(x_2-y_2)\,\frac {  
    dr_2dr_3}{r^{m+1}_2r^{m+1}_3}\,d\mathbf{y}\,dx_2.
\end{align*}
Next, we will explore the Poisson type estimate given by $\varphi^{(3)}_{r_3}$. Consider that
\begin{align*}
     &\left( \triangle^{N_3}_{twist}b_{R}*_3\varphi^{(3)}_{r_3}\right)(y_1,y_2-v)=r^{-2N_3}_3\cdot\int_{\mathbb{R}^m}b_{R}(y_1-z,y_2-v-z)\cdot\left(\triangle^{N_3}_{3}\varphi^{(3)}\right)_{r_3}(z)\,dz.
\end{align*}
Since  $|v|\leq\frac{|x_2-c_{I_2}|}{4}$, $|y_2-v-z-c_{I_2}|\leq\frac{\ell_2}{2}$ and the cone structure gives that $|x_2-y_2|\leq\frac{|x_2-c_{I_2}|}{2}$, we then have
\begin{align*}
    \frac{1}{4}|x_2-c_{I_2}|\geq|z|&\geq|x_2-c_{I_2}|-|y_2-v-z-c_{I_2}|-|y_2-x_2|-|v|\geq\frac{1}{4}|x_2-c_{I_2}|-\frac{\ell_2}{2},
\end{align*}
which leads to $|z|\simeq |x_2-c_{I_2}|$. Then the Poisson type bound gives that
\begin{align*}
     &\left|  \triangle^{N_3}_{twist}b_{R}*_3\varphi^{(3)}_{r_3}\right|(y_1,y_2-v)
     \lesssim\frac{r^{M_3-2N_3}_3}{\left(r_3+|x_2-c_{I_2}|\right)^{m+M_3}}\int_{\mathbb{R}^m}\left|b_{R}(y_1-z,y_2-v-z)\right|\,dz,
\end{align*}
 where $M_3$ is a positive integer, which further implies that the term $II(\mathcal{C}42)$ is dominated by
 \begin{align*}
&(2^i\ell_1)^{-4N_1}\int_{|x_2-c_{I_2}|\simeq2^j\ell_2}\int_{ \mathbb{R}^{2m}}\int^{\frac{|x_2-c_{I_2}|}{4}}_{\ell_2}\int^{\frac{|x_2-c_{I_2}|}{4} }_{\ell_2}\frac{r^{2M_2-4N_2}_2}{\left(r_2+|x_2-c_{I_2}|\right)^{2m+2M_2}}\cdot\frac{r^{2M_3-4N_3}_3}{\left(r_3+|x_2-c_{I_2}|\right)^{2m+2M_3}}\\
&\times\left[\int_{|v|\simeq|x_2-c_{I_2}|} \int_{|z|\simeq|x_2-c_{I_2}|} \left| b_{R}(y_1-z,y_2-v-z)\right|\,dz\,dv\right]^2\cdot\left(\chi_{B_{r_2}(0)}*\chi_{B_{r_3}(0)}\right)(x_2-y_2)\,\frac {  
    dr_2dr_3}{r^{m+1}_2r^{m+1}_3}\,d\mathbf{y}\,dx_2.
\end{align*}
To continue, by Minkowski's inequality and H\"older's inequality, one has
\begin{align*}
    &\int_{\mathbb{R}^m}\left[ \int_{|v|\simeq|x_2-c_{I_2}|} \int_{|z|\simeq|x_2-c_{I_2}|} \left| b_{R}(y_1-z,y_2-v-z)\right|\,dz\,dv\right]^2\,dy_1\\
    &\leq\left[ \int_{|z|\simeq|x_2-c_{I_2}|} \int_{|v|\simeq|x_2-c_{I_2}|}\left(\int_{\mathbb{R}^m}\left| b_{R}(y_1-z,y_2-v-z)\right|^2\,dy_1\right)^{\frac{1}{2}}\,dv\,dz\right]^2\\
     &\lesssim\left[ \int_{|z|\simeq|x_2-c_{I_2}|} \left(\int_{\mathbb{R}^m}\int_{\mathbb{R}^m}\left|b_{R}(y_1-z,y_2-v-z)\right|^2\,dy_1\,dv\right)^{\frac{1}{2}}\,dz\right]^2\cdot|x_2-c_{I_2}|^{m}\\
       &\simeq |x_2-c_{I_2}|^{3m}\cdot\left\|b_{R}\right\|^2_{L^2(\mathbb{R}^{2m})}.
\end{align*}
Therefore, if $M_2-2N_2<0$ and $M_3-2N_3<0$, the term $II(\mathcal{C}42)$ can be controlled by
 \begin{align*}
&(2^i\ell_1)^{-4N_1}\int_{|x_2-c_{I_2}|\simeq2^j\ell_2}\int_{ \mathbb{R}^{m}}\int^{\frac{|x_2-c_{I_2}|}{4}}_{\ell_2}\int^{\frac{|x_2-c_{I_2}|}{4} }_{\ell_2}\frac{r^{2M_2-4N_2}_2}{\left(r_2+|x_2-c_{I_2}|\right)^{2m+2M_2}}\cdot\frac{r^{2M_3-4N_3}_3}{\left(r_3+|x_2-c_{I_2}|\right)^{2m+2M_3}}\\
&\qquad\times|x_2-c_{I_2}|^{3m}\cdot\left\|b_{R}\right\|^2_{L^2(\mathbb{R}^{2m})}\cdot\left(\chi_{B_{r_2}(0)}*\chi_{B_{r_3}(0)}\right)(x_2-y_2)\,\frac {  
    dr_3dr_2}{r^{m+1}_3r^{m+1}_2}\,dy_2\,dx_2\\
&\lesssim(2^i\ell_1)^{-4N_1}\int_{|x_2-c_{I_2}|\simeq2^j\ell_2}\int^{\frac{|x_2-c_{I_2}|}{4}}_{\ell_2}\int^{\frac{|x_2-c_{I_2}|}{4} }_{\ell_2}\frac{r^{2M_2-4N_2}_2}{\left(r_2+|x_2-c_{I_2}|\right)^{2m+2M_2}}\cdot\frac{r^{2M_3-4N_3}_3}{\left(r_3+|x_2-c_{I_2}|\right)^{2m+2M_3}}\\
&\qquad\times|x_2-c_{I_2}|^{3m}\cdot\left\| b_{R}\right\|^2_{L^2(\mathbb{R}^{2m})}\,\frac {  dr_3dr_2}{r_3r_2}\,dx_2\\
&\simeq(2^i\ell_1)^{-4N_1}\left(2^j\ell_2\right)^{4m}\left\|b_{R}\right\|^2_{L^2(\mathbb{R}^{2m})}\times\Bigg[\int^{\frac{2^j\ell_2}{4}}_{\ell_2} \frac{r^{2M_2-4N_2}_2}{\left(r_2+2^j\ell_2\right)^{2m+2M_2}} \,\frac{dr_2}{r_2}\times\int^{\frac{2^j\ell_2}{4} }_{\ell_2}\frac{r^{2M_3-4N_3}_3}{\left(r_3+2^j\ell_2\right)^{2m+2M_3}}\frac{dr_3}{r_3}\Bigg]\\
&\lesssim(2^i\ell_1)^{-4N_1}\left(2^j\ell_2\right)^{4m}\left\|b_{R}\right\|^2_{L^2(\mathbb{R}^{2m})}\times\Bigg[ \frac{\ell^{2M_2-4N_2}_2}{\left(2^j\ell_2\right)^{2m+2M_2}} \times\frac{\ell^{2M_3-4N_3}_2}{\left(2^j\ell_2\right)^{2m+2M_3}}\Bigg]\\
&\simeq(2^i)^{-4N_1}\cdot(2^j)^{-2M_2-2M_3}\cdot\ell_1^{-4N_1}\ell^{-4N_2-4N_3}_2\left\|b_{R}\right\|^2_{L^2(\mathbb{R}^{2m})}.
\end{align*}
Finally, if we choose $\frac{m}{2}<2N_1$ and $\frac{m}{2}<M_2+M_3$, then we have
\begin{align*}
       |R|^{\frac{1}{2}}\sum^\infty_{i=6}\sum^\infty_{j=j_0}2^{\frac{(i+j)m}{2}}\cdot\left(II(\mathcal{C}42)\right)^{\frac{1}{2}}
       &\lesssim |R|^{\frac{1}{2}}\sum^\infty_{i=6}\sum^\infty_{j=j_0}2^{i(\frac{m}{2}-2N_1)}\cdot2^{j(\frac{m}{2}-M_2-M_3)}\cdot\ell_1^{-2N_1}\ell^{-2N_2-2N_3}_2\left\|b_{R}\right\|_{L^2(\mathbb{R}^{2m})}\\
       &\simeq|R|^{\frac{1}{2}}\cdot\ell_1^{-2N_1}\ell^{-2N_2-2N_3}_2\left\|b_{R}\right\|_{L^2(\mathbb{R}^{2m})}\cdot \bigg({\ell(\widehat{I_2})\over\ell(I_2)}\bigg)^{(\frac{m}{2}-M_2-M_3)}.
   \end{align*}

  \begin{itemize}
     \item {\it Estimate on $(\mathcal{C}43)$}:
 \end{itemize}
 Now, we estimate the term $(\mathcal{C}43)$. Decompose this term as the following
   \begin{align*}
     &\int_{\substack{|x_1-c_{I_1}|\simeq2^i\ell_1\\|x_2-c_{I_2}|\simeq2^j\ell_2}}\int_{ \mathbb{R}^{2m}}\int^{\frac{|x_1-c_{I_1}|}{4}}_{\ell_1}\int^\infty_
     {\frac{|x_2-c_{I_2}|}{4}}
     \Bigg(\int^{\frac{|x_1-c_{I_1}|}{8}}_0+\int^\infty_{\frac{|x_1-c_{I_1}|}{8}}\Bigg)\\
     &\qquad|a_R* \varphi_{\mathbf{r} } |^2(\mathbf{y})\cdot\chi_{\mathbf{r} }(\mathbf{x}-\mathbf{y})\,\frac{ dr_3dr_2dr_1}{r_3r_2r_1}\,d\mathbf{y}\,d\mathbf{x}\\
    &=:{I(\mathcal{C}43)+II(\mathcal{C}43)}.
 \end{align*}
 For $I(\mathcal{C}43)$, one can examine the support condition given by the first variable and conclude that this term will eventually vanish; while, the term $II(\mathcal{C}43)$ can be easily obtained by following the similar method used in estimating the term 
$III(\mathcal{C}42)$. 
\smallskip
 \begin{itemize}
     \item {\it Estimate on $(\mathcal{C}44)$}
 \end{itemize}
 
We consider the following decomposition
 \begin{align*}
(\mathcal{C}44)
    &=\int_{\substack{|x_1-c_{I_1}|\simeq2^i\ell_1\\|x_2-c_{I_2}|\simeq2^j\ell_2}}\int_{ \mathbb{R}^{2m}}\int^{\frac{|x_1-c_{I_1}|}{4}}_{\ell_1}\int^
     {\frac{|x_2-c_{I_2}|}{4}}_{\ell_2} \bigg[\int^{\ell_1}_0+ \int^{\max\big\{\frac{|x_1-c_{I_1}|}{8},\frac{|x_2-c_{I_2}|}{4} \big\}}_{\ell_1}\\
     &\qquad+\int^\infty_{\max\big\{\frac{|x_1-c_{I_1}|}{8},\frac{|x_2-c_{I_2}|}{4} \big\}}\bigg]\quad |a_R* \varphi_{\mathbf{r} } |^2(\mathbf{y})\cdot\chi_{\mathbf{r} }(\mathbf{x}-\mathbf{y})\,\frac{ dr_3dr_2dr_1}{r_3r_2r_1}\,d\mathbf{y}\,d\mathbf{x}\\&=:I(\mathcal{C}44)+II(\mathcal{C}44)+III(\mathcal{C}44).
 \end{align*}
 For the term $I(\mathcal{C}44)$, observe that
 \begin{align*}
     I(\mathcal{C}44)&=\int_{\substack{|x_1-c_{I_1}|\simeq2^i\ell_1\\|x_2-c_{I_2}|\simeq2^j\ell_2}}\int_{ \mathbb{R}^{2m}}\int^{\frac{|x_1-c_{I_1}|}{4}}_{\ell_1}\int^
     {\frac{|x_2-c_{I_2}|}{4}}_{\ell_2} \int^{\ell_1}_0 \left|\int_{\mathbb{R}^{{2m}}}\left(a_R*_3\varphi^{(3)}_{r_3}\right)(\mathbf{w})\left(\varphi^{(1)}_{r_1}\otimes\varphi^{(2)}_{r_2}\right) (\mathbf{y}-\mathbf{w}) \,d\mathbf{w}\right|^2\\
     &\qquad\times\left(\chi_{B_{r_1}(0)}\otimes\chi_{B_{r_2}(0)}\right)*_3\chi_{B_{r_3}(0)}(\mathbf{x}-\mathbf{y})\,\frac {  
    dr_3dr_2dr_1}{r^{m+1}_3r^{m+1}_2r^{m+1}_1}\,d\mathbf{y}\,d\mathbf{x}.
 \end{align*}
 If we investigate the support condition given by the first variable, then we have
\begin{align*}
\frac{|x_1-c_{I_1}|}{4}\geq|y_1-w_1|&\geq|x_1-c_{I_1}|- |x_1-y_1|-|w_1-c_{I_1}|
\geq\frac{3|x_1-c_{I_1}|}{4}-\frac{5\ell_1}{2},
\end{align*}
which contradicts to the support condition of $x_1$, and thus the term $I(\mathcal{C}44)$ vanishes.
\smallskip

For the term $II(C44)$, without loss of generality, we deal with the case that \begin{align}\label{constrain}
\frac{|x_1-c_{I_1}|}{8}\leq\frac{|x_2-c_{I_2}|}{4}.\end{align}
In order to derive the covering decay, we need to dig out the pointwise estimate from the function $\varphi$, remark that
\begin{align*}
|a_R* \varphi_{\mathbf{r} } |(\mathbf{y})
&=r^{-2N_2}_2r^{-2N_3}_3\left|\triangle^{N_1}_1b_R* \left(\varphi^{(1)}_{r_1}\otimes(\triangle^{N_2}_2\varphi^{(2)})_{r_2}\right)*_3(\triangle^{N_3}_3\varphi^{(3)})_{{r_3}} \right|(\mathbf{y})\\
&=  r^{-2N_2}_2r^{-2N_3}_3 \left|\int_{\mathbb{R}^{{2m}}}\left(\triangle^{N_1}_1b_R*_3(\triangle^{N_3}_3\varphi^{(3)})_{{r_3}}\right)(\mathbf{w})\left(\varphi^{(1)}_{r_1}\otimes(\triangle^{N_2}_2\varphi^{(2)})_{r_2}\right) (y_1-w_1,y_2-w_2) \,d\mathbf{w}\right|.
\end{align*}
  We investigate the support condition given by the second variable.
By the cone structure and (\ref{constrain}), we have
$|x_2-y_2|\leq \frac{|x_2-c_{I_2}|}{2}$; moreover, the support condition of $(\triangle^{N_3}_3\varphi^{(3)})_{{r_3}}$ gives that $|w_2-c_{I_2}|\leq \frac{\ell_2}{2}+\frac{|x_2-c_{I_2}|}{4}$. Therefore, we have
\begin{align*}
\frac{|x_2-c_{I_2}|}{4}\geq|y_2-w_2|&\geq|x_2-c_{I_2}|- |x_2-y_2|-|w_2-c_{I_2}|
   \geq\frac{6|x_2-c_{I_2}|}{25},
\end{align*} 
that is $|y_2-w_2|\simeq |x_2-c_{I_2}|$.
Next, since that
\begin{align*}
\left(\triangle^{N_1}_1b_R*_3(\triangle^{N_3}_3\varphi^{(3)})_{{r_3}}\right)(\mathbf{w})=\int_{\mathbb{R}^m} \triangle^{N_1}_1b_R(w_1-z,w_2-z)\cdot(\triangle^{N_3}_3\varphi^{(3)})_{{r_3}}(z)\,dz,
\end{align*}
then by investigating the support condition and the fact that
$\frac{|x_2-c_{I_2}|}{4}\geq |z|\geq|x_2-c_{I_2}|-|z+c_{I_2}-w_2|-|w_2-y_2|-|y_2-x_2|$, one can deduce $|z|\simeq|x_2-c_{I_2}|$. Thus by the pointwise estimate of Poisson type bound, $|a_R* \varphi_{\mathbf{r} } |(\mathbf{y})$ can be dominated by
\begin{align*}
&\int_{|w_1-y_1|\leq{|x_1-c_{I_1}|}}\bigg[\int_{\substack{|w_2-y_2|\simeq|x_2-c_{I_2}|\\|z|\simeq|x_2-c_{I_2}|}}\left|\triangle^{N_1}_1b_R(w_1-z,w_2-z)\right|\,dw_2\\
&\qquad\times\left|\varphi^{(1)}_{r_1}(y_1-w_1) \right|\,dz\bigg]\,dw_1\cdot\frac{r^{M_2-2N_2}_2}{(r_2+|x_2-c_{I_2}|)^{m+M_2}}\cdot\frac{r^{M_3-2N_3}_3}{(r_3+|x_2-c_{I_2}|)^{m+M_3}}\\
&\lesssim\int_{|z|\simeq|x_2-c_{I_2}|}\mathcal{M}_1\Bigg[\int_{{|w_2-y_2|\simeq|x_2-c_{I_2}|}}\left|\triangle^{N_1}_1b_R(\cdot,w_2-z)\right|\,dw_2\Bigg](y_1-z)\,dz\\
&\qquad\times\frac{r^{M_2-2N_2}_2}{(r_2+|x_2-c_{I_2}|)^{m+M_2}}\cdot\frac{r^{M_3-2N_3}_3}{(r_3+|x_2-c_{I_2}|)^{m+M_3}},
\end{align*}
in which $\mathcal{M}_1$ denotes the Hardy--Littlewood maximal function in the first coordinate. Then, by applying H\"older's inequality, integrating the cone structure $\chi_{\mathbf{r}}$ against $y_2$ and the trivial estimate of $\chi_{B_{r_1}(0)}*\chi_{B_{r_3}(0)}$, we have 
\begin{align*}
II(\mathcal{C}44)&\lesssim\int_{\substack{|x_1-c_{I_1}|\simeq2^i\ell_1\\|x_2-c_{I_2}|\simeq2^j\ell_2\\|x_1-c_{I_1}|\lesssim|x_2-c_{I_2}|}}\int_{\mathbb{R}^{m}}\int^{\frac{|x_1-c_{I_1}|}{4}}_{\ell_1}\int^
     {\frac{|x_2-c_{I_2}|}{4}}_{\ell_2} \int^{\frac{|x_2-c_{I_2}|}{4}}_{\ell_1}|x_2-c_{I_2}|^m\Bigg[\int_{|z|\simeq|x_2-c_{I_2}|}\\
&\qquad\bigg(\mathcal{M}_1\bigg[\int_{\mathbb{R}^m}\bigg|\triangle^{N_1}_1b_R(\cdot,w_2)\bigg|\,dw_2\bigg](y_1-z)\bigg)^2\,dz\Bigg]\frac{r^{2M_2-4N_2}_2}{(r_2+|x_2-c_{I_2}|)^{2m+2M_2}}\\
     &\qquad\times\frac{r^{2M_3-4N_3}_3}{(r_3+|x_2-c_{I_2}|)^{2m+2M_3}}\,\frac {  
    dr_3dr_2dr_1}{r_3 r_2r^{m+1}_1}\,dy_1\,d\mathbf{x}.
\end{align*}
To continue, by interchange of integral and $L^2$-boundedness of the Hardy--Littlewood maximal function, we have the term $II(\mathcal{C}44)$ is bounded above by
\begin{align*}
&\|\triangle^{N_1}_1b_R\|^2_{L^2(\mathbb{R}^{2m})}\cdot\ell^m_2\cdot\int_{\substack{|x_1-c_{I_1}|\simeq2^i\ell_1\\|x_2-c_{I_2}|\simeq2^j\ell_2\\|x_1-c_{I_1}|\lesssim|x_2-c_{I_2}|}}\int^{\frac{|x_1-c_{I_1}|}{4}}_{\ell_1}\int^
     {\frac{|x_2-c_{I_2}|}{4}}_{\ell_2} \int^{\frac{|x_2-c_{I_2}|}{4}}_{\ell_1}\frac{r^{2M_2-4N_2}_2}{(r_2+|x_2-c_{I_2}|)^{2m+2M_2}}\\
     &\qquad\times\frac{r^{2M_3-4N_3}_3}{(r_3+|x_2-c_{I_2}|)^{2m+2M_3}}\cdot |x_2-c_{I_2}|^{2m}\,\frac {  
    dr_3dr_2dr_1}{r_3r_2r^{m+1}_1}\,d\mathbf{x} \\&\lesssim\|\triangle^{N_1}_1b_R\|^2_{L^2(\mathbb{R}^{2m})}\cdot\ell^m_2\cdot\int_{\substack{|x_1-c_{I_1}|\simeq2^i\ell_1\\|x_2-c_{I_2}|\simeq2^j\ell_2\\|x_1-c_{I_1}|\lesssim|x_2-c_{I_2}|}}\ell^{-m}_1\cdot\frac{\ell^{2M_2-4N_2}_2}{|x_2-c_{I_2}|^{2M_2}}\cdot\frac{\ell^{2M_3-4N_3}_1}{|x_2-c_{I_2}|^{2m+2M_3}}\,d\mathbf{x}\\
&\lesssim\|\triangle^{N_1}_1b_R\|^2_{L^2(\mathbb{R}^{2m})}\cdot\ell^m_2\cdot\int_{\substack{|x_1-c_{I_1}|\simeq2^i\ell_1\\|x_2-c_{I_2}|\simeq2^j\ell_2}}\ell^{-m}_1\cdot\frac{\ell^{2M_2-4N_2}_2}{|x_2-c_{I_2}|^{2M_2}}\cdot\left[\frac{\ell^{2M_3-4N_3}_1}{|x_2-c_{I_2}|^{2m}}\cdot\frac{1}{|x_1-c_{I_1}|^{2M_3}}\right]\,d\mathbf{x}\\
&\simeq\ell_2^{-4N_2}\ell^{-4N_3}_1\|\triangle^{N_1}_1b_R\|^2_{L^2(\mathbb{R}^{2m})}\cdot2^{i(m-2M_3)}2^{-j(m+2M_2)},  
\end{align*}
provided that $2M_2-4N_2,2M_3-4N_3<0$. To wrap up, we take $M_3>m$, then we have
\begin{align*}
       &|R|^{\frac{1}{2}}\sum^\infty_{i=6}\sum^\infty_{j=j_0}2^{\frac{(i+j)m}{2}}\cdot\left(II(C44)\right)^{\frac{1}{2}}\\
       &\lesssim |R|^{\frac{1}{2}}\sum^\infty_{i=6}\sum^\infty_{j=j_0}2^{\frac{(i+j)m}{2}}\cdot\ell_2^{-2N_2}\ell^{-2N_3}_1\|\triangle^{N_1}_1b_R\|_{L^2(\mathbb{R}^{2m})}\cdot2^{i(\frac{m}{2}-M_3)}2^{-j(\frac{m}{2}+M_2)}\\
         &\simeq|R|^{\frac{1}{2}}\cdot\ell_2^{-2N_2}\ell^{-2N_3}_1\|\triangle^{N_1}_1b_R\|_{L^2(\mathbb{R}^{2m})}\cdot\bigg( {\ell(\widehat{I_2})\over \ell(I_2)}\bigg)^{-M_2}.
   \end{align*}
\smallskip

For the term $III(\mathcal{C}44)$, based on the similar idea as in dominating $III(\mathcal{C}42)$, one can show
\begin{align*}
    III(\mathcal{C}44)\lesssim \max\bigg\{\frac{2^i\ell_1}{8},\frac{2^j\ell_2}{4}\bigg\}^{-4N_3}\left\|\triangle^{N_1}_1\triangle^{N_2}_{2}b_R\right\|^2_{L^2(\mathbb{R}^{2m})},
\end{align*}
and hence to wrap up, take $2N_3>m$,
\begin{align*}
       &|R|^{\frac{1}{2}}\sum^\infty_{i=6}\sum^\infty_{j=j_0}2^{\frac{(i+j)m}{2}}\cdot\left(III(\mathcal{C}44)\right)^{\frac{1}{2}}
       \lesssim |R|^{\frac{1}{2}}\cdot\left(\ell_1\ell_2\right)^{-N_3}\left\|\triangle^{N_1}_1\triangle^{N_2}_{2}b_R\right\|_{L^2(\mathbb{R}^{2m})}\cdot\bigg( {\ell(\widehat{I_2})\over \ell(I_2)}\bigg)^{-(N_3-\frac{m}{2})}.
   \end{align*}
Therefore, by combining the estimate in each term and the property of the atom, we finish the proof for the atom in the first type.
\medskip

$ \clubsuit$ We next consider the case that $\Diamond={\rm I\!I}$. For each tube $R=I_1\times_t I_2\in m^{\rm I\!I}(\Omega^{\rm I\!I})$ of type ${\rm I\!I}$, we recall that $|I_1|\leq|I_2|$. Let $k_0\in\mathbb{N}_0$ such that
\[\frac{|I_2|}{|I_1|}=2^{k_0 m}.\]
To utilize the covering Lemma, Lemma \ref{lem:Journe 1}, we have two cases. 

$\bullet$ Case $1$:
\[
|(I_1)_{k_0}\times_t I_2\cap\Omega^{\rm I\!I}|\leq\frac{1}{2}|(I)_{k_0}\times_t I_2|.
\] 
We define $\widehat{I}_1\subset\mathbb{R}^m$ to be the maximal dyadic interval such that $I_1\subset\widehat{I}_1$ and
\begin{align*}
    {\widehat{I}_1\times_t I_2\subset\widetilde{\Omega^{\rm I\!I}}:=\bigg\{\mathbf{x}\in\mathbb{R}^{2m}:\,M^{\rm I\! I}_{tube}(\chi_{\Omega^{\rm I\!I}})(\mathbf{x})>\frac{1}{2}\bigg\}.}
\end{align*}
Then we have $|\widehat{I}_1|<|I_2|$ and hence the tube $\widehat{I}_1\times_t I_2$ is in the set $m^{\rm I\!I}_1(\widetilde{\Omega^{\rm I\!I}})$.
Next, define $\widehat{I}_2\subset\mathbb{R}^m$ to be the maximal dyadic interval such that $I_2\subset\widehat{I}_2$ and
\begin{align*}
  {\widehat{I}_1\times_t \widehat{I}_2\subset\widetilde{\widetilde{\Omega^{\rm I\!I}}}:=\bigg\{\mathbf{x}\in\mathbb{R}^{2m}:\,M^{\rm I\! I}_{tube}(\chi_{\widetilde{\Omega^{\rm I\!I}}})(\mathbf{x})>\frac{1}{2}\bigg\}.}
\end{align*}

$\bullet$ Case $2$:
\[
|(I_1)_{k_0}\times_t I_2\cap\Omega^{\rm I\!I}|>\frac{1}{2}|(I_1)_{k_0}\times_t I_2|.
\] 
We define $\widehat{I}_1\subset\mathbb{R}^m$ to be the maximal dyadic interval such that $I_1\subset\widehat{I}_1=(I_1)_{k_0+l}$ and
\begin{align*}
{\left((I_1)_{k_0}\times_t I_2\right)_{l}}=   {(I_1)_{k_0+l}\times_t (I_2)_{l}\subset\widetilde{\Omega^{\rm I\!I}}:=\bigg\{\mathbf{x}\in\mathbb{R}^{2m}:\,M^{\rm I\! I}_{tube}(\chi_{\Omega^{\rm I\!I}})(\mathbf{x})>\frac{1}{2}\bigg\}.}
\end{align*}
Then we have $|\widehat{I}_1|=|(I_2)_l|$ and hence the slant cube $\widehat{I}_1\times_t (I_2)_l$ is in the set $m^{\rm I\!I}_1(\widetilde{\Omega^{\rm I\!I}})$.
Next, define $\widehat{I}_2\subset\mathbb{R}^m$ to be the maximal dyadic interval such that $(I_2)_l\subset\widehat{I}_2$ and
\begin{align*}
  {\widehat{I}_1\times_t \widehat{I}_2\subset\widetilde{\widetilde{\Omega^{\rm I\!I}}}:=\bigg\{\mathbf{x}\in\mathbb{R}^{2m}:\,M^{\rm I\! I}_{tube}(\chi_{\widetilde{\Omega^{\rm I\!I}}})(\mathbf{x})>\frac{1}{2}\bigg\}.}
\end{align*}
\smallskip
To complete the inclusion, it suffices to show that the $L^1$-norm of $  S_{area,\varphi}(a_{\Omega^{\rm I\!I}})$ is uniformly bounded. Note that if we let {$\widehat{R}$ be the $100$-fold dilation of $ \widehat{I}_1\times_t \widehat{I}_2$ concentric with $ \widehat{I}_1\times_t \widehat{I}_2$,} then we have, similar to (\ref{ESTI}),
\begin{align*}
    \|S_{area,\varphi}(a_{\Omega^{\rm I\!I}})\|_{L^1(\mathbb{R}^{2m})}
    \leq 1+\sum_{\substack{R\in m^{\rm I\!I}(\Omega^{\rm I\!I})\\R=I_1\times_t I_2 }}\int_{\left(\widehat{R}\right)^c}S_{area,\,\varphi}(a_R)(\mathbf{x})\,d\mathbf{x}.
\end{align*}
 It remains to estimate 
\[
\sum_{\substack{R\in m^{\rm I\!I}(\Omega^{\rm I\!I})\\R=I_1\times_t I_2 }}\int_{\left(\widehat{R}\right)^c}S_{area,\,\varphi}(a_R)(\mathbf{x})\,d\mathbf{x}=\Bigg(\sum_{\substack{R\in m^{\rm I\!I}(\Omega^{\rm I\!I})\\R=I_1\times_t I_2\\\text{Case} 1 }} +\sum_{\substack{R\in m^{\rm I\!I}(\Omega^{\rm I\!I})\\R=I_1\times_t I_2 \\\text{Case} 2}}\Bigg)\int_{\left(\widehat{R}\right)^c}S_{area,\,\varphi}(a_R)(\mathbf{x})\,d\mathbf{x}.
\]
For simplicity, we only estimate the second term
\begin{align}\label{second}
\sum_{\substack{R\in m^{\rm I\!I}(\Omega^{\rm I\!I})\\R=I_1\times_t I_2 \\\text{Case} 2}}\int_{\left(\widehat{R}\right)^c}S_{area,\,\varphi}(a_R)(\mathbf{x})\,d\mathbf{x};
\end{align}
the estimate of the first term can be  achieved by the similar manner.

Note that (\ref{second}) can be dominated by
\begin{align*}
   & \sum_{\substack{R\in m^{\rm I\!I}(\Omega^{\rm I\!I})\\R=I_1\times_t I_2 \\\text{Case} 2}} \int_{(100\widehat{I_1})^c\times_t (100I_2)}S_{area,\,\varphi}(a_R)(\mathbf{x})\,d\mathbf{x}+\sum_{\substack{R\in m^{\rm I\!I}(\Omega^{\rm I\!I})\\R=I_1\times_t I_2\\\text{Case} 2 }} \int_{(100\widehat{I_1})^c\times_t (100I_2)^c}S_{area,\,\varphi}(a_R)(\mathbf{x})\,d\mathbf{x}\\
  &\quad+\sum_{\substack{R\in m^{\rm I\!I}(\Omega^{\rm I\!I})\\R=I_1\times_t I_2 \\\text{Case} 2}}\int_{(100{I_1})^c\times_t (100\widehat{I_2})^c}\ S_{area,\,\varphi}(a_R)(\mathbf{x})\,d\mathbf{x}+\sum_{\substack{R\in m^{\rm I\!I}(\Omega^{\rm I\!I})\\R=I_1\times_t I_2 \\\text{Case} 2}}\int_{(100{I_1})\times_t (100\widehat{I_2})^c} S_{area,\,\varphi}(a_R)(\mathbf{x})\,d\mathbf{x}\\
  &=:\sum_{\substack{R\in m^{\rm I\!I}(\Omega^{\rm I\!I})\\R=I_1\times_t I_2\\\text{Case} 2 }}(\mathfrak{A})+\sum_{\substack{R\in m^{\rm I\!I}(\Omega^{\rm I\!I})\\R=I_1\times_t I_2 \\\text{Case} 2}}(\mathfrak{B})+\sum_{\substack{R\in m^{\rm I\!I}(\Omega^{\rm I\!I})\\R=I_1\times_t I_2\\\text{Case} 2 }}(\mathfrak{C})+\sum_{\substack{R\in m^{\rm I\!I}(\Omega^{\rm I\!I})\\R=I_1\times_t I_2\\\text{Case} 2 }}(\mathfrak{D}).
\end{align*}
Based on the covering lemma, which preserves the tube structure in each case, we only need to estimate the case $({\mathfrak{C}})$ and the case $({\mathfrak{D}})$. Due to the twisted phenomenon, the new ingredient for the estimate on $(\mathfrak{C})$ is the {\it molecule decomposition}.
\smallskip

$\bullet$ {\it Case $({\mathfrak{C}})$}:
 To estimate the term $(\mathfrak{C})$, we first apply the dyadic decomposition to $(100{I_1})^c\times_t (100\widehat{I_2})^c$ and then adopt H\"older's inequality to get 
\begin{align*}
(\mathfrak{C})&= \int_{(100{I_1})^c\times_t (100\widehat{I_2})^c}\ S_{area,\,\varphi}(a_R)(\mathbf{x})\,d\mathbf{x}
\leq\sum^\infty_{i=6}\sum^\infty_{j=j_0+l}\int_{\substack{|x_1-x_2-c_{I_1}|\simeq2^i\ell_1\\|x_2-c_{I_2}|\simeq2^j\ell_2}}\ S_{area,\,\varphi}(a_R)(\mathbf{x})\,d\mathbf{x}\notag\\
&\leq |R|^{\frac{1}{2}}\sum^\infty_{i=6}\sum^\infty_{j=j_0+l}2^{\frac{(i+j)m}{2}}\cdot\left(\int_{\substack{|x_1-x_2-c_{I_1}|\simeq2^i\ell_1\\|x_2-c_{I_2}|\simeq2^j\ell_2}}\left|S_{area,\,\varphi}(a_R)(\mathbf{x})\right|^2\,d\mathbf{x}\right)^{\frac{1}{2}},
\end{align*}
where $j_0\in\mathbb{N}$ is such that $2^{j_0m}|(I_2)_l|\simeq{|\widehat{I_2}|}$.
 Besides, the $L^2$-norm of $S_{area,\,\varphi}(a_R)(\mathbf{x})$ over the dyadic region can be further decomposed as
 \begin{align*}
&\int_{\substack{|x_1-x_2-c_{I_1}|\simeq2^i\ell_1\\|x_2-c_{I_2}|\simeq2^j\ell_2}}\left|S_{area,\,\varphi}(a_R)(\mathbf{x})\right|^2\,d\mathbf{x}\\
&=\int_{\substack{|x_1-x_2-c_{I_1}|\simeq2^i\ell_1\\|x_2-c_{I_2}|\simeq2^j\ell_2}}\int_{ \mathbb{R}^{2m}}\left(\int^{\ell_1}_0\int^
     {\ell_2}_0+\int^{\ell_1}_0\int^\infty_{\ell_2}+\int^\infty_{\ell_1}\int^
{\ell_2}_0+\int^\infty_{\ell_1}\int^\infty_
{\ell_2}\right)\int^\infty_0 \\
&\quad|a_R* \varphi_{\mathbf{r} } |^2(\mathbf{y})\cdot\chi_\mathbf{r}(\mathbf{x}-\mathbf{y})\,\frac {  
    dr_2dr_3dr_1}{r_2r_3r_1}\,d\mathbf{y}\,d\mathbf{x}\\ 
    &=:(\mathfrak{C}1)+{(\mathfrak{C}2)+(\mathfrak{C}3)}+(\mathfrak{C}4).
 \end{align*}
\smallskip
$\blacksquare$ {\it Estimate on the term $(\mathfrak{C}1)$}:

 To estimate this term, we need to further split the range of $r_2$. 
 \begin{align*}
     (\mathfrak{C}1) &=\int_{\substack{|x_1-x_2-c_{I_1}|\simeq2^i\ell_1\\|x_2-c_{I_2}|\simeq2^j\ell_2}}\int_{ \mathbb{R}^{2m}}\int^{\ell_1}_0\int^
     {\ell_2}_0\bigg(\int^{{{\ell_2}}}_0 +\int^{\frac{1}{4}\max\{|x_1-x_2-c_{I_1}|,\,|x_2-c_{I_2}|\} }_{{\ell_2}} \\
     &\quad+\int^\infty_{\frac{1}{4}\max\{|x_1-x_2-c_{I_1}|,\,|x_2-c_{I_2}|\} } \bigg)|a_R* \varphi_{\mathbf{r} } |^2(\mathbf{y})\cdot\chi_\mathbf{r}(\mathbf{x}-\mathbf{y})\,\frac {  
    dr_2dr_3dr_1}{r_2r_3r_1}\,d\mathbf{y}\,d\mathbf{x}\\
    &=:(\mathfrak{C}11)+(\mathfrak{C}12)+(\mathfrak{C}13).
 \end{align*}
 \begin{itemize}
     \item {\it Estimate on $(\mathfrak{C}11)$}:
 \end{itemize}
By checking the support condition given by the second variable, one can conclude that the term $(\mathfrak{C}11)$ vanishes.

\begin{itemize}
     \item {\it Estimate on $(\mathfrak{C}12)$}:
 \end{itemize}
 Now, we turn to estimate term $(\mathfrak{C}12)$, which is equal to
\begin{align*}
     &\int_{\substack{|x_1-x_2-c_{I_1}|\simeq2^i\ell_1\\|x_2-c_{I_2}|\simeq2^j\ell_2}}\int_{ \mathbb{R}^{2m}}\int^{\ell_1}_0\int^
     {\ell_2}_0\int^{\frac{1}{4}\max\{|x_1-x_2-c_{I_1}|,\,|x_2-c_{I_2}|\} }_{{\ell_2}} \left|a_R*\varphi_{\mathbf{r}}\right|^2(\mathbf{y})\cdot\chi_\mathbf{r}(\mathbf{x}-\mathbf{y})\,\frac {  
    dr_2 dr_3dr_1}{r_2r_3r_1}\,d\mathbf{y}\,d\mathbf{x}.
\end{align*}
Suppose that $|x_2-c_{I_2}|\geq|x_1-x_2-c_{I_1}|$. Then by investigating the support condition given by the second variable, this term would vanish. 

Assume that $|x_2-c_{I_2}|<|x_1-x_2-c_{I_1}|$. We now investigate the support condition given by the first variable.
By the cone structure, we have
$|x_1-y_1|\leq \ell_1+\ell_2$; moreover, since
\begin{align*}
a_R*\varphi_{\mathbf{r}}(\mathbf{y})=\int_{\mathbb{R}^{{2m}}}\left(a_R*_3\varphi^{(3)}_{r_3}\right)(\mathbf{w})\left(\varphi^{(1)}_{r_1}\otimes\varphi^{(2)}_{r_2}\right) (y_1-w_1,y_2-w_2) \,d\mathbf{w},
\end{align*}
then we have $|y_1-w_1|\leq\ell_1$ and $|w_1-w_2-c_{I_1}|\leq \frac{\ell_1}{2}$. Therefore, one can deduce that 
\begin{align*}
\ell_1\geq|y_1-w_1|&\geq|x_1-x_2-c_{I_1}|- |x_1-y_1|-|x_2-c_{I_2}|-|w_1-w_2-c_{I_1}|-|w_2-c_{I_2}|\\
   &\geq |x_1-x_2-c_{I_1}|-|x_2-c_{I_2}|-\frac{3\ell_1}{2}-\frac{5\ell_2}{2},
\end{align*}
combining with the facts that $\ell_1\leq\frac{|x_1-x_2-c_{I_1}|}{50}$ and $\ell_2\leq\frac{|x_2-c_{I_2}|}{50}$, which leads to \begin{align}\label{D6666}
    |x_2-c_{I_2}|\leq|x_1-x_2-c_{I_1}|\leq \frac{21}{19}|x_2-c_{I_2}|.
\end{align}
On the other hand, by investigating the support condition given by the second variable together with (\ref{D6666}), we have that
\begin{align*}
 \frac{107}{380}|x_2-c_{I_2}|\geq|y_2-w_2|&\geq|x_2-c_{I_2}|- |w_2-c_{I_2}|-|x_2-y_2|\\
   &\geq |x_2-c_{I_2}|-\frac{3\ell_2}{2}-\frac{|x_1-x_2-c_{I_1}|}{4}-\ell_2\\
   &\geq\frac{254}{380}|x_2-c_{I_2}|,
\end{align*}
 which is obviously a contradiction. So, in either case, $(\mathfrak{C}12)=0$.

\begin{itemize}
     \item {\it Estimate on $(\mathfrak{C}13)$}:
 \end{itemize}

We apply the Young's convolution inequality and Littlewood--Paley inequality to eliminate $\varphi^{(1)}_{r_1}$ and $\varphi^{(3)}_{r_3}$, and thus $(\mathfrak{C}13)$ is dominated by
\begin{align*}
\int_{\mathbb{R}^m}\int^\infty_{\frac{1}{4}\max\{2^i\ell_1,\,2^j\ell_2\}}\int_{ \mathbb{R}^{m}}|a_R*_2\varphi^{(2)}_{r_2} |^2(y_1,y_2)\,dy_2\,\frac{dr_2}{r_2}\,dy_1,
\end{align*}
and we also have
 \begin{align*}
\left\|a_R*_2\varphi^{(2)}_{r_2}(y_1,\,y_2)\right\|^2_{L^2(\mathbb{R}^{m},\,d{y}_2)}\lesssim r_2^{-4N_2}\left\|\left({ {\triangle}_1^{ N_1 } \triangle^{N_3}_{twist}b_{R}}\right)(\mathbf{y})\right\|^2_{L^2(\mathbb{R}^{m},\,d{y}_2)},
   \end{align*}
   which leads to
   \begin{align*}
       (\mathfrak{C}13)&\lesssim \max\{2^i\ell_1,\,2^j\ell_2\}^{-4N_2}\left\|\left({ {\triangle}_1^{ N_1 } \triangle^{N_3}_{twist}b_{R}}\right)(\mathbf{y})\right\|^2_{L^2(\mathbb{R}^{2m},\,d\mathbf{y})}.
   \end{align*}
   To wrap up, if we take $N_2=m$, then we have
   \begin{align*}
       &|R|^{\frac{1}{2}}\sum^\infty_{i=6}\sum^\infty_{j=j_0+l}2^{\frac{(i+j)m}{2}}\cdot(\mathfrak{C}13)^{\frac{1}{2}}\\
       &\lesssim |R|^{\frac{1}{2}}\sum^\infty_{i=6}\sum^\infty_{j=j_0+l}2^{\frac{(i+j)m}{2}}\cdot\max\{2^i\ell_1,\,2^j\ell_2\}^{-2N_2}\left\|\left({ {\triangle}_1^{ N_1 } \triangle^{N_3}_{twist}b_{R}}\right)(\mathbf{y})\right\|_{L^2(\mathbb{R}^{2m},\,d\mathbf{y})}\\
       &\leq |R|^{\frac{1}{2}}\cdot\ell_1^{-N_2}\ell_2^{-N_2}\left\|\left({ {\triangle}_1^{ N_1 } \triangle^{N_3}_{twist}b_{R}}\right)(\mathbf{y})\right\|_{L^2(\mathbb{R}^{2m},\,d\mathbf{y})}\cdot\bigg( {\ell(\widehat{I_2})\over \ell(I_2)}\bigg)^{-\frac{m}{2}}.
   \end{align*}
 \smallskip

$\blacksquare$ {\it Estimate on the term $(\mathfrak{C}2)$}:

We now estimate the term $(\mathfrak{C}2)$. By further splitting the range of $r_3$, we have that
\begin{align*}
     (\mathfrak{C}2)
&=\int_{\substack{|x_1-x_2-c_{I_1}|\simeq2^i\ell_1\\|x_2-c_{I_2}|\simeq2^j\ell_2}}\int_{ \mathbb{R}^{2m}}\int^{\ell_1}_0\left(\int^{\frac{|x_2-c_{I_2}|}{4}}_{\ell_2}+\int^\infty_{\frac{|x_2-c_{I_2}|}{4}}\right)\int^\infty_0|a_R* \varphi_{\mathbf{r} } |^2(\mathbf{y})\cdot\chi_\mathbf{r}(\mathbf{x}-\mathbf{y})\,\frac {  
dr_2    dr_3dr_1}{r_2r_3r_1}\,d\mathbf{y}\,d\mathbf{x}\\ 
  &=:(\mathfrak{C}21)+{(\mathfrak{C}22)}.
\end{align*}

\begin{itemize}
     \item {\it Estimate on $(\mathfrak{C}21)$}:
 \end{itemize}
Remark that
 \begin{align*}
 a_R*\varphi_{\mathbf{r}}(\mathbf{y})=  \int_{\mathbb{R}^{{2m}}}\left(a_R*_3\varphi^{(3)}_{r_3}\right)(\mathbf{w})\left(\varphi^{(1)}_{r_1}\otimes\varphi^{(2)}_{r_2}\right) (y_1-w_1,y_2-w_2) \,d\mathbf{w}.
 \end{align*}
Then if we investigate the support condition given by the first variable, then we have
\begin{align*}
\ell_1&\geq|y_1-w_1|\geq|\left(x_1-x_2-c_{I_1}\right)+\left(x_2-c_{I_2}\right)|- |x_1-y_1|-|w_1-w_2-c_{I_1}|-|w_2-c_{I_2}|\\
&\geq |\left(x_1-x_2-c_{I_1}\right)+\left(x_2-c_{I_2}\right)|-\frac{|x_2-c_{I_2}|}{2}-\frac{3\ell_1}{2}-\frac{\ell_2}{2},
\end{align*}
which implies the condition that reveals the position of $x_2$, that is
\begin{align}\label{similar}
    \frac{95}{151}|x_1-x_2-c_{I_1}|\leq|x_2-c_{I_2}|\leq \frac{104}{45}|x_1-x_2-c_{I_1}|.
\end{align}
Next, if we investigate the support condition given by the second variable, then we have
\begin{align*}
r_2\geq|y_2-w_2|&\geq|x_2-c_{I_2}|- |w_2-c_{I_2}|-|x_2-y_2|\\
   &\geq |x_2-c_{I_2}|-r_2-\frac{|x_2-c_{I_2}|}{2}-\frac{\ell_2}{2}\\
   &\geq\frac{931}{3020}|x_1-x_2-c_{I_1}|-r_2,
\end{align*}
that is $r_2\geq\frac{931}{6040}|x_1-x_2-c_{I_1}|$.
As a consequence, 
\begin{align*}
    (\mathfrak{C}21)&=\int_{\substack{|x_1-x_2-c_{I_1}|\simeq2^i\ell_1\\|x_2-c_{I_2}|\simeq2^j\ell_2}}\int_{ \mathbb{R}^{2m}}\int^{\ell_1}_0\int^{\frac{|x_2-c_{I_2}|}{4}}_{\ell_2}\int^\infty_{\frac{930}{6040}|x_1-x_2-c_{I_1}|}|a_R* \varphi_{\mathbf{r} } |^2(\mathbf{y})\cdot\chi_\mathbf{r}(\mathbf{x}-\mathbf{y})\,\frac {  
    dr_2dr_3dr_1}{r_2r_3r_1}\,d\mathbf{y}\,d\mathbf{x}.
\end{align*}
Similar to the case $(\mathfrak{C}13)$, we have that
\begin{align*}
(\mathfrak{C}21)&\lesssim\left(2^i\ell_1\right)^{-4N_2}\left\|\triangle^{N_1}_1\triangle^{N_3}_{twist}b_R\right\|^2_{L^2(\mathbb{R}^{2m})},
\end{align*}
and hence to wrap up,
if we take $N_2=m$ and combine with (\ref{similar}), then we have
   \begin{align*}
       &|R|^{\frac{1}{2}}\sum^\infty_{i=6}\sum^\infty_{j=j_0+l}2^{\frac{(i+j)m}{2}}\cdot\left((\mathfrak{C}21\right)^{\frac{1}{2}}\\
       &\lesssim |R|^{\frac{1}{2}}\sum^\infty_{i=6}\sum^\infty_{j=j_0+l}2^{\frac{(i+j)m}{2}}\cdot(2^i\ell_1)^{-N_2}\cdot(2^j\ell_2)^{-N_2}\left\|\left({ {\triangle}_1^{ N_1 } \triangle^{N_3}_{twist}b_{R}}\right)\right\|_{L^2(\mathbb{R}^{2m})}\\
       &\leq |R|^{\frac{1}{2}}\cdot\ell_1^{-N_2}\ell_2^{-N_2}\left\|\left({ {\triangle}_1^{ N_1 } \triangle^{N_3}_{twist}b_{R}}\right)\right\|_{L^2(\mathbb{R}^{2m})}\cdot\bigg( {\ell(\widehat{I_2})\over \ell(I_2)}\bigg)^{-\frac{m}{2}}.
   \end{align*}
\smallskip

\begin{itemize}
     \item {\it Estimate on $(\mathfrak{C}22)$}:
 \end{itemize}
 To estimate this term, we further split it into
 \begin{align*}
(\mathfrak{C}22)&:=\int_{\substack{|x_1-x_2-c_{I_1}|\simeq2^i\ell_1\\|x_2-c_{I_2}|\simeq2^j\ell_2}}\int_{ \mathbb{R}^{2m}}\int^{\ell_1}_0\int^\infty_{\frac{|x_2-c_{I_2}|}{4}}\Bigg(\int^{\ell_1}_{0}+\int^{\frac{|x_1-x_2-c_{I_1}|}{8}}_{\ell_1}+\int^\infty_{\frac{|x_1-x_2-c_{I_1}|}{8}}\Bigg) \\
&\quad\quad|a_R* \varphi_{\mathbf{r} } |^2(\mathbf{y})\cdot\chi_{\mathbf{r}}(\mathbf{x}-\mathbf{y})\,\frac {  
    dr_2dr_3dr_1}{r_2r_3r_1}\,d\mathbf{y}\,d\mathbf{x}\\
    &=:{\underset{\text{molecule~decomposition}}{I(\mathfrak{C}22)+II(\mathfrak{C}22)}}+III(\mathfrak{C}22).
 \end{align*}
We will use the {\it molecule decomposition technique} to estimate $I(\mathfrak{C}22)$ and $II(\mathfrak{C}22)$. For the simplicity, we demonstrate this technique on the estimate of the term $II(\mathfrak{C}43)$ (see the forthcoming subsection for detail and comment). 

For the estimate of the term $III(\mathfrak{C}22)$, interchanging the integral and apply Littlewood--Paley's inequality to eliminate $\varphi^{(1)}_{r_1}$, we have $III(\mathfrak{C}22)$ is bounded above by
\begin{align*}
 \int^\infty_{\frac{2^j\ell_2}{4}}\int^\infty_{\frac{2^i\ell_1}{4}} \left\|a_R*_2\varphi^{(2)}_{r_2}*_3 \varphi^{(3)}_{{r}_3 } \right\|^2_{L^2(\mathbb{R}^{2m})}\,\frac {dr_3dr_2}{r_3r_2}.
\end{align*}
To continue, by applying Young's convolution inequality, we know that
\begin{align*}
III(\mathfrak{C}22)\lesssim\int^\infty_{\frac{2^j\ell_2}{4}}\int^\infty_{\frac{2^i\ell_1}{4}} \left\|\triangle^{N_1}_1b_R\right\|^2_{L^2(\mathbb{R}^{2m})}\,\frac {dr_3dr_2}{r^{4N_3+1}_3r^{4N_2+1}_2},
\end{align*}
and hence to wrap up, take $N_2=N_3=m$, we have
\begin{align*}
       |R|^{\frac{1}{2}}\sum^\infty_{i=6}\sum^\infty_{j=(j_0+l)}2^{\frac{(i+j)m}{2}}\cdot\left(III(\mathfrak{C}22)\right)^{\frac{1}{2}}
       &\lesssim  |R|^{\frac{1}{2}}\sum^\infty_{i=6}\sum^\infty_{j=(j_0+l)}2^{{i(\frac{m}{2}-N_2)}} 2^{{j(\frac{m}{2}-N_3)}} \left(\ell_1\ell_2\right)^{-2m}\left\|\triangle^{N_1}_1b_R\right\|_{L^2(\mathbb{R}^{2m})}\\
       &\simeq  |R|^{\frac{1}{2}}\cdot\left(\ell^{-2N_2}_1\ell^{-2N_2}_2\right)\left\|\triangle^{N_1}_1b_R\right\|_{L^2(\mathbb{R}^{2m})}\cdot\bigg( {\ell(\widehat{I_2})\over \ell(I_2)}\bigg)^{-\frac{m}{2}}.
   \end{align*}
\smallskip
$\blacksquare$ {\it Estimate on the term $(\mathfrak{C}3)$}: We decompose this term into
$(\mathfrak{C}31)$ and $(\mathfrak{C}32)$.
\begin{align*}
(\mathfrak{C}3)&=\int_{\substack{|x_1-x_2-c_{I_1}|\simeq2^i\ell_1\\|x_2-c_{I_2}|\simeq2^j\ell_2}}\int_{ \mathbb{R}^{2m}}\left(\int^\infty_{\frac{|x_1-x_2-c_{I_1}|}{4}}+\int^{\frac{|x_1-x_2-c_{I_1}|}{4}}_{\ell_1}\right)\int^
{\ell_2}_0\int^\infty_0 \\
&\quad\quad|a_R* \varphi_{\mathbf{r} } |^2(\mathbf{y})\cdot\chi_{\mathbf{r}}(\mathbf{x}-\mathbf{y})\,\frac {  
    dr_2dr_3dr_1}{r_2r_3r_1}\,d\mathbf{y}\,d\mathbf{x}\\ 
    &=:(\mathfrak{C}31)+(\mathfrak{C}32).
\end{align*}

\begin{itemize}
     \item {\it Estimate on $(\mathfrak{C}31)$}:
 \end{itemize}
 We further decompose the range of $r_2$ and get that
 \begin{align*}
(\mathfrak{C}31)&=\int_{\substack{|x_1-x_2-c_{I_1}|\simeq2^i\ell_1\\|x_2-c_{I_2}|\simeq2^j\ell_2}}\int_{ \mathbb{R}^{2m}}\int^\infty_{\frac{|x_1-x_2-c_{I_1}|}{4}}\int^
{\ell_2}_0\left(\int^\infty_{\frac{|x_2-c_{I_2}|}{4}}+\int^{\frac{|x_2-c_{I_2}|}{4}}_0\right)\\
&\qquad|a_R* \varphi_{\mathbf{r} } |^2(\mathbf{y})\cdot\chi_{\mathbf{r}}(\mathbf{x}-\mathbf{y})\,\frac {  
    dr_2dr_3dr_1}{r_2r_3r_1}\,d\mathbf{y}\,d\mathbf{x}\\ &=:I(\mathfrak{C}31)+II(\mathfrak{C}31).
 \end{align*}
 For the term $I(\mathfrak{C}31)$, we apply the trivial estimate to the cone structure, Littlewood--Paley inequality to eliminate $\varphi^{(3)}_{r_3}$ and eventually, following the similar estimate as $III(\mathfrak{C}22)$, we have
 \begin{align}\label{IC31}
 I(\mathfrak{C}31)
   &\lesssim  \int_{\substack{|x_1-x_2-c_{I_1}|\simeq2^i\ell_1\\|x_2-c_{I_2}|\simeq2^j\ell_2}}\int^\infty_{\frac{|x_1-x_2-c_{I_1}|}{4}}\int^\infty_{\frac{|x_2-c_{I_2}|}{4}}\left\|\triangle^{N_3}_{twist}b_R\right\|^2_{L^2(\mathbb{R}^{2m})}\frac {  
   dr_2}{r^{4N_2+m+1}_2}\,\frac{dr_1}{r^{4N_1+m+1}_1}\,d\mathbf{x}\notag\\ 
   &\simeq(2^i\ell_1)^{-4N_1}\cdot(2^j\ell_2)^{-4N_2}\left\|\triangle^{N_3}_{twist}b_R\right\|^2_{L^2(\mathbb{R}^{2m})}.
 \end{align}
 If we take $N_1=N_2=m$, then combines with (\ref{IC31})
   \begin{align*}
       &|R|^{\frac{1}{2}}\sum^\infty_{i=6}\sum^\infty_{j=j_0+l}2^{\frac{(i+j)m}{2}}\cdot\left(I(C31)\right)^{\frac{1}{2}}
       \lesssim |R|^{\frac{1}{2}}\cdot\ell_1^{-2N_1}\cdot\ell_2^{-2N_2}\left\|\triangle^{N_3}_{twist}b_R\right\|_{L^2(\mathbb{R}^{2m})}\cdot\bigg( {\ell(\widehat{I_2})\over \ell(I_2)}\bigg)^{-\frac{3m}{2}}.
   \end{align*}
\smallskip
For the term $II(\mathfrak{C}31)$, we investigate the support condition given by the second variable. Then one will get a contradiction to the support condition of $x_2$ and thereby the term $II(\mathfrak{C}31)$ vanishes.
\smallskip

\begin{itemize}
     \item {\it Estimate on $(\mathfrak{C}32)$}:
 \end{itemize}
  We further decompose the range of $r_2$ and get that
 \begin{align*}
(\mathfrak{C}32)&=\int_{\substack{|x_1-x_2-c_{I_1}|\simeq2^i\ell_1\\|x_2-c_{I_2}|\simeq2^j\ell_2}}\int_{ \mathbb{R}^{2m}}\int^{\frac{|x_1-x_2-c_{I_1}|}{4}}_{\ell_1}\int^
{\ell_2}_0\left(\int^\infty_{\max\big\{\frac{|x_1-x_2-c_{I_1}|}{100},\frac{|x_2-c_{I_2}|}{4} \big\}}+\int^{\max\big\{\frac{|x_1-x_2-c_{I_1}|}{100},\frac{|x_2-c_{I_2}|}{4} \big\}}_0\right)\\
&\qquad |a_R* \varphi_{\mathbf{r} } |^2(\mathbf{y})\cdot\chi_\mathbf{r}(\mathbf{x}-\mathbf{y})\,\frac {  
    dr_2dr_3dr_1}{r_2r_3r_1}\,d\mathbf{y}\,d\mathbf{x}\\  
    &=:I(\mathfrak{C}32)+II(\mathfrak{C}32).
 \end{align*}
 We only need to estimate the term $I(\mathfrak{C}32)$. Since that the second term $II(\mathfrak{C}32)$ would vanish. To be more specific, if $$
 {\max\bigg\{\frac{|x_1-x_2-c_{I_1}|}{100},\frac{|x_2-c_{I_2}|}{4} \bigg\}}=\frac{|x_2-c_{I_2}|}{4},
 $$
 then by adapting the as the same argument in $II(\mathfrak{C}31)$, one can verify that $II(\mathfrak{C}32)=0$; and if 
$$
 {\max\bigg\{\frac{|x_1-x_2-c_{I_1}|}{100},\frac{|x_2-c_{I_2}|}{4} \bigg\}}=\frac{|x_1-x_2-c_{I_1}|}{100},
 $$
 then we have that
 $$
 |x_2-c_{I_2}|\leq \frac{|x_1-x_2-c_{I_1}|}{25}
 $$
 and hence if we expand the convolution $a_R* \varphi_{\mathbf{r} }$ and investigate the support condition given by $\varphi^{(1)}_{r_1}$, then we can derive that
 $$
 \frac{1}{4}|x_1-x_2-c_{I_1}|\geq\frac{69}{100}|x_1-x_2-c_{I_1}|,
 $$
 which is a contradiction. Hence, in either case, there is no contribution in this term. 
 
 Now, we estimate the term $I(\mathfrak{C}32)$. By first interchanging the integral to integrate against $\mathbf{x}$ and Littlewood--Paley inequality to eliminate $\varphi^{(1)}_{r_1}$ and $\varphi^{(3)}_{r_3}$, one has that
 \begin{align*}
     I(\mathfrak{C}32)&\lesssim\int_{ \mathbb{R}^{2m}}\int^{\frac{2^i\ell_1}{4}}_{\ell_1}\int^
{\ell_2}_0\int^\infty_{\max\{\frac{2^i\ell_1}{100},\frac{2^j\ell_2}{4} \}}|a_R* \varphi_{\mathbf{r} } |^2(\mathbf{y})\,\frac {  
    dr_2dr_3dr_1}{r_2r_3r_1}\,d\mathbf{y}\\
    &\lesssim\int^\infty_{\max\{\frac{2^i\ell_1}{100},\frac{2^j\ell_2}{4} \}}\left\|a_R*_2 \varphi^{(2)}_{r_2}\right\|^2_{L^2(\mathbb{R}^{2m})}\,\frac {  
   dr_2}{r_2}\\
   &\lesssim\left(2^i\ell_1\right)^{-2N_2}\cdot\left(2^j\ell_2\right)^{-2N_2}\left\|\triangle^{N_1}_1\triangle^{N_3}_{twist}b_R\right\|^2_{L^2(\mathbb{R}^{2m})},
 \end{align*}
 and hence to wrap up, take $N_2=m$ and one has
\begin{align*}
|R|^{\frac{1}{2}}\sum^\infty_{i=6}\sum^\infty_{j=j_0+l}2^{\frac{(i+j)m}{2}}\cdot\left(I(\mathfrak{C}32)\right)^{\frac{1}{2}}
       &\lesssim |R|^{\frac{1}{2}}\sum^\infty_{i=6}\sum^\infty_{j=j_0+l}2^{\frac{(i+j)m}{2}}\cdot\left(2^i\ell_1\right)^{-2N_2}\cdot\left(2^j\ell_2\right)^{-2N_2}\cdot\left\|\triangle^{N_1}_1\triangle^{N_3}_{twist}b_R\right\|_{L^2(\mathbb{R}^{2m})}\\
&\simeq|R|^{\frac{1}{2}}\cdot\ell_1^{-2N_2}\cdot\ell_2^{-2N_2}\cdot\left\|\triangle^{N_1}_1\triangle^{N_3}_{twist}b_R\right\|_{L^2(\mathbb{R}^{2m})}\cdot\bigg( {\ell(\widehat{I_2})\over \ell(I_2)}\bigg)^{\frac{-3m}{2}}.
   \end{align*}
   
\smallskip

$\blacksquare$ {\it Estimate on the term $(\mathfrak{C}4)$}: Consider the following decomposition
 \begin{align*}
(\mathfrak{C}4)&=\int_{\substack{|x_1-x_2-c_{I_1}|\simeq2^i\ell_1\\|x_2-c_{I_2}|\simeq2^j\ell_2}}\int_{ \mathbb{R}^{2m}}\Bigg(\int^\infty_{\frac{|x_1-x_2-c_{I_1}|}{4}}\int^\infty_
     {\frac{|x_2-c_{I_2}|}{4}}+\int^\infty_{\frac{|x_1-x_2-c_{I_1}|}{4}}\int^{\frac{|x_2-c_{I_2}|}{4}}_{\ell_2}+\int^{\frac{|x_1-x_2-c_{I_1}|}{4}}_{\ell_1}\int^\infty_
     {\frac{|x_2-c_{I_2}|}{4}}\\
     &\qquad+\int^{\frac{|x_1-x_2-c_{I_1}|}{4}}_{\ell_1}\int^
     {\frac{|x_2-c_{I_2}|}{4}}_{\ell_2} \Bigg)\int^\infty_0 |a_R* \varphi_{\mathbf{r} } |^2(\mathbf{y})\cdot\chi_\mathbf{r}(\mathbf{x}-\mathbf{y})\,\frac {  
    dr_2dr_3dr_1}{r_2r_3r_1}\,d\mathbf{y}\,d\mathbf{x}\\
    &=:(\mathfrak{C}41)+(\mathfrak{C}42)+(\mathfrak{C}43)+(\mathfrak{C}44).
 \end{align*}
\begin{itemize}
     \item {\it Estimate on $(\mathfrak{C}41)$}:
 \end{itemize}
 By interchanging the integral, integrate against the variable $\mathbf{x}$, Littlewood--Paley inequality to eliminate $\varphi^{(2)}_{r_2}$ and the cancellation property of the atom, we have
 \begin{align*}
(\mathfrak{C}41)\lesssim \left(2^i\ell_1\right)^{-4N_1}\cdot\left(2^j\ell_2\right)^{-4N_2}\cdot\left\|\triangle^{N_3}_{twist}b_R\right\|^2_{L^2(\mathbb{R}^{2m})}.
 \end{align*}
 To wrap up, take $N_1=N_2=m$, it gives that
  \begin{align*}
|R|^{\frac{1}{2}}\sum^\infty_{i=6}\sum^\infty_{j=j_0+l}2^{\frac{(i+j)m}{2}}\cdot\left((\mathfrak{C}41)\right)^{\frac{1}{2}}
&\lesssim|R|^{\frac{1}{2}}\cdot\ell_1^{-2N_1}\cdot\ell_2^{-2N_2}\cdot\left\|\triangle^{N_3}_{twist}b_R\right\|^2_{L^2(\mathbb{R}^{2m})}\cdot \bigg( {\ell(\widehat{I_2})\over \ell(I_2)}\bigg)^{\frac{-3m}{2}}.
   \end{align*}
   
   \begin{itemize}
     \item {\it Estimate on $(\mathfrak{C}42)$}:
 \end{itemize}
 We further split this term $(\mathfrak{C}42)$. Then we have
 \begin{align*}
  (\mathfrak{C}42)&=\int_{\substack{|x_1-x_2-c_{I_1}|\simeq2^i\ell_1\\|x_2-c_{I_2}|\simeq2^j\ell_2}}\int_{ \mathbb{R}^{2m}}\int^\infty_{\frac{|x_1-x_2-c_{I_1}|}{4}}\int^{\frac{|x_2-c_{I_2}|}{4}}_{\ell_2}\left(\int^{\ell_2}_0+\int^{\frac{|x_2-c_{I_2}|}{4} }_{\ell_2}+\int^\infty_{\frac{|x_2-c_{I_2}|}{4}}\right) \\
  &\qquad|a_R* \varphi_{\mathbf{r} } |^2(\mathbf{y})\cdot\chi_\mathbf{r}(\mathbf{x}-\mathbf{y})\,\frac {  
    dr_2dr_3dr_1}{r_2r_3r_1}\,d\mathbf{y}\,d\mathbf{x}\\&=:I(\mathfrak{C}42)+II(\mathfrak{C}42)+III(\mathfrak{C}42).
 \end{align*}
The term $I(\mathfrak{C}42)$ would be equal to zero, which can be verified by investigating the support condition of the second variable.  For the term $III(\mathfrak{C}42)$, by interchanging the integral to integrate against the variable $\mathbf{x}$ and the Littlewood--Paley inequality to eliminate the function $\varphi^{(3)}_{r_3}$, we have
\begin{align*}
III(\mathfrak{C}42)&\lesssim\int^\infty_{\frac{2^i\ell_1}{4}}\int^\infty_{\frac{2^j\ell_2}{4}}\left\|a_R*\left( \varphi^{(1)}_{{r}_1 } \otimes\varphi^{(2)}_{{r}_2}\right)\right\|^2_{L^2(\mathbb{R}^{2m})}\,\frac {  
    dr_2dr_1}{r_2r_1},
\end{align*}
and therefore the cancellation property of the atom will lead to
\begin{align*}
    III(\mathfrak{C}42)\lesssim \left(2^i\ell_1\right)^{-4N_1}\left(2^j\ell_2\right)^{-4N_2}\left\|{\triangle}_{twist}^{ N_3 } b_{R}\right\|^2_{L^2(\mathbb{R}^{2m})},
\end{align*}
Hence, to wrap up, take $N_1=N_2=m$ and one has
\begin{align*}
       |R|^{\frac{1}{2}}\sum^\infty_{i=6}\sum^\infty_{j=j_0+l}2^{\frac{(i+j)m}{2}}\cdot\left(III(\mathfrak{C}42)\right)^{\frac{1}{2}}
       \lesssim |R|^{\frac{1}{2}}\cdot \ell_1^{-2N_1}\cdot\ell_2^{-2N_2}\cdot\left\|{\triangle}_{twist}^{ N_3 } b_{R}\right\|_{L^2(\mathbb{R}^{2m})}\cdot
       \bigg( {\ell(\widehat{I_2})\over \ell(I_2)}\bigg)^{-\frac{3m}{2}}.
   \end{align*}
   
For the term $II(\mathfrak{C}42)$, we first eliminate the factor $\varphi^{(1)}_{r_1}$. By cancellation property of the atom and the Young's convolution inequality,
 \begin{align} \label{S555}
\left\|a_R*\varphi_{\mathbf{r}}(y_1,\,y_2)\right\|^2_{L^2(\mathbb{R}^{m},\,d{y}_1)}\lesssim r_1^{-4N_1}\left\|\left({ {\triangle}_2^{ N_2 } \triangle^{N_3}_{twist}b_{R}}*_2\varphi^{(2)}_{r_2}*_3\varphi^{(3)}_{r_3}\right)(\mathbf{y})\right\|^2_{L^2(\mathbb{R}^{m},\,d{y}_1)},
   \end{align}
and hence by interchanging of integral and (\ref{S555}),
\begin{align}\label{tttttttttttt}
    II(\mathfrak{C}42)
    &\lesssim\int_{|x_2-c_{I_2}|\simeq2^j\ell_2}\int_{ \mathbb{R}^{m}}\int^{\frac{|x_2-c_{I_2}|}{4}}_{\ell_2}\int^{\frac{|x_2-c_{I_2}|}{4} }_{\ell_2}\int^\infty_{\frac{2^i\ell_1}{4}}\left\|a_R*\varphi_{\mathbf{r}}(y_1,\,y_2)\right\|^2_{L^2(\mathbb{R}^{m},\,d{y}_1)}\,\frac{dr_1}{r_1}\notag\\
  &\qquad\times\left(\chi_{B_{r_2}(0)}*\chi_{B_{r_3}(0)}\right)(x_2-y_2)\,\frac {  
    dr_2dr_3}{r^{m+1}_2r^{m+1}_3}\,d{y}_2\,dx_2\notag\\
      &\lesssim\int_{|x_2-c_{I_2}|\simeq2^j\ell_2}\int_{ \mathbb{R}^{m}}\int^{\frac{|x_2-c_{I_2}|}{4}}_{\ell_2}\int^{\frac{|x_2-c_{I_2}|}{4} }_{\ell_2}\Bigg[\int^\infty_{\frac{2^i\ell_1}{4}}\left\|\left({ {\triangle}_2^{ N_2 } \triangle^{N_3}_{twist}b_{R}}*_2\varphi^{(2)}_{r_2}*_3\varphi^{(3)}_{r_3}\right)(\mathbf{y})\right\|^2_{L^2(\mathbb{R}^{m},\,d{y}_1)}\notag\\
      &\qquad\frac{dr_1}{r^{1+4N_1}_1}\Bigg]\times\left(\chi_{B_{r_2}(0)}*\chi_{B_{r_3}(0)}\right)(x_2-y_2)\,\frac {  
    dr_2dr_3}{r^{m+1}_2r^{m+1}_3}\,d{y}_2\,dx_2\notag\\
    &\simeq(2^i\ell_1)^{-4N_1}\int_{|x_2-c_{I_2}|\simeq2^j\ell_2}\int_{ \mathbb{R}^{2m}}\int^{\frac{|x_2-c_{I_2}|}{4}}_{\ell_2}\int^{\frac{|x_2-c_{I_2}|}{4} }_{\ell_2}\left| \left({ {\triangle}_2^{ N_2 } \triangle^{N_3}_{twist}b_{R}}*_2\varphi^{(2)}_{r_2}*_3\varphi^{(3)}_{r_3}\right)(\mathbf{y}) \right|^2\notag\\
      &\qquad\times\left(\chi_{B_{r_2}(0)}*\chi_{B_{r_3}(0)}\right)(x_2-y_2)\,\frac {  
    dr_2dr_3}{r^{m+1}_2r^{m+1}_3}\,d\mathbf{y}\,dx_2.
\end{align}
To estimate (\ref{tttttttttttt}), we will utilize the support condition from the cone and dig out the pointwise estimate. Since we have that
\begin{align*}
      &\left| \left({ {\triangle}_2^{ N_2 } \triangle^{N_3}_{twist}b_{R}}*_2\varphi^{(2)}_{r_2}*_3\varphi^{(3)}_{r_3}\right)(\mathbf{y}) \right|^2\\&=r^{-4N_2}_2\left|\int_{\mathbb{R}^m} \left( \triangle^{N_3}_{twist}b_{R}*_3\varphi^{(3)}_{r_3}\right)(y_1,y_2-v)\cdot\left({\triangle}_2^{ N_2 } \varphi^{(2)}\right)_{r_2}(v)\,dv\right|^2\\
      &=r^{-4N_2}_2\left|\int_{\mathbb{R}^m} \left(  \triangle^{N_3}_{twist}b_{R}*_3\varphi^{(3)}_{r_3}\right)(y_1,y_2-v)\cdot\left({\triangle}_2^{ N_2 } \varphi^{(2)}\right)_{r_2}(v)\,dv\right|^2;
\end{align*}
and from the cone structure in (\ref{tttttttttttt}), we have that $|x_2-y_2|\leq r_2+r_3\leq \frac{|x_2-c_{I_2}|}{2}$; on the other hand, the support condition of $\varphi^{(3)}_{r_3}$ and $\varphi^{(2)}_{r_2}$ give that $|y_2-v-c_{I_2}|\leq \frac{\ell_2}{2}+ \frac{|x_2-c_{I_2}|}{4}$ and $|v|\leq\frac{|x_2-c_{I_2}|}{4}$, respectively. Remark that
\begin{align*}
    |v|&\geq |x_2-c_{I_2}|-|y_2-v-c_{I_2}|-|y_2-x_2|
    \geq \frac{|x_2-c_{I_2}|}{4}-\frac{\ell_2}{2}
    \geq\frac{6|x_2-c_{I_2}|}{25},
\end{align*}
where the last inequality can be deduced from the support condition of $x_2$, and hence $|v|\simeq|x_2-c_{I_2}|$. Thus, we can deduce that
\begin{align*}
    &\left| \left({ {\triangle}_2^{ N_2 } \triangle^{N_3}_{twist}b_{R}}*_2\varphi^{(2)}_{r_2}*_3\varphi^{(3)}_{r_3}\right)(\mathbf{y}) \right|^2\\
&\lesssim\frac{r^{2M_2-4N_2}_2}{\left(r_2+|x_2-c_{I_2}|\right)^{2m+2M_2}}\cdot \left|\int_{|v|\simeq|x_2-c_{I_2}|} \left( \triangle^{N_3}_{twist}b_{R}*_3\varphi^{(3)}_{r_3}\right)(y_1,y_2-v)\,dv\right|^2,
\end{align*}
for some $M_2>0$,
which implies that $II(\mathfrak{C}42)$ is bounded above by
\begin{align*}
      &(2^i\ell_1)^{-4N_1}\int_{|x_2-c_{I_2}|\simeq2^j\ell_2}\int_{ \mathbb{R}^{2m}}\int^{\frac{|x_2-c_{I_2}|}{4}}_{\ell_2}\int^{\frac{|x_2-c_{I_2}|}{4} }_{\ell_2}\frac{r^{2M_2-4N_2}_2}{\left(r_2+|x_2-c_{I_2}|\right)^{2m+2M_2}}\\
&\times\left|\int_{|v|\simeq|x_2-c_{I_2}|} \left( \triangle^{N_3}_{twist}b_{R}*_3\varphi^{(3)}_{r_3}\right)(y_1,y_2-v)\,dv\right|^2\cdot\left(\chi_{B_{r_2}(0)}*\chi_{B_{r_3}(0)}\right)(x_2-y_2)\,\frac {  
    dr_2dr_3}{r^{m+1}_2r^{m+1}_3}\,d\mathbf{y}\,dx_2.
\end{align*}
Next, we will explore the Poisson type estimate given by $\varphi^{(3)}_{r_3}$. Consider that
\begin{align*}
     &\left( \triangle^{N_3}_{twist}b_{R}*_3\varphi^{(3)}_{r_3}\right)(y_1,y_2-v)=r^{-2N_3}_3\cdot\int_{\mathbb{R}^m}b_{R}(y_1-z,y_2-v-z)\cdot\left(\triangle^{N_3}_{3}\varphi^{(3)}\right)_{r_3}(z)\,dz.
\end{align*}
Since that  $|v|\simeq|x_2-c_{I_2}|$, $|y_2-v-z-c_{I_2}|\leq\frac{\ell_2}{2}$ and the cone structure gives that $|x_2-y_2|\leq r_2+r_3\leq\frac{|x_2-c_{I_2}|}{2}$, we then have
\begin{align*}
    |z|&\geq|x_2-c_{I_2}|-|y_2-v-z-c_{I_2}|-|y_2-x_2|-|v|
    \geq|x_2-c_{I_2}|-\frac{\ell_2}{2}-\frac{|x_2-c_{I_2}|}{2}-\frac{6|x_2-c_{I_2}|}{25}
    \geq\frac{1}{4}|x_2-c_{I_2}|,
\end{align*}
which leads to $|z|\simeq |x_2-c_{I_2}|$. Then the Poisson type bound gives that
\begin{align*}
     &\left|  \triangle^{N_3}_{twist}b_{R}*_3\varphi^{(3)}_{r_3}\right|(y_1,y_2-v)
     \lesssim\frac{r^{M_3-2N_3}_3}{\left(r_3+|x_2-c_{I_2}|\right)^{m+M_3}}\int_{\mathbb{R}^m}\left|b_{R}(y_1-z,y_2-v-z)\right|\,dz,
\end{align*}
 where $M_3$ is a positive integer, which further implies that the term $II(\mathfrak{C}42)$ is dominated by
 \begin{align*}
      &(2^i\ell_1)^{-4N_1}\int_{|x_2-c_{I_2}|\simeq2^j\ell_2}\int_{ \mathbb{R}^{2m}}\int^{\frac{|x_2-c_{I_2}|}{4}}_{\ell_2}\int^{\frac{|x_2-c_{I_2}|}{4} }_{\ell_2}\frac{r^{2M_2-4N_2}_2}{\left(r_2+|x_2-c_{I_2}|\right)^{2m+2M_2}}\cdot\frac{r^{2M_3-4N_3}_3}{\left(r_3+|x_2-c_{I_2}|\right)^{2m+2M_3}}\\
&\times\left[\int_{|v|\simeq|x_2-c_{I_2}|} \int_{|z|\simeq|x_2-c_{I_2}|} \left| b_{R}(y_1-z,y_2-v-z)\right|\,dz\,dv\right]^2\cdot\left(\chi_{B_{r_2}(0)}*\chi_{B_{r_3}(0)}\right)(x_2-y_2)\,\frac {  
    dr_2dr_3}{r^{m+1}_2r^{m+1}_3}\,d\mathbf{y}\,dx_2.
\end{align*}
To continue, by Minkowski's inequality and H\"older's inequality, one has
\begin{align*}
    &\int_{\mathbb{R}^m}\left[ \int_{|v|\simeq|x_2-c_{I_2}|} \int_{|z|\simeq|x_2-c_{I_2}|} \left| b_{R}(y_1-z,y_2-v-z)\right|\,dz\,dv\right]^2\,dy_1\\
    &\leq\left[ \int_{|z|\simeq|x_2-c_{I_2}|} \int_{|v|\simeq|x_2-c_{I_2}|}\left(\int_{\mathbb{R}^m}\left| b_{R}(y_1-z,y_2-v-z)\right|^2\,dy_1\right)^{\frac{1}{2}}\,dv\,dz\right]^2\\
     &\lesssim\left[ \int_{|z|\simeq|x_2-c_{I_2}|} \left(\int_{\mathbb{R}^m}\int_{\mathbb{R}^m}\left|b_{R}(y_1-z,y_2-v-z)\right|^2\,dy_1\,dv\right)^{\frac{1}{2}}\,dz\right]^2\cdot|x_2-c_{I_2}|^{m}\\
       &\simeq |x_2-c_{I_2}|^{3m}\cdot\left\|b_{R}\right\|^2_{L^2(\mathbb{R}^{2m})}.
\end{align*}
Therefore, if $M_2-2N_2<0$ and $M_3-2N_3<0$, the term $II(\mathfrak{C}42)$ can be controlled by
 \begin{align*}
&(2^i\ell_1)^{-4N_1}\int_{|x_2-c_{I_2}|\simeq2^j\ell_2}\int_{ \mathbb{R}^{m}}\int^{\frac{|x_2-c_{I_2}|}{4}}_{\ell_2}\int^{\frac{|x_2-c_{I_2}|}{4} }_{\ell_2}\frac{r^{2M_2-4N_2}_2}{\left(r_2+|x_2-c_{I_2}|\right)^{2m+2M_2}}\cdot\frac{r^{2M_3-4N_3}_3}{\left(r_3+|x_2-c_{I_2}|\right)^{2m+2M_3}}\\
&\quad\times|x_2-c_{I_2}|^{3m}\cdot\left\|b_{R}\right\|^2_{L^2(\mathbb{R}^{2m})}\cdot\left(\chi_{B_{r_2}(0)}*\chi_{B_{r_3}(0)}\right)(x_2-y_2)\,\frac {  
    dr_2dr_3}{r^{m+1}_2r^{m+1}_3}\,dy_2\,dx_2\\
&\leq(2^i\ell_1)^{-4N_1}\int_{|x_2-c_{I_2}|\simeq2^j\ell_2}\int^{\frac{|x_2-c_{I_2}|}{4}}_{\ell_2}\int^{\frac{|x_2-c_{I_2}|}{4} }_{\ell_2}\frac{r^{2M_2-4N_2}_2}{\left(r_2+|x_2-c_{I_2}|\right)^{2m+2M_2}}\cdot\frac{r^{2M_3-4N_3}_3}{\left(r_3+|x_2-c_{I_2}|\right)^{2m+2M_3}}\\
&\quad\times|x_2-c_{I_2}|^{3m}\cdot\left\| b_{R}\right\|^2_{L^2(\mathbb{R}^{2m})}\,\frac {  dr_2dr_3}{r_2r_3}\,dx_2\\
&\simeq(2^i\ell_1)^{-4N_1}\left(2^j\ell_2\right)^{4m}\left\|b_{R}\right\|^2_{L^2(\mathbb{R}^{2m})}\times\Bigg[\int^{\frac{2^j\ell_2}{4}}_{\ell_2} \frac{r^{2M_2-4N_2}_2}{\left(r_2+2^j\ell_2\right)^{2m+2M_2}} \,\frac{dr_2}{r_2}\times\int^{\frac{2^j\ell_2}{4} }_{\ell_2}\frac{r^{2M_3-4N_3}_3}{\left(r_3+2^j\ell_2\right)^{2m+2M_3}}\frac{dr_3}{r_3}\Bigg]\\
&\simeq(2^i)^{-4N_1}\cdot(2^j)^{-2M_2-2M_3}\cdot\ell_1^{-4N_1}\ell^{-4N_2-4N_3}_2\left\|b_{R}\right\|^2_{L^2(\mathbb{R}^{2m})}.
\end{align*}
Finally, if we choose $\frac{m}{2}<2N_1$ and $\frac{m}{2}<M_2+M_3$, then we have
\begin{align*}
       &|R|^{\frac{1}{2}}\sum^\infty_{i=6}\sum^\infty_{j=j_0+l}2^{\frac{(i+j)m}{2}}\cdot\left(II(\mathfrak{C}42)\right)^{\frac{1}{2}}
       \lesssim |R|^{\frac{1}{2}}\cdot\ell_1^{-2N_1}\ell^{-2N_2-2N_3}_2\left\|b_{R}\right\|_{L^2(\mathbb{R}^{2m})}\cdot \bigg( {\ell(\widehat{I_2})\over \ell(I_2)}\bigg)^{(\frac{m}{2}-M_2-M_3)}.
   \end{align*}
\smallskip

\begin{itemize}
     \item {\it Estimate on $(\mathfrak{C}43)$}:
 \end{itemize}
 Now, we estimate the term $(\mathfrak{C}43)$. Decompose this term as the following
   \begin{align*}&\int_{\substack{|x_1-x_2-c_{I_1}|\simeq2^i\ell_1\\|x_2-c_{I_2}|\simeq2^j\ell_2}}\int_{ \mathbb{R}^{2m}}\int^{\frac{|x_1-x_2-c_{I_1}|}{4}}_{\ell_1}\int^\infty_
     {\frac{|x_2-c_{I_2}|}{4}}
     \Bigg(\int^{\ell_1}_{0}+\int^{\frac{|x_1-x_2-c_{I_1}|}{8}}_{\ell_1}+\int^\infty_{\frac{|x_1-x_2-c_{I_1}|}{8}}\Bigg)\\
     &\qquad|a_R* \varphi_{\mathbf{r} } |^2(\mathbf{y})\cdot\chi_{\mathbf{r}}(\mathbf{x}-\mathbf{y})\,\frac { dr_2 
    dr_3dr_1}{r_2r_3r_1}\,d\mathbf{y}\,d\mathbf{x}\\
    &=:{\underset{\text{molecule~decomposition}}{I(\mathfrak{C}43)+II(\mathfrak{C}43)}}+III(\mathfrak{C}43).
 \end{align*}
The term $III(\mathfrak{C}43)$ can be easily obtained by following the similar method used in estimating the term 
$III(\mathfrak{C}42)$. For the estimate on the terms $I(\mathfrak{C}43)$ and $II(\mathfrak{C}43)$, we only demonstrate the estimate on $II(\mathfrak{C}43)$ in the subsection $\ref{Molecule}$; the estimate for the term $I(\mathfrak{C}43)$ would be simpler.

 \begin{itemize}
     \item {\it Estimate on $(\mathfrak{C}44)$}
 \end{itemize}
 
We consider the following decomposition
 \begin{align*}
(\mathfrak{C}44)&=\int_{\substack{|x_1-x_2-c_{I_1}|\simeq2^i\ell_1\\|x_2-c_{I_2}|\simeq2^j\ell_2}}\int_{ \mathbb{R}^{2m}}\int^{\frac{|x_1-x_2-c_{I_1}|}{8}}_{\ell_1}\int^
     {\frac{|x_2-c_{I_2}|}{4}}_{\ell_2} \bigg(\int^{\ell_1}_0 +\int^{\max\{\frac{|x_1-x_2-c_{I_1}|}{8},\frac{|x_2-c_{I_2}|}{4} \}}_{\ell_1}\\
     &\qquad+\int^\infty_{\max\{\frac{|x_1-x_2-c_{I_1}|}{8},\frac{|x_2-c_{I_2}|}{4} \}} \bigg)\,|a_R* \varphi_{\mathbf{r} } |^2(\mathbf{y})\cdot\chi_\mathbf{r}(\mathbf{x}-\mathbf{y})\,\frac { dr_2 
    dr_3dr_1}{r_2r_3r_1}\,d\mathbf{y}\,d\mathbf{x}\\&=:I(\mathfrak{C}44)+II(\mathfrak{C}44)+III(\mathfrak{C}44).
 \end{align*}

 For the term $I(\mathfrak{C}44)$, observe that
 \begin{align*}
     I(\mathfrak{C}44)&=\int_{\substack{|x_1-x_2-c_{I_1}|\simeq2^i\ell_1\\|x_2-c_{I_2}|\simeq2^j\ell_2}}\int_{ \mathbb{R}^{2m}}\int^{\frac{|x_1-x_2-c_{I_1}|}{8}}_{\ell_1}\int^
     {\frac{|x_2-c_{I_2}|}{4}}_{\ell_2} \int^{\ell_1}_0 \left|\int_{\mathbb{R}^{{2m}}}\left(a_R*_3\varphi^{(3)}_{r_3}\right)(\mathbf{w})\left(\varphi^{(1)}_{r_1}\otimes\varphi^{(2)}_{r_2}\right) (\mathbf{y}-\mathbf{w}) \,d\mathbf{w}\right|^2\\
     &\qquad\times\left(\chi_{B_{r_1}(0)}\otimes\chi_{B_{r_2}(0)}\right)*\chi_{B_{r_3}(0)}(\mathbf{x}-\mathbf{y})\,\frac { dr_2 
    dr_3dr_1}{r^{m+1}_2r^{m+1}_3r^{m+1}_1}\,d\mathbf{y}\,d\mathbf{x}.
 \end{align*}
 If we investigate the support condition given by the first variable, then we have
\begin{align*}
\frac{|x_1-x_2-c_{I_1}|}{8}\geq|y_1-w_1|&\geq|\left(x_1-x_2-c_{I_1}\right)+\left(x_2-c_{I_2}\right)|- |x_1-y_1|-|w_1-w_2-c_{I_1}|-|w_2-c_{I_2}|,
\end{align*}
and hence by triangle inequality,
\begin{align}\label{55599}
    \frac{126}{49}|x_1-x_2-c_{I_1}|\geq|x_2-c_{I_2}|\geq\frac{74}{151}|x_1-x_2-c_{I_1}|;
\end{align}
On the other hand, if we investigate the support condition given by the second variable. Then
by the cone structure, we have
$|x_2-y_2|\leq \frac{|x_2-c_{I_2}|}{4}+\ell_1$; moreover,  $|w_2-c_{I_2}|\leq \frac{\ell_2}{2}+\frac{|x_2-c_{I_2}|}{4}$. Therefore, we have
\begin{align*}
\ell_1\geq|y_2-w_2|&\geq|x_2-c_{I_2}|- |x_2-y_2|-|w_2-c_{I_2}|,
\end{align*}
which leads to \begin{align*}
  |x_1-x_2-c_{I_1}|\geq\frac{49}{3}|x_2-c_{I_2}|,
\end{align*}
which is a contradiction to (\ref{55599}). So, the term $I(\mathfrak{C}44)$ vanishes. 

For the term $III(\mathfrak{C}44)$, we first interchange the integral to integrate against the variable $\mathbf{x}$ and apply the Littlewood--Paley inequality to eliminate $\varphi^{(1)}_{r_1}$ and $\varphi^{(3)}_{r_3}$, then
\begin{align*}
    III(\mathfrak{C}44)\lesssim \int^\infty_{\max\{\frac{2^i\ell_1}{8},\frac{2^j\ell_2}{4}\}} \left\|a_R*_2\varphi^{(2)}_{{r}_2 } \right\|^2_{L^2(\mathbb{R}^{2m})}\frac{dr_2}{r_2},
\end{align*}
and hence from the cancellation property of the atom, we have
\begin{align*}
III(\mathfrak{C}44)\lesssim \max\Big\{\frac{2^i\ell_1}{8},\frac{2^j\ell_2}{4}\Big\}^{-4N_2}\left\|\triangle^{N_1}_1\triangle^{N_3}_{twist}b_R\right\|^2_{L^2(\mathbb{R}^{2m})}.
\end{align*}
To wrap up, take $2N_2>m$, we have
\begin{align*}
       &|R|^{\frac{1}{2}}\sum^\infty_{i=6}\sum^\infty_{j=j_0+l}2^{\frac{(i+j)m}{2}}\cdot\left(III(\mathfrak{C}44)\right)^{\frac{1}{2}}
   \lesssim|R|^{\frac{1}{2}}\cdot\left(\ell_1\ell_2\right)^{-N_2}\left\|\triangle^{N_1}_1\triangle^{N_3}_{twist}b_R\right\|_{L^2(\mathbb{R}^{2m})}\cdot\bigg( {\ell(\widehat{I_2})\over \ell(I_2)}\bigg)^{-(N_2-\frac{m}{2})}.
   \end{align*}

Finally, for the term $II(\mathfrak{C}44)$, without loss of generality, we assume that  \[
\frac{|x_1-x_2-c_{I_1}|}{8}\leq\frac{|x_2-c_{I_2}|}{4} ,
\]
and observe that in this case $II(\mathfrak{C}44)$ is equal to
\begin{align*}
&\int_{\substack{|x_1-x_2-c_{I_1}|\simeq2^i\ell_1\\|x_2-c_{I_2}|\simeq2^j\ell_2\\ \frac{|x_1-x_2-c_{I_1}|}{8}\leq\frac{|x_2-c_{I_2}|}{4}}}\int_{\mathbb{R}^{2m}}\int^{\frac{|x_1-x_2-c_{I_1}|}{8}}_{\ell_1}\int^
     {\frac{|x_2-c_{I_2}|}{4}}_{\ell_2} \int^{\frac{|x_2-c_{I_2}|}{4}}_{\ell_1}|a_R* \varphi_{\mathbf{r} } |^2(\mathbf{y}) \\
     &\quad\times\left(\chi_{B_{r_1}(0)}\otimes\chi_{B_{r_2}(0)}\right)*\chi_{B_{r_3}(0)}(\mathbf{x}-\mathbf{y})\,\frac { dr_2 
    dr_3dr_1}{r^{m+1}_2r^{m+1}_3r^{m+1}_1}\,d\mathbf{y}\,d\mathbf{x}.
\end{align*}
We need to dig out the pointwise estimate from the function $\varphi$. Remark that
\begin{align*}
    |a_R* \varphi_{\mathbf{r} } |(\mathbf{y})&=r^{-2N_2}_2r^{-2N_3}_3\left|\triangle^{N_1}_1b_R* \left(\varphi^{(1)}_{r_1}\otimes\left(\triangle^{N_2}_2\varphi^{(2)}\right)_{r_2}\right)*_3(\triangle^{N_3}_3\varphi^{(3)})_{{r_3}} \right|(\mathbf{y})\\&=r^{-2N_2}_2r^{-2N_3}_3\left|\int_{\mathbb{R}^{{2m}}}\left(\triangle^{N_1}_1b_R*_3(\triangle^{N_3}_3\varphi^{(3)})_{{r_3}} \right)(\mathbf{w})\left(\varphi^{(1)}_{r_1}\otimes\left(\triangle^{N_2}_2\varphi^{(2)}\right)_{r_2}\right) (y_1-w_1,y_2-w_2) \,d\mathbf{w}\right|.
\end{align*}
  We investigate the support condition given by the second variable.
By the cone structure, we have
$|x_2-y_2|\leq \frac{|x_2-c_{I_2}|}{2}$; moreover, the centralized condition gives that $|w_2-c_{I_2}|\leq \frac{\ell_2}{2}+\frac{|x_2-c_{I_2}|}{4}$. Therefore, we have
\begin{align*}
\frac{|x_2-c_{I_2}|}{4}\geq|y_2-w_2|&\geq|x_2-c_{I_2}|- |x_2-y_2|-|w_2-c_{I_2}|
   \geq\frac{6|x_2-c_{I_2}|}{25},
\end{align*} 
that is $|y_2-w_2|\simeq |x_2-c_{I_2}|$
on the other hand, the support condition given by the first variable gives us that
$
|y_1-w_1|\geq|x_1-x_2-c_{I_1}+x_2-c_{I_2}|- |x_1-y_1|-|w_1-w_2-c_{I_1}|-|w_2-c_{I_2}|,
$
by triangle inequality, this implies that 
\[\frac{18}{7}|x_1-x_2-c_{I_1}|\geq|x_2-c_{I_2}|\geq\frac{1}{2}|x_1-x_2-c_{I_1}|,\]
that is 
$|x_2-c_{I_2}|\simeq|x_1-x_2-c_{I_1}|.$
Next, since that
\begin{align*}
\left(\triangle^{N_1}_1b_R*_3(\triangle^{N_3}_3\varphi^{(3)})_{{r_3}} \right)(\mathbf{w})=\int_{\mathbb{R}^m} \triangle^{N_1}_1b_R(w_1-z,w_2-z)\cdot(\triangle^{N_3}_3\varphi^{(3)})_{{r_3}}(z)\,dz,
\end{align*}
then by investigating the support condition and the fact that
$\frac{|x_2-c_{I_2}|}{4}\geq |z|\geq|x_2-c_{I_2}|-|z+c_{I_2}-w_2|-|w_2-y_2|-|y_2-x_2|$, one can deduce that
\begin{align*}
   \frac{|x_2-c_{I_2}|}{4}\geq  |z|\geq\frac{6|x_2-c_{I_2}|}{25},
\end{align*}
which leads to $|z|\simeq|x_1-x_2-c_{I_1}|\simeq|x_2-c_{I_2}|$.
Then by the pointwise estimate of Poisson type bound, $|a_R* \varphi_{\mathbf{r} } |(\mathbf{y})$ can be dominated by
\begin{align*}
&\int_{|w_1-y_1|\leq\frac{|x_1-x_2-c_{I_1}|}{8}}\int_{\substack{|w_2-y_2|\simeq|x_2-c_{I_2}|\\|z|\simeq|x_2-c_{I_2}|}}\left|\triangle^{N_1}_1b_R(w_1-z,w_2-z)\right|\,dw_2\,\left|\varphi^{(1)}_{r_1}(y_1-w_1) \right|\,dz\,dw_1\\
&\qquad\times\frac{r^{M_2-2N_2}_2}{(r_2+|x_2-c_{I_2}|)^{m+M_2}}\cdot\frac{r^{M_3-2N_3}_3}{(r_3+|x_2-c_{I_2}|)^{m+M_3}}\\
&\lesssim\int_{|z|\simeq|x_2-c_{I_2}|}\mathcal{M}_1\Bigg[\int_{{|w_2-y_2|\simeq|x_2-c_{I_2}|}}\left|\triangle^{N_1}_1b_R(\cdot,w_2-z)\right|\,dw_2\Bigg](y_1-z)\,dz\\
&\qquad\times\frac{r^{M_2-2N_2}_2}{(r_2+|x_2-c_{I_2}|)^{m+M_2}}\cdot\frac{r^{M_3-2N_3}_3}{(r_3+|x_2-c_{I_2}|)^{m+M_3}},
\end{align*}
in which $\mathcal{M}_1$ denotes the Hardy--Littlewood maximal function in the first coordinate. Then, by applying H\"older's inequality, we have $II(\mathfrak{C}44)$ is dominated by
\begin{align*}
&\int_{\substack{|x_1-x_2-c_{I_1}|\simeq2^i\ell_1\\|x_2-c_{I_2}|\simeq2^j\ell_2\\ \frac{|x_1-x_2-c_{I_1}|}{8}\leq\frac{|x_2-c_{I_2}|}{4}}}\int_{\mathbb{R}^{m}}\int^{\frac{|x_1-x_2-c_{I_1}|}{8}}_{\ell_1}\int^
     {\frac{|x_2-c_{I_2}|}{4}}_{\ell_2} \int^{\frac{|x_2-c_{I_2}|}{4}}_{\ell_2}|x_2-c_{I_2}|^m\cdot\frac{r^{2M_2-4N_2}_2}{(r_2+|x_2-c_{I_2}|)^{2m+2M_2}} \\
    &\quad\times\frac{r^{2M_3-4N_3}_3}{(r_3+|x_2-c_{I_2}|)^{2m+2M_3}}\int_{|z|\simeq|x_2-c_{I_2}|}\left(\mathcal{M}_1\Bigg[\int_{\mathbb{R}^m}\left|\triangle^{N_1}_1b_R(\cdot,w_2)\right|\,dw_2\Bigg](y_1-z)\right)^2\,dz\,\frac {  
    dr_2dr_3dr_1}{r_2 r_3r^{m+1}_1}\,dy_1\,d\mathbf{x}.
\end{align*}
To continue, by interchange of integral and $L^2$-boundedness of the Hardy--Littlewood maximal function, we have the term $II(\mathfrak{C}44)$ is bounded above by
\begin{align*}
&\|\triangle^{N_1}_1b_R\|^2_{L^2(\mathbb{R}^{2m})}\cdot\ell^m_2\cdot\int_{\substack{|x_1-x_2-c_{I_1}|\simeq2^i\ell_1\\|x_2-c_{I_2}|\simeq2^j\ell_2\\\frac{|x_1-x_2-c_{I_1}|}{8}\leq\frac{|x_2-c_{I_2}|}{4}}}\int^{\frac{|x_1-x_2-c_{I_1}|}{8}}_{\ell_1}\int^
     {\frac{|x_2-c_{I_2}|}{4}}_{\ell_2} \int^{\frac{|x_2-c_{I_2}|}{4}}_{\ell_1}\frac{r^{2M_2-4N_2}_2}{(r_2+|x_2-c_{I_2}|)^{2m+2M_2}}\\
     &\qquad\times\frac{r^{2M_3-4N_3}_3}{(r_3+|x_2-c_{I_2}|)^{2m+2M_3}}\cdot |x_2-c_{I_2}|^{2m}\,\frac {  
    dr_2dr_3dr_1}{r_2r_3r^{m+1}_1}\,d\mathbf{x} \\&\leq\|\triangle^{N_1}_1b_R\|^2_{L^2(\mathbb{R}^{2m})}\cdot\ell^m_2\cdot\int_{\substack{|x_1-x_2-c_{I_1}|\simeq2^i\ell_1\\|x_2-c_{I_2}|\simeq2^j\ell_2\\\frac{|x_1-x_2-c_{I_1}|}{8}\leq\frac{|x_2-c_{I_2}|}{4}}}\int^{\frac{|x_1-x_2-c_{I_1}|}{8}}_{\ell_1}\int^
     {\frac{|x_2-c_{I_2}|}{4}}_{\ell_2} \int^{\frac{|x_2-c_{I_2}|}{4}}_{\ell_1}
     \frac{r^{2M_2-4N_2}_2}{(r_2+|x_2-c_{I_2}|)^{2M_2}}\\
     &\qquad\times\frac{r^{2M_3-4N_3}_3}{(r_3+|x_2-c_{I_2}|)^{2m+2M_3}}\,\frac{
    dr_2dr_3dr_1}{r_2r_3r^{m+1}_1}\,d\mathbf{x}\\
&\simeq\ell_2^{-4N_2}\ell^{-4N_3}_1\|\triangle^{N_1}_1b_R\|^2_{L^2(\mathbb{R}^{2m})}\cdot2^{i(m-2M_3)}2^{-j(m+2M_2)},  
\end{align*}
provided that $2M_2-4N_2, 2M_3-4N_3<0$. To wrap up, we take $M_3>m$, then we have
\begin{align*}
       &|R|^{\frac{1}{2}}\sum^\infty_{i=6}\sum^\infty_{j=j_0+l}2^{\frac{(i+j)m}{2}}\cdot\left(II(\mathfrak{C}44)\right)^{\frac{1}{2}}
       \lesssim |R|^{\frac{1}{2}}\cdot\ell_2^{-2N_2}\ell^{-2N_3}_1\|\triangle^{N_1}_1b_R\|_{L^2(\mathbb{R}^{2m})}\cdot\bigg( {\ell(\widehat{I_2})\over \ell(I_2)}\bigg)^{-M_2}.
   \end{align*}
   As for the case $({\mathfrak{D}})$, one can combine the idea of support estimate presented in the estimate in $({\mathfrak{C}})$ and the estimate of $({\mathcal{D}})$ to argue it.
   Thus, from the cancellation property of the atom, we complete the proof for the case of atom being of type ${\rm I\!I}$ and the proof is complete.  \end{proof}

\subsection{Molecule decomposition technique}\label{Molecule}
We demonstrate the molecule decomposition technique on dominating the term $I(\mathfrak{C}22), II(\mathfrak{C}22), I(\mathfrak{C}43)$ and $II(\mathfrak{C}43)$ with the emphasis on the term $II(\mathfrak{C}43)$.

\subsubsection{Rough idea and a few remarks}
The main obstacle in estimating these four terms is that the convolution 
$a_R*\varphi_{\mathbf{r}}$ has neither a support condition nor a suitable decay rate. To conquer this difficulty, one needs to perform a suitable decomposition to force the decay rate to appear.

Recall that our goal is to estimate 
\begin{align}\label{sum}
     &|R|^{\frac{1}{2}}\sum^\infty_{i=6}\sum^\infty_{j=j_0+l}2^{\frac{(i+j)m}{2}}\cdot\left(II(\mathfrak{C}43)\right)^{\frac{1}{2}}.
\end{align}
For all $i\geq6$, let $U(i):=2^i\left({I_1}\times_t \widehat{I_2}\right)-2^{i-1}\left({I_1}\times_t \widehat{I_2}\right)$ be the slant annulus and $A(i):=\left((100{I_1})^c\times_t (100\widehat{I_2})^c\right)\cap U(i)$.

\begin{figure}[htbp]
    \centering
    \includegraphics[width=0.4\textwidth]{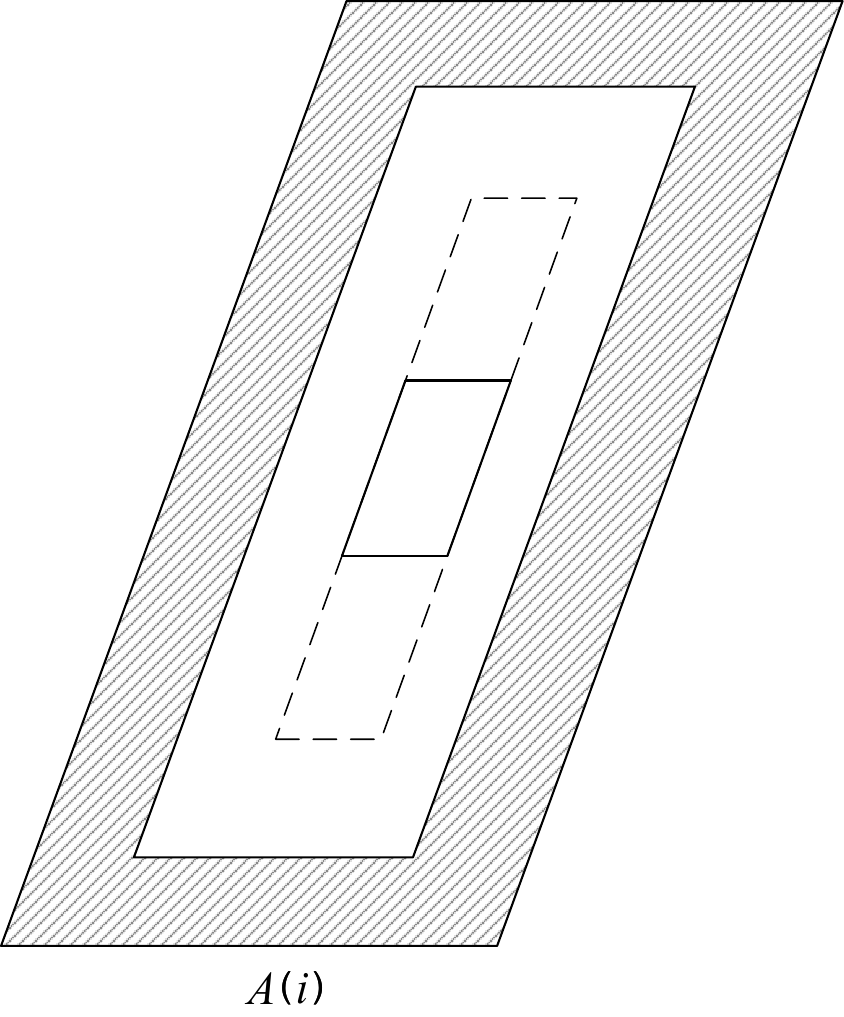}
    \label{2:pdf}
\end{figure}

It turns out that, by the Cauchy-Schwartz inequality,  the sum (\ref{sum}) can be further dominated by
\begin{align}\label{sum2}
 |R|^{\frac{1}{2}}\sum^\infty_{i=6}2^{\frac{(2i+l+j_0)m}{2}}\cdot\left(\widetilde{II(\mathfrak{C}43)}\right)^{\frac{1}{2}},
\end{align}
where $j_0\in\mathbb{N}$ is such that $2^{j_0m}|(I_2)_l|\simeq{|\widehat{I_2}|}$ and 
\begin{align*}
\widetilde{II(\mathfrak{C}43)}=:\int_{A(i)}\int_{ \mathbb{R}^{2m}}\int^{\frac{|x_1-x_2-c_{I_1}|}{4}}_{\ell_1}\int^\infty_
     {\frac{|x_2-c_{I_2}|}{4}}
     \int^{\frac{|x_1-x_2-c_{I_1}|}{8}}_{\ell_1} |a_R* \varphi_{\mathbf{r} } |^2(\mathbf{y})\cdot\chi_{\mathbf{r} }(\mathbf{x}-\mathbf{y})\,\frac { dr_2 
    dr_3dr_1}{r_2r_3r_1}\,d\mathbf{y}\,d\mathbf{x}.
\end{align*}
\smallskip

Now, we further decompose $\widetilde{II(\mathfrak{C}43)}$. For all $i\geq6$, if we let \[
{A(i)}^{\mathcal{H}}:=\left\{\mathbf{x}\in A(i):\,|x_1-x_2-c_{I_1}|\simeq 2^i\ell_1,\,50\widehat{\ell_2}\leq |x_2-c_{I_2}|\leq2^{i-1}\widehat{\ell_2}\right\}
\]
and
\[
{A(i)}^{\mathcal{V}}:=\left\{\mathbf{x}\in A(i):\,50\ell_1\leq|x_1-x_2-c_{I_1}|\leq 2^{i-1}\ell_1,\, |x_2-c_{I_2}|\simeq2^{i}\widehat{\ell_2}\right\}
\]
and
\[
{A(i)}^{\mathcal{D}}:=\left\{\mathbf{x}\in A(i):\,|x_1-x_2-c_{I_1}|\simeq2^i\ell_1,\, |x_2-c_{I_2}|\simeq2^{i}\widehat{\ell_2}\right\},
\]
then we have
\begin{align*}
\widetilde{II(\mathfrak{C}43)}&=\sum^{i-3}_{i_1=0}\sum^{i-4}_{i_2=0}\int_{A(i)^{\mathcal{H}}\cup A(i)^{\mathcal{V}}\cup A(i)^{\mathcal{D}}}\int_{ \mathbb{R}^{2m}}\int^{2^{i_1+1}\ell_1}_{2^{i_1}\ell_1}\int^\infty_
     {\frac{|x_2-c_{I_2}|}{4}}
    \int^{2^{i_2+1}\ell_1}_{2^{i_2}\ell_1}\\
    &\qquad\left|a_R* \varphi_{\mathbf{r} }\right|^2(\mathbf{y})\cdot\chi_{\mathbf{r} }(\mathbf{x}-\mathbf{y})\,\frac { dr_2 
    dr_3dr_1}{r_2r_3r_1}\,d\mathbf{y}\,d\mathbf{x}\\
&=:\widetilde{II(\mathfrak{C}43)}^{\mathcal{H}}+\widetilde{II(\mathfrak{C}43)}^{\mathcal{V}}+\widetilde{II(\mathfrak{C}43)}^{\mathcal{D}}.
\end{align*}
From the decay rate of $|x_2-c_{I_2}|$, it suffices to estimate 
\begin{align}\label{sumHori}
 |R|^{\frac{1}{2}}\sum^\infty_{i=6}2^{\frac{(2i+l+j_0)m}{2}}\cdot\left(\widetilde{II(\mathfrak{C}43)}^{\mathcal{H}}\right)^{\frac{1}{2}}.
\end{align}
To get the decay for the summation, we will decompose the space $\mathbb{R}^{2m}_\mathbf{y}$ to extract the decay. For each $i\geq6$ and for all $k\geq1$, define the slant sets
\[
S_0=S_0(i):=\left\{\mathbf{y}\in\mathbb{R}^{2m}:|y_1-y_2-c_{I_1}|\leq\frac{|x_1-x_2-c_{I_1}|}{2},\,\,|y_2-c_{I_2}|\leq2^i\ell_2\right\}
\]
and
\[
 S_k=S_k(i):=\left\{\mathbf{y}\in\mathbb{R}^{2m}:|y_1-y_2-c_{I_1}|\leq\frac{|x_1-x_2-c_{I_1}|}{2},\,\,|y_2-c_{I_2}|\simeq2^{k+i}\ell_2\right\},
\]
then we have \[
\bigsqcup^\infty_{k=0}S_k=\left\{\mathbf{y}\in\mathbb{R}^{2m}:\,|y_1-y_2-c_{I_1}|\leq\frac{|x_1-x_2-c_{I_1}|}{2}\right\}.
\]
Then by the support condition of $a_R$, the range of $\mathbf{r}$ and the cancellation property of the atom, we have
\begin{align*}
    II(\mathfrak{C}43)^{\mathcal{H}} &=\sum^{i-3}_{i_1=0}\sum^{i-4}_{i_2=0}\sum^\infty_{k=0}\int_{{A(i)^{\mathcal{H}}}}\int_{ S_k}\int^{2^{i_1+1}\ell_1}_{2^{i_1}\ell_1}\int^\infty_
     {\frac{|x_2-c_{I_2}|}{4}}
    \int^{2^{i_2+1}\ell_1}_{2^{i_2}\ell_1}\\
    &\qquad\left|\left(\triangle^{N_3}_{twist}b_R\right)*_3\varphi^{(3)}_{r_3}*\left(\left[\left(\triangle^{N_1}_1\varphi^{(1)}\right)_{r_1}\otimes\left(\triangle^{N_2}_2\varphi^{(2)}\right)_{r_2}\right]\right)\right|^2(\mathbf{y})\\
    &\qquad\chi_{\mathbf{r} }(\mathbf{x}-\mathbf{y})\,\frac {  
   dr_2}{r^{4N_2+1}_2}\,\frac{dr_3}{r_3}\,\frac {  
    dr_1}{r^{4N_1+1}_1}\,d\mathbf{y}\,d\mathbf{x}
\end{align*}
and
\begin{align*}
&\left|\left(\triangle^{N_3}_{twist}b_R\right)*_3\varphi^{(3)}_{r_3}*\left(\left[\left(\triangle^{N_1}_1\varphi^{(1)}\right)_{r_1}\otimes\left(\triangle^{N_2}_2\varphi^{(2)}\right)_{r_2}\right]\right)\right|^2(\mathbf{y})\\
&=\left|\left[\sum^\infty_{n=0}\chi_{S_n}\cdot\left(\left(\triangle^{N_3}_{twist}b_R\right)*_3\varphi^{(3)}_{r_3}\right)\right]*\left(  \triangle^{N_1}_1\varphi^{(1)})_{{r_1}}\otimes (\triangle^{N_2}_2\varphi^{(2)})_{{r_2}} \right) \right|^2(\mathbf{y}).
\end{align*}
Then $ (II(C43)^{\mathcal{H}})^{1/2}$ is bounded by
\begin{align*}
     &\sum^{i-3}_{i_1=0}(2^{i_1}\ell_1)^{-2N_1}\sum^{i-4}_{i_2=0}(2^{i_2}\ell_1)^{-2N_2}\sum^\infty_{k=0}\left(\sum_{n\leq k-2}+\sum^{k+1}_{n=k-1}+\sum^\infty_{k=0}\sum_{n\geq k+2}\right)\Bigg[\int_{ S_k}\int^{2^{i_1+1}\ell_1}_{2^{i_1}\ell_1}\int^\infty_
     {2^{(l+j_0)}\ell_2}
    \\&\int^{2^{i_2+1}\ell_1}_{2^{i_2}\ell_1}\left|\left[\chi_{S_n}\cdot\left(\left(\triangle^{N_3}_{twist}b_R\right)*_3\varphi^{(3)}_{r_3}\right)\right]*\left(  \triangle^{N_1}_1\varphi^{(1)})_{{r_1}}\otimes (\triangle^{N_2}_2\varphi^{(2)})_{{r_2}} \right) \right|^2(\mathbf{y})\,\frac {  dr_2
    dr_3dr_1}{r_2r_3r_1}\,d\mathbf{y}\Bigg]^{1/2}\\
&=:\widetilde{(\mathbb{A})}+{\widetilde{(\mathbb{B})}}+\widetilde{(\mathbb{C})}.
\end{align*}
By symmetry, it suffices to estimate the term ${\widetilde{(\mathbb{B})}}$ and $\widetilde{(\mathbb{C})}$. 
\smallskip

$\blacksquare$ {\it Estimate on the term $\widetilde{(\mathbb{C})}$}:

Note that
for all $n\geq k+2$, $\mathbf{y}\in S_k$,
\begin{align*}
&\left[\chi_{S_n}\cdot\left(\left(\triangle^{N_3}_{twist}b_R\right)*_3\varphi^{(3)}_{r_3}\right)\right]*\left(  \triangle^{N_1}_1\varphi^{(1)})_{{r_1}}\otimes (\triangle^{N_2}_2\varphi^{(2)})_{{r_2}} \right) (\mathbf{y})\\
&=\int_{S_n}\left(\left(\triangle^{N_3}_{twist}b_R\right)*_3\varphi^{(3)}_{r_3}\right)(\mathbf{u})\left(  \triangle^{N_1}_1\varphi^{(1)})_{{r_1}}\otimes (\triangle^{N_2}_2\varphi^{(2)})_{{r_2}} \right)(\mathbf{y}-\mathbf{u})\,d\mathbf{u}
\end{align*}
and thus we have that 
\begin{align*}
    |y_2-u_2|\geq|u_2-c_{I_2}|-|y_2-c_{I_2}|\simeq\left(2^{n+i}-2^{k+i}\right)\ell_2\end{align*}
    and
    \begin{align*}
   |y_1-u_1|
   &\geq |y_2-u_2|-|(y_1-y_2-c_{I_1})-(u_1-u_2-c_{I_1})|
   \geq|y_2-u_2|-|x_1-x_2-c_{I_1}|
   \simeq |y_2-u_2|,
\end{align*}
where the last inequality follows from the fact that the atom is of type ${\rm I\!I}$, which leads to
\begin{align*}
    |\varphi^{(1)}_{r_1}(y_1-u_1)|\lesssim\frac{r^{N_1}_1}{(r_1+|y_2-u_2|)^{m+N_1}}\quad\text{and}\quad  |\varphi^{(2)}_{r_2}(y_2-u_2)|\lesssim\frac{r^{N_2}_2}{(r_2+|y_2-u_2|)^{m+N_2}}.
\end{align*}
Therefore, for all $n\geq k+2$, we get
\begin{align}\label{MO1}
&\left[\chi_{S_n}\cdot\left(\left(\triangle^{N_3}_{twist}b_R\right)*_3\varphi^{(3)}_{r_3}\right)\right]*\left(  \triangle^{N_1}_1\varphi^{(1)})_{{r_1}}\otimes (\triangle^{N_2}_2\varphi^{(2)})_{{r_2}} \right) (\mathbf{y})\notag\\
&\leq \frac{r^{N_1}_1}{(r_1+(2^{n+i}-2^{k+i})\ell_2)^{m+N_1}}\frac{r^{N_2}_2}{(r_2+(2^{n+i}-2^{k+i})\ell_2)^{m+N_2}}
\ \int_{S_n}\left|\left(\triangle^{N_3}_{twist}b_R\right)*_3\varphi^{(3)}_{r_3}\right|(\mathbf{u})\,d\mathbf{u}\notag\\
&=\frac{r^{N_1}_1}{(r_1+(2^{n+i}-2^{k+i})\ell_2)^{m+N_1}}\frac{r^{N_2}_2}{(r_2+(2^{n+i}-2^{k+i})\ell_2)^{m+N_2}}\notag\\
&\qquad\times r^{-2N_3}_3\int_{S_n}\left|b_R*_3\left(\triangle^{N_3}_3\varphi^{(3)}\right)_{r_3}\right|(\mathbf{u})\,d\mathbf{u}.
\end{align}
We can dig out the pointwise estimate from $\left(\triangle^{N_3}_3\varphi^{(3)}\right)_{r_3}$ as well. Note that
\begin{align*}
b_R*_3\left(\triangle^{N_3}_3\varphi^{(3)}\right)_{r_3}(\mathbf{u}):=\int_{\mathbb{R}^{2m}}b_R(u_1-z,u_2-z)\cdot\left(\triangle^{N_3}_3\varphi^{(3)}\right)_{r_3}(z)\,dz,\quad\mathbf{u}\in S_n,
\end{align*}
then from the support condition that $\mathbf{u}-(z,z)\in R$, we must have $|z|\simeq2^{n+i}\ell_2$, and hence
\begin{align}\label{point}
\left|b_R*_3\left(\triangle^{N_3}_3\varphi^{(3)}\right)_{r_3}(\mathbf{u})\right|\lesssim\frac{r^{N_3}_3}{(r_3+2^{n+i}\ell_2)^{m+N_3}}\int_{|z|\simeq2^{n+i}\ell_2}\left|b_R(u_1-z,u_2-z)\right|\,dz.
\end{align}
Thus, combines (\ref{MO1}) and (\ref{point}) with H\"older's inequality, we have
\begin{align*}
&\left[\chi_{S_n}\cdot\left(\left(\triangle^{N_3}_{twist}b_R\right)*_3\varphi^{(3)}_{r_3}\right)\right]*\left(  \triangle^{N_1}_1\varphi^{(1)})_{{r_1}}\otimes (\triangle^{N_2}_2\varphi^{(2)})_{{r_2}} \right) (\mathbf{y})\notag\\
&\lesssim \frac{r^{N_1}_1}{(r_1+(2^{n+i}-2^{k+i})\ell_2)^{m+N_1}}\frac{r^{N_2}_2}{(r_2+(2^{n+i}-2^{k+i})\ell_2)^{m+N_2}}\notag\\
&\qquad\times r^{-2N_3}_3\frac{r^{N_3}_3}{(r_3+2^{n+i}\ell_2)^{m+N_3}}\left(2^{n+i}\ell_2\right)^m|S_n|^{1/2}\left\|b_R\right\|_{L^2(\mathbb{R}^{2m})}.
\end{align*}
Therefore, one can conclude that
\begin{align*}
\widetilde{(\mathbb{C})} &\lesssim\left\|b_R\right\|_{L^2(\mathbb{R}^{2m})}\sum^{i-3}_{i_1=0}(2^{i_1}\ell_1)^{-2N_1}\sum^{i-4}_{i_2=0}(2^{i_2}\ell_1)^{-2N_2}\sum^\infty_{k=0}\sum_{n\geq k+2}\Bigg[\int^\infty_
     {2^{(l+j_0)}\ell_2}r^{-4N_3}_3\frac{r^{2N_3}_3}{(r_3+2^{n+i}\ell_2)^{2m+2N_3}}\,\frac {  
    dr_3}{r_3}\Bigg]^{1/2}\\
    &\qquad\frac{(2^{i_1}\ell_1)^{N_1}}{(2^{i_1}\ell_1+(2^{n+i}-2^{k+i})\ell_2)^{m+N_1}}\frac{(2^{i_2}\ell
    _1)^{N_2}}{(2^{i_2}\ell
    _1+(2^{n+i}-2^{k+i})\ell_2)^{m+N_2}}\cdot \left(2^{n+i}\ell_2\right)^{m}\ell^m_1\ell^m_2 \cdot2^{im}2^{(2i+k+n)m/2}\\
&\lesssim\left[\ell_1^{-2(N_1+N_2)}\ell_2^{-2N_3}\|b_R\|_{L^2(\mathbb{R}^{2m})}\right]\sum^{i-3}_{i_1=0}2^{-2N_1i_1}\sum^{i-4}_{i_2=0}2^{-2N_2i_2}\sum^\infty_{k=0}\sum_{n\geq k+2}(2^{n+i})^{-m-N_3}\left(2^{l+j_0}\right)^{-N_3}\\
    &\qquad\frac{(2^{i_1}\ell_1)^{N_1}}{(2^{i_1}\ell_1+2^{n+i}\ell_2)^{m+N_1}}\frac{(2^{i_2}\ell_1)^{N_2}}{(2^{i_2}\ell_1+2^{n+i}\ell_2)^{m+N_2}}\cdot \left(2^{n+i}\right)^{m}\ell^m_1\ell^m_2 \cdot2^{im}2^{(i+n)m}\\
&\lesssim\left(2^{l+j_0}\right)^{-N_3}2^{-i(N_3-m)}\cdot\ell_1^{-2(N_1+N_2)}\ell_2^{-2N_3}\|b_R\|_{L^2(\mathbb{R}^{2m})},
\end{align*}
which leads to that if we take $N_3=3m$, then
\begin{align*}
    |R|^{\frac{1}{2}}\sum^\infty_{i=6}2^{\frac{(2i+l+j_0)m}{2}}\cdot\widetilde{(\mathbb{C})}
\lesssim|R|^{\frac{1}{2}}\cdot\ell_1^{-2(N_1+N_2)}\ell_2^{-2N_3}\|b_R\|_{L^2(\mathbb{R}^{2m})}\cdot\bigg( {\ell(\widehat{I_2})\over \ell(I_2)}\bigg)^{-\frac{5m}{2}}.
\end{align*}
\smallskip

$\blacksquare$ {\it Estimate on the term $\widetilde{(\mathbb{B})}$}:

Next, we estimate $\widetilde{(\mathbb{B})}$.
By the Littlewood--Paley inequality, we have
\begin{align*}
  \widetilde{(\mathbb{B})}  &\leq\sum^{i-3}_{i_1=0}(2^{i_1}\ell_1)^{-2N_1}\sum^{i-4}_{i_2=0}(2^{i_2}\ell_1)^{-2N_2}\sum^\infty_{k=0}\sum^{k+1}_{n=k-1}\Bigg[\int^{2^{i_1+1}\ell_1}_{2^{i_1}\ell_1}\int^\infty_
     {2^{(l+j_0)}\ell_2}
    \int^{2^{i_2+1}\ell_1}_{2^{i_2}\ell_1}\\
    &\qquad\left\|\left[\chi_{S_n}\cdot\left(\left(\triangle^{N_3}_{twist}b_R\right)*_3\varphi^{(3)}_{r_3}\right)\right]*\left(  \triangle^{N_1}_1\varphi^{(1)})_{{r_1}}\otimes (\triangle^{N_2}_2\varphi^{(2)})_{{r_2}} \right)\right\|^2_{L^2(\mathbb{R}^{2m})}\,\frac {  
    dr_2dr_3dr_1}{r_2r_3r_1}\Bigg]^{1/2}\\
    &\leq\sum^{i-3}_{i_1=0}(2^{i_1}\ell_1)^{-2N_1}\sum^{i-4}_{i_2=0}(2^{i_2}\ell_1)^{-2N_2}\sum^\infty_{k=0}\sum^{k+1}_{n=k-1}\Bigg[\int^\infty_
     {2^{(l+j_0)}\ell_2}
\left\|\left(\triangle^{N_3}_{twist}b_R\right)*_3\varphi^{(3)}_{r_3}\right\|^2_{L^2(S_n)}\,\frac {  
    dr_3}{r_3}\Bigg]^{1/2}.
\end{align*}

Expand $b_R*_3\left(\triangle^{N_3}_3\varphi^{(3)}\right)_{r_3}(\mathbf{u})$ into
\begin{align*}
\int_{\mathbb{R}^{2m}}b_R(u_1-z,u_2-z)\cdot\left(\triangle^{N_3}_3\varphi^{(3)}\right)_{r_3}(z)\,dz,\quad\mathbf{u}\in S_n,
\end{align*}
then from the support condition that $\mathbf{u}-(z,z)\in R$, we must have $|z|\simeq2^{n+i}\ell_2$, and hence
\begin{align}\label{point2}
\left|b_R*_3\left(\triangle^{N_3}_3\varphi^{(3)}\right)_{r_3}(\mathbf{u})\right|\lesssim\frac{r^{N_3}_3}{(r_3+2^{n+i}\ell_2)^{m+N_3}}\int_{|z|\simeq2^{n+i}\ell_2}\left|b_R(u_1-z,u_2-z)\right|\,dz,\quad\mathbf{u}\in S_n.
\end{align}
Therefore, by (\ref{point2}) and Minkowski's integral inequality, one has
\begin{align*}
\left\|\left(\triangle^{N_3}_{twist}b_R\right)*_3\varphi^{(3)}_{r_3}\right\|^2_{L^2(S_n)}
&=r^{-4N_3}_3\int_{S_n}\left|b_R*_3\left(\triangle^{N_3}_3\varphi^{(3)}\right)_{r_3}(\mathbf{u})\right|^2\,d\mathbf{u}\\
&\lesssim r^{-4N_3}_3\frac{r^{2N_3}_3}{(r_3+2^{n+i}\ell_2)^{2m+2N_3}}\cdot\|b_R\|^2_{L^2(\mathbb{R}^{2m})}(2^{n+i}\ell_2)^{2m},
\end{align*}
which implies that
\begin{align*}
   \widetilde{(\mathbb{B})}
    &\lesssim\sum^{i-3}_{i_1=0}(2^{i_1}\ell_1)^{-2N_1}\sum^{i-4}_{i_2=0}(2^{i_2}\ell_1)^{-2N_2}\sum^\infty_{k=0}\Bigg[\int^\infty_
     {2^{(l+j_0)}\ell_2}
\frac{1}{(r_3+2^{k+i}\ell_2)^{2m+2N_3}}\,\frac {  
    dr_3}{r^{2N_3+1}_3}\Bigg]^{1/2}\|b_R\|_{L^2(\mathbb{R}^{2m})}(2^{k+i}\ell_2)^{m}\\
      &\leq\sum^{i-3}_{i_1=0}(2^{i_1}\ell_1)^{-2N_1}\sum^{i-4}_{i_2=0}(2^{i_2}\ell_1)^{-2N_2}\sum^\infty_{k=0}\Bigg[\int^\infty_
     {2^{(l+j_0)}\ell_2}
\,\frac {  
    dr_3}{r^{2N_3+1}_3}\Bigg]^{1/2}\times\|b_R\|_{L^2(\mathbb{R}^{2m})}(2^{k+i}\ell_2)^{-N_3}\\
            &\lesssim2^{-N_3(l+j_0+i)}\cdot\ell_1^{-2(N_1+N_2)}\ell_2^{-2N_3}\|b_R\|_{L^2(\mathbb{R}^{2m})}.
\end{align*}
If we take $N_3=3m$, then
\begin{align*}
    |R|^{\frac{1}{2}}\sum^\infty_{i=6}2^{\frac{(2i+l+j_0)m}{2}}\cdot\widetilde{(\mathbb{B})}
\lesssim|R|^{\frac{1}{2}}\cdot\ell_1^{-2(N_1+N_2)}\ell_2^{-2N_3}\|b_R\|_{L^2(\mathbb{R}^{2m})}\cdot \bigg( {\ell(\widehat{I_2})\over \ell(I_2)}\bigg)^{-\frac{5m}{2}}.
\end{align*}
Therefore, we obtain the estimate (\ref{sumHori}).

  
\smallskip
  
\section{Applications: a new phase-shift converter via the twisted singular integrals}

In this section, we apply the theory of the twisted multiparameter singular integral and the real-variable methods established in previous chapters to analyze the spectral filtering characteristics of the two-dimensional {\it Twisted Hilbert Transform} (THT) $H_{tw}$. Specifically, we examine the operator $H_{tw}$ defined by the Fourier multiplier:
\begin{equation}
    \mathcal{F}(H_{tw} f)(\xi_1, \xi_2) = m(\xi_1, \xi_2) \widehat{f}(\xi_1, \xi_2),
\end{equation}
where the multiplier is defined as:
\begin{equation}
    m(\xi_1, \xi_2) := -i \operatorname{sgn}(\xi_1) \operatorname{sgn}(\xi_2) \operatorname{sgn}(\xi_1+\xi_2).
\end{equation}
This multiplier arises naturally in the study of bilinear singular integrals and captures the interaction between the coordinate axes singularities and the diagonal singularity $\xi_1 + \xi_2 = 0$. We refer to \cite{MRS,GP,GMS} for Fourier multipliers in different settings.

\subsection{Isometric phase modulation}

A fundamental property of the operator $H_{tw}$ is its behavior as a unimodular Fourier multiplier. We observe that for almost every $(\xi_1, \xi_2) \in \mathbb{R}^2$:
\begin{equation}
    |m(\xi_1, \xi_2)| = |-i| \cdot |\operatorname{sgn}(\xi_1)| \cdot |\operatorname{sgn}(\xi_2)| \cdot |\operatorname{sgn}(\xi_1+\xi_2)| = 1.
\end{equation}
Consequently, by Plancherel's theorem, $H_{tw}$ is an isometry on $L^2(\mathbb{R}^2)$. In the context of signal processing, this classifies $H_{tw}$ as an \textbf{all-pass filter}: it preserves the energy of the input signal exactly, modifying only the phase spectrum.

The phase shift introduced by $m$ is discrete. Since $-i = e^{-i\pi/2}$ and the product of the signum functions takes values in $\{ \pm 1 \}$, the total phase shift $\theta(\xi_1, \xi_2)$ satisfies:
\begin{equation}
    e^{i\theta} \in \{ \pm i \} \implies \theta \in \left\{ -\frac{\pi}{2}, \frac{\pi}{2} \right\} \pmod{2\pi}.
\end{equation}
Thus, the operator acts as a direction-dependent $\pi/2$ phase shifter.

\subsection{Anisotropic frequency partition}

The geometric structure of the multiplier is determined by the intersection of the nodal lines $\xi_1 = 0$, $\xi_2 = 0$, and $\xi_1 + \xi_2 = 0$. These lines partition the frequency plane into six open cones (sectors), denoted as $\Omega_{\mathrm{I}}, \dots, \Omega_{\mathrm{VI}}$. The behavior of $m$ within these regions is strictly alternating, exhibiting a ``checkerboard'' pattern of phase leads and lags.

We categorize the regions based on the sign of the product $\sigma = \operatorname{sgn}(\xi_1)\operatorname{sgn}(\xi_2)\operatorname{sgn}(\xi_1+\xi_2)$:

\begin{enumerate}
    \item \textbf{Phase lag regions} ($m = -i$): These correspond to regions where $\sigma = +1$.
    \begin{itemize}
        \item \textit{Region {\rm I}} (first quadrant): $\xi_1 > 0, \xi_2 > 0 \implies \sigma = (+)(+)(+) = +1$.
        \item \textit{Region {\rm III}} (second quadrant, below diagonal): $\xi_1 < 0, \xi_2 > 0,\xi_1+\xi_2 < 0 \implies \sigma = (-)(+)(-) = +1$.
        \item \textit{Region {\rm V}} (fourth quadrant, below diagonal): $\xi_1 > 0, \xi_2 < 0, \xi_1+\xi_2 < 0 \implies \sigma = (+)(-)(-) = +1$.
    \end{itemize}

    \item \textbf{Phase lead regions} ($m = i$): These correspond to regions where $\sigma = -1$.
    \begin{itemize}
        \item \textit{Region {\rm II}} (second quadrant, above diagonal): $\xi_1 < 0,\xi_2 > 0, \xi_1+\xi_2 > 0 \implies \sigma = (-)(+)(+) = -1$.
        \item \textit{Region {\rm IV}} (third quadrant): $\xi_1 < 0, \xi_2 < 0 \implies \sigma = (-)(-)(-) = -1$.
        \item \textit{Region {\rm VI}} (fourth quadrant, above diagonal): $\xi_1 > 0, \xi_2 < 0, \xi_1+\xi_2 > 0 \implies \sigma = (+)(-)(+) = -1$.
    \end{itemize}
\end{enumerate}

\begin{figure}[h]
    \centering
    \begin{tikzpicture}[scale=1.5]
        \definecolor{lagcolor}{RGB}{220, 240, 255} 
        \definecolor{leadcolor}{RGB}{255, 230, 230} 

        \clip (-3,-3) rectangle (3,3);
        
        \fill[lagcolor] (0,0) rectangle (3,3);
        \node at (1.5, 1.5) {\large \textbf{I}};
        \node at (2.2, 2.2) {\tiny $m=-i$};

        \fill[leadcolor] (0,0) -- (0,3) -- (-3,3) -- cycle;
        \node at (-0.7, 2) {\large \textbf{II}};
        \node at (-1.2, 2.5) {\tiny $m=i$};

        \fill[lagcolor] (0,0) -- (-3,3) -- (-3,0) -- cycle;
        \node at (-2, 0.7) {\large \textbf{III}};
        \node at (-2.5, 1.2) {\tiny $m=-i$};

        \fill[leadcolor] (0,0) rectangle (-3,-3);
        \node at (-1.5, -1.5) {\large \textbf{IV}};
        \node at (-2.2, -2.2) {\tiny $m=i$};

        \fill[lagcolor] (0,0) -- (0,-3) -- (3,-3) -- cycle;
        \node at (1.5, -2.5) {\large \textbf{V}}; 
        \node at (2.2, -2.8) {\tiny $m=-i$};

        \fill[leadcolor] (0,0) -- (3,-3) -- (3,0) -- cycle;
        \node at (2, -1) {\large \textbf{VI}};
        \node at (2.5, -1.5) {\tiny $m=i$};

        \draw[->, thick] (-3.2,0) -- (3.2,0) node[right] {$\xi_1$};
        \draw[->, thick] (0,-3.2) -- (0,3.2) node[above] {$\xi_2$};

        \draw[dashed, thick] (-3,3) -- (3,-3) node[right] {$\xi_1+\xi_2=0$};

        \node[blue] at (0.5, 2.5) {$+++$};
        \node[red] at (-0.5, 2.8) {$-++$};
        \node[blue] at (-2.6, 0.5) {$-+-$};
        \node[red] at (-2.5, -0.5) {$---$};
        \node[blue] at (0.5, -2.8) {$+--$};
        \node[red] at (2.4, -0.5) {$+-+$};

    \end{tikzpicture}
    \caption{Frequency domain partition for the multiplier $m(\xi_1, \xi_2)$. Regions shaded in blue correspond to a phase lag ($m=-i$), while regions in red correspond to a phase lead ($m=i$). The dashed line represents the singularity $\xi_1+\xi_2=0$.}
    \label{fig:mht_partition}
\end{figure}
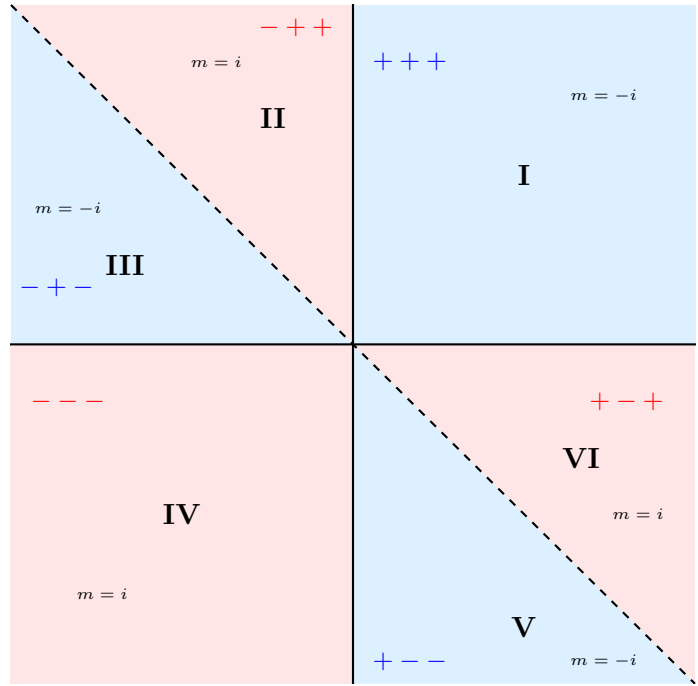

This partition demonstrates that the THT distinguishes frequency components not just by quadrant, but by their proximity to the anti-diagonal singularity. This results in a highly anisotropic filtering effect that is sensitive to the geometric orientation of the input signal's frequency support.

\subsection{Comparative analysis}

We conclude by positioning the THT relative to classical phase-shifting operators.

\begin{enumerate}
    \item \textbf{Comparison with the classical Hilbert transform:}
    \begin{itemize}
        \item \textit{Fundamental Consistency:} Both operators share the same core mechanism: they realize \textbf{isometric phase modulation} via Fourier multipliers. In both cases, the amplitude spectrum remains invariant.
        \item \textit{Increased Complexity:} The distinction lies in the geometric structure. While the 1D transform depends only on the sign of a single variable, the THT is governed by the \textbf{six-region partition} derived in the previous section. This reflects the complex \textbf{coupling between directionality and frequency} inherent in two-dimensional signals (e.g., images), capturing features that simpler 1D tensor products may miss.
    \end{itemize}

    \item \textbf{Comparison with the classical Riesz transform:}
    Standard high-dimensional generalizations of the Hilbert transform are typically achieved via Riesz transforms, characterized by multipliers of the form:
    \begin{equation}
        R_j(\xi) = -i\frac{\xi_j}{|\xi|}.
    \end{equation}
    The Riesz transforms are fundamentally \textbf{isotropic} (rotationally invariant). In contrast, the THT multiplier $m(\xi_1, \xi_2)$ exhibits explicit dependence on the coordinate axes and the diagonal direction. Consequently, the THT is inherently \textbf{direction-sensitive} and \textbf{anisotropic}, isolating specific directional frequency components rather than treating all directions uniformly.
\end{enumerate}
\bigskip
\bigskip

\noindent {\bf Acknowledgement:} J. Li is supported by ARC DP 260100485. C.-W. Liang is supported by MQ Cotutelle PhD scholarhsip.   W. Wang is supported by the National Natural Science Foundation of China (No. 12371082). Q. Wu is supported by the National Natural Science Foundation of China (No. 12171221) and Taishan Scholars Program for Young Experts of Shandong Province (tsqn202507265).

\bigskip
\bigskip
\bigskip

\noindent (Z. Fu)  ICT School, The University of Suwon, Hwaseong-si, 18323, South Korea\\
{\it E-mail}: \texttt{zwfu@suwon.ac.kr}\\

\noindent (J. Li) School of Mathematical and Physical Sciences, Macquarie University, NSW, 2109, Australia\\ 
{\it E-mail}: \texttt{ji.li@mq.edu.au}\\

\noindent (C.-W. Liang) School of Mathematical and Physical Sciences, Macquarie University, NSW, 2109, Australia\\ 
{\it E-mail}: \texttt{chongweiliang1228@gmail.com}\\

\noindent (W. Wang)  School of Mathematical Science, Zhejiang University (Zijingang campus), Zhejiang 310058, China\\
 {\it E-mail}: \texttt{wwang@zju.edu.cn}\\

\noindent  (Q. Wu)  School of Mathematics and Statistics, Linyi University, Linyi 276000, China\\
{\it E-mail}: \texttt{wuqingyan@lyu.edu.cn}\\

\end{document}